\documentclass[12pt, reqno]{amsproc}
\usepackage[utf8]{inputenc}
\usepackage[T2A]{fontenc}
\usepackage[english]{babel}

\usepackage{graphicx}
\usepackage[margin = 2.5cm]{geometry}
\usepackage{enumitem}
\usepackage[svgcolors, dvipsnames]{xcolor}
\usepackage[all]{xy}
\usepackage{amssymb, amsmath, amscd, amsthm}
\usepackage[colorlinks=true]{hyperref}
\hypersetup{urlcolor=blue, citecolor=red}

\title[Diffeomorphism groups of Morse-Bott foliations on lens spaces]{Homotopy types of diffeomorphism groups of polar Morse-Bott foliations on lens spaces, 1}

\author{Oleksandra Khokhliuk}
\address{Liceum ``Educator'', Yelyzavety Chavdar St, 11а, Kyiv, 02140, Ukraine}
\email{khokhliyk@gmail.com}

\author{Sergiy Maksymenko}
\address{Algebra and Topology Department, Institute of Mathematics NAS of Ukraine \\
Teresh\-chen\-kivska str., 3, Kyiv, 01024, Ukraine}
\email{maks@imath.kiev.ua}

\newcommand\mycolor[1]{}

\setlist[enumerate]{itemsep=0.3ex, topsep=0.3ex, label={\rm(\arabic*)}}
\setlist[itemize]{itemsep=0.3ex, topsep=0.3ex, leftmargin=4ex}

\newtheorem{subtheorem}[subsubsection]{Theorem}
\newtheorem{sublemma}[subsubsection]{Lemma}

\newtheorem{subcorollary}[subsubsection]{Corollary}

\newtheorem{subremark}[subsubsection]{Remark}
\newtheorem{subexample}[subsubsection]{Example}

\makeatletter
\@addtoreset{subsection}{section}
\@addtoreset{equation}{section}
\@addtoreset{figure}{section}
\@addtoreset{table}{section}
\makeatother

\makeatletter
\newcommand\testshape{family=\f@family; series=\f@series; shape=\f@shape.}
\def\myemphInternal#1{\if n\f@shape%
\begingroup\itshape #1\endgroup\/%
\else\begingroup\sf\itshape\small #1\endgroup%
\fi}
\def\myemph{\futurelet\testchar\MaybeOptArgmyemph}
\def\MaybeOptArgmyemph{\ifx[\testchar \let\next\OptArgmyemph
                 \else \let\next\NoOptArgmyemph \fi \next}
\def\OptArgmyemph[#1]#2{\index{#1}\myemphInternal{#2}}
\def\NoOptArgmyemph#1{\myemphInternal{#1}}
\makeatother

\newcommand\monoArrow{\lhook\joinrel\rightarrow}

\newcommand\xmonoArrow[1]{\lhook\joinrel\xrightarrow{~#1~}}

\newcommand\epiArrow{\rightarrow\!\!\!\!\!\to}

\newcommand\xepiArrow[1]{\xrightarrow{#1}\!\!\!\!\!\to}

\newcommand\amatr[4]{\left(\!\begin{smallmatrix}#1\ &#2\\[0.5mm] #3\ &#4 \end{smallmatrix}\!\right)}
\newcommand\avect[2]{\left(\begin{smallmatrix}#1 \\ #2 \end{smallmatrix}\right)}
\newcommand\ddd[2]{\tfrac{\partial#1}{\partial#2}}

\newcommand\Bman{B}
\newcommand\Cman{C}

\newcommand\Eman{E}

\newcommand\Jman{J}
\newcommand\Kman{K}
\newcommand\Lman{L}
\newcommand\Mman{M}

\newcommand\Qman{Q}

\newcommand\Sman{S}
\newcommand\Tman{T}
\newcommand\Uman{U}
\newcommand\Vman{V}

\newcommand\cov[1]{\tilde{#1}} %

\newcommand\tQman{\cov{\Qman}}

\newcommand\bC{\mathbb{C}}
\newcommand\bD{\mathbb{D}}
\newcommand\bN{\mathbb{N}}

\newcommand\bR{\mathbb{R}}
\newcommand\bZ{\mathbb{Z}}

\newcommand\id{\mathrm{id}}          
\newcommand\Int{\mathrm{Int}}        
\newcommand\rank{\mathsf{rank}}      
\newcommand\eps{\varepsilon}                   
\newcommand\leftmaspto{\leftarrow\!\shortmid}  
\newcommand\GL{\mathrm{GL}}

\newcommand\SO{\mathrm{SO}}
\newcommand\Ort{\mathrm{O}}

\newcommand\Aut{\mathrm{Aut}}       
\newcommand\Diff{\mathcal{D}}       
\newcommand\Isom{\mathrm{Isom}}     
 
\newcommand\DiffId{\Diff_{\id}}     
\newcommand\Cr[1]{\mathcal{C}^{#1}}
\newcommand\Cinfty{\mathcal{C}^{\infty}}
\newcommand\Crm[3]{\Cr{#1}\!\left(#2,#3\right)}
\newcommand\Cont[2]{\Crm{0}{#1}{#2}}                         
\newcommand\Ci[2]{\mathcal{C}^{\infty}(#1,#2)}               
\newcommand\Wtop{\mathsf{W}}
\newcommand\Stop{\mathsf{S}}
\newcommand\Wr[1]{\Wtop^{#1}}
\newcommand\Sr[1]{\Stop^{#1}}

\newcommand\fixsymbol{\mathrm{fix}}
\newcommand\invsymbol{}
\newcommand\nbsymbol{\mathrm{nb}}
\newcommand\folsymbol{*}

\newcommand\DiffInv[3][\empty]{\Diff_{\invsymbol}(#2,#3\ifx\empty #1\relax\else,#1\fi)}

\newcommand\DiffFix[3][\empty]{\Diff_{\fixsymbol}(#2,#3\ifx\empty #1\relax\else,#1\fi)}

\newcommand\DiffNb[3][\empty]{\Diff_{\nbsymbol}(#2,#3\ifx\empty #1\relax\else,#1\fi)}

\newcommand\DiffHFix[3][\empty]{\Diff^{0}_{\fixsymbol}(#2,#3\ifx\empty #1\relax\else,#1\fi)}

\newcommand\DiffHNb[3][\empty]{\Diff^{0}_{\nbsymbol}(#2,#3\ifx\empty #1\relax\else,#1\fi)}

\newcommand\DiffPlusFix[3][\empty]{\Diff^{+}_{\fixsymbol}(#2,#3\ifx\empty #1\relax\else,#1\fi)}

\newcommand\FDiff[2][\empty]{\Diff^{\folsymbol}(#2\ifx\empty #1\relax\else,#1\fi)}
\newcommand\FDiffFix[2][\empty]{\Diff^{\folsymbol}_{\fixsymbol}(#2\ifx\empty #1\relax\else,#1\fi)}
\newcommand\FDiffA[2][\empty]{\Diff^{=}(#2\ifx\empty #1\relax\else,#1\fi)}

\newcommand\VBAut[2][\empty]{\GL(#2\ifx\empty #1\relax\else,#1\fi)}

\newcommand\DiffLP{\Diff}  
\newcommand\DiffLPInv[3][\empty]{\DiffLP_{inv}(#2,#3\ifx\empty#1\relax\else,#1\fi)}

\newcommand\DiffLPFix[3][\empty]{\DiffLP_{fix}(#2,#3\ifx\empty#1\relax\else,#1\fi)}

\newcommand\DiffLPNb[3][\empty]{\DiffLP_{nb}(#2,#3\ifx\empty#1\relax\else,#1\fi)}

\newcommand\func{f}
\newcommand\gfunc{g}
\newcommand\dif{h}
\newcommand\gdif{g}

\newcommand\px{x}
\newcommand\py{y}
\newcommand\pz{z}
\newcommand\pu{u}
\newcommand\pv{v}

\newcommand\pvi[1]{\pv_{#1}}

\newcommand\Circle{S^1}\newcommand\TLman{\mathbf{\Lman}}             
\newcommand\Lpq[2]{L_{#1,#2}}                 
\newcommand\FolLpq[2]{\mathcal{F}_{#1,#2}}    

\newcommand\FLP{$\Foliation$-leaf preserving}

\newcommand\Ghom{{\mycolor{blue}G}}   

\newcommand\DXid{\mathcal{D}_0}
\newcommand\OBD{\Omega(\DiffId(\Bman),\id_{\Bman})}
\newcommand\OBDid{\Omega_0(\DiffId(\Bman),\id_{\Bman})}

\newcommand\vbp{{\mycolor{blue}p}}
\newcommand\vbq{{\mycolor{blue}q}}
\newcommand\hvbp{\tilde{\vbp}}

\newcommand\por{c}
\newcommand\pnor{c'}
\newcommand\qor{q_{\alpha}}
\newcommand\qnor{q_{\beta}}
\newcommand\qs{q_{\xi}}

\newcommand\DiffFixFQ{\DiffPlusFix[\apath]{\Foliation}{\Qman}}
\newcommand\DiffFixFQS{\DiffFix[\apath]{\Foliation}{\Qman\cup\Circle}}
\newcommand\DiffFixNFQS{\DiffHFix[\apath]{\Foliation}{\Qman\cup\Circle}}

\newcommand\apath{\eta}

\newcommand\bRnR{\bR\times\bR^{n}}
\newcommand\SRn{\Circle\times\bR^{n}}
\newcommand\StRn{\Circle\tilde{\times}\bR^{n}}

\newcommand\DiffInvMB{\DiffInv{\Mman}{\Bman}}
\newcommand\DiffInvLMB{\DiffInv[\vbp]{\Mman}{\Bman}}
\newcommand\DiffInvLUMB{\DiffInv[\vbp,\Uman]{\Mman}{\Bman}}
\newcommand\DiffFixMB{\DiffFix{\Mman}{\Bman}}
\newcommand\DiffNbMB{\DiffNb{\Mman}{\Bman}}
\newcommand\DiffFixLMB{\DiffFix[\vbp]{\Mman}{\Bman}}
\newcommand\DiffFixLUMB{\DiffFix[\vbp,\Uman]{\Mman}{\Bman}}

\newcommand\Foliation{\mathcal{F}}
\newcommand\GFoliation{\mathcal{G}}

\newcommand\aConst{0.2}
\newcommand\bConst{0.8}
\newcommand\oConst{1}

\newcommand\DiffFol{\Diff(\Foliation)}
\newcommand\DiffLFol{\DiffInv[\vbp]{\Foliation}{\Bman}}
\newcommand\DiffFixFolB{\DiffFix{\Foliation}{\Bman}}
\newcommand\DiffFixLFolB{\DiffFix[\vbp]{\Foliation}{\Bman}}

\newcommand\tang[2][\empty]{\mathsf{T}\ifx\empty\relax\else_{#1}\fi#2}  
\newcommand\tfib[1]{\tang[\!\mathsf{fib}]{#1}}                          

\newcommand\regNbh[1]{\Tman_{#1}}
\newcommand\Roo{\regNbh{\oConst}}
\newcommand\Ra{\regNbh{\aConst}}
\newcommand\Rb{\regNbh{\bConst}}

\newcommand\tregNbh[1]{\widetilde{\Tman}_{#1}}
\newcommand\tRoo{\tregNbh{\oConst}}

\newcommand\tRb{\tregNbh{\bConst}}

\newcommand\stimes{\tilde{\times}}
\newcommand\tdif{\tilde{\dif}}
\newcommand\tgdif{\tilde{\gdif}}
\newcommand\teta{\tilde{\eta}}
\newcommand\tfunc{\tilde{\func}}
\newcommand\tgfunc{\lambda}

\newcommand\uphi{\pu}
\newcommand\tuphi{\tilde{\pu}}
\newcommand\tphi{\tilde{\phi}}

\newcommand\DiffPR{\widetilde{\Diff}(\bR)}
\newcommand\Fld{F}
\newcommand\GFld{G}
\newcommand\Flow{\mathbf{F}}

\newcommand\ff[2]{{\mycolor{red}#2(#1)}} 

\newcommand\ggi[1]{\mathcal{G}_{#1}}

\newcommand\vba{a}
\newcommand\vbb{b}
\newcommand\vbc{c}

\newcommand\lift[1]{#1'}
\newcommand\llift[1]{#1''}

\newcommand\ufunc{u}
\newcommand\lufunc{\lift{u}}
\newcommand\llufunc{\llift{u}}

\newcommand\ldif{\lift{\dif}}
\newcommand\lldif{\llift{\dif}}

\newcommand\al{{\mycolor{red}w}}
\newcommand\az{{\mycolor{red}\pz}}
\newcommand\ACircle{\Lambda}
\newcommand\ADisk{Z}

\newcommand\as{{\mycolor{blue}s}}   
\newcommand\BRline{\lift{S}}
\newcommand\BDisk{\lift{\ADisk}}

\newcommand\arad{{\mycolor{brown}r}}
\newcommand\aphi{{\mycolor{brown}\phi}}
\newcommand\CRline{\llift{S}}
\newcommand\CRadius{\llift{R}}
\newcommand\CArg{\llift{\Phi}}

\newcommand\RCircle[1]{S_{#1}}   
\newcommand\RDisk[1]{D_{#1}}     
\newcommand\RTor[1]{\mathbf{T}_{#1}}     
\newcommand\RTD[1]{T_{#1}}     
\newcommand\RBd[1]{\mathbf{C}^{#1}}       

\newcommand\XTorus{\RTor{1}}
\newcommand\XBd{\RBd{\bConst}}
\newcommand\XC{\RTor{0}}

\newcommand\DiffFdT{\DiffFix{\Foliation}{\partial\XTorus}}

\newcommand\DiffFNbdT{\DiffFix{\Foliation}{\XBd}}
\newcommand\DiffFHNbdTC{\DiffHFix{\Foliation}{\XC\cup\XBd, \apath}}
\newcommand\DiffFHLNbdTC{\DiffHFix{\Foliation}{\XC\cup\XBd,\apath,\vbp, \RTor{\aConst}}}
\newcommand\DiffFHNbdTNbC{\Diff^{0,0}_{\fixsymbol}(\Foliation,\RTor{\aConst}\cup\XBd, \apath)}
\newcommand\DiffFNbdTNbC{\DiffFix{\Foliation}{\RTor{\aConst}\cup\XBd}}

\newcommand\DiffGJT{\DiffFix{\GFoliation}{\Jman\times T^2}}

\newcommand\ATor{\mathbf{T}}
\newcommand\BTor{\ATor'}

\newcommand\dATor{\partial\ATor}

\newcommand\AFoliation{\Foliation}
\newcommand\BFoliation{\Foliation'}

\newcommand\RNb[1]{\mathbf{N}^{#1}}

\newcommand\MTor{T^2}

\newcommand\DiffGdTA{\DiffFix{\FolLpq{p}{q}}{\MTor}}
\newcommand\DiffGdN{\DiffFix{\FolLpq{p}{q}}{\RNb{\bConst}}}
\newcommand\rmap{\rho}

\newcommand\affr[2]{\mathsf{h}_{#2,#1}}

\newcommand\affst[2]{\widetilde{\dif}_{#2,#1}}
\newcommand\afft[2]{\dif_{#2,#1}}
\newcommand\matrt{\eta}

\newcommand\stObj[1]{{\mycolor{red}#1'}}
\newcommand\lpqObj[1]{{\mycolor{red}#1''}}

\newcommand\tJincl{\mathsf{e}}
\newcommand\stJincl{\stObj{\tJincl}}
\newcommand\lpqJincl{\lpqObj{\tJincl}}

\newcommand\tRestr{\mathsf{r}}
\newcommand\stRestr{\stObj{\tRestr}}
\newcommand\lpqRestr{\tRestr}

\newcommand\hDtwist{d}
\newcommand\hLambda{\lambda}
\newcommand\hMu{\mu}
\newcommand\hTau{\tau}

\newcommand\AffSubgr{\widetilde{\mathcal{A}}} 
\newcommand\tAffSubgr{\mathcal{A}}            

\newcommand\GAffSubgr{\widetilde{\mathcal{T}}}   
\newcommand\stAffSubgr{\mathcal{T}}                 

\newcommand\stMapClGr{\mathcal{G}}   
\newcommand\lpqAffSubgr{\mathcal{L}}  

\newcommand\RotSub{\mathcal{R}}      

\newcommand\hhom{\zeta}

\newcommand\aGrp{A}
\newcommand\bGrp{B}
\newcommand\cGrp{C}
\newcommand\kGrp{K}

\newcommand\restr[2]{#1\vert_{#2}}
\newcommand\nrm[1]{\vert#1\vert}


\begin{document}
\begin{abstract}
Let $T= S^1\times D^2$ be the solid torus, $\mathcal{F}$ the Morse-Bott foliation on $T$ into $2$-tori parallel to the boundary and one singular circle $S^1\times 0$, which is the central circle of the torus $T$, and $\mathcal{D}(\mathcal{F},\partial T)$ the group of diffeomorphisms of $T$ fixed on $\partial T$ and leaving each leaf of the foliation $\mathcal{F}$ invariant.
We prove that $\mathcal{D}(\mathcal{F},\partial T)$ is contractible.

Gluing two copies of $T$ by some diffeomorphism between their boundaries, we will get a lens space $L_{p,q}$ with a Morse-Bott foliation $\mathcal{F}_{p,q}$ obtained from $\mathcal{F}$ on each copy of $T$.
We also compute the homotopy type of the group $\mathcal{D}(\mathcal{F}_{p,q})$ of diffeomorphisms of $L_{p,q}$ leaving invariant each leaf of $\mathcal{F}_{p,q}$.
\end{abstract}

\maketitle

\section{Introduction}\label{sect:intro}
One of the main features and applications of algebraic topology and, especially, homotopy theory, to real world problems is that very often homotopy invariants of ``good'' spaces are discrete and can be described as certain properties ``stable under small perturbations''.

On the other hand, the concrete computation of homotopy types can be done in not so many cases.
Actually, the homotopy type of a CW-complex $X$ is, in principle, determined by Postnikov tower: one can associate to $X$ a certain ``canonical'' representative $\bar{X}$ being homotopy equivalent to $X$, see~\cite[\S4.3, p.~410]{Hatcher:AT:2002}.
However, that construction is more theoretical, since $\bar{X}$ is usually infinite dimensional, and to construct $\bar{X}$ one need to know all the homotopy groups of $X$, which we currently do not completely know even for spheres.
Real computations of homotopy types usually are done by other methods, like theory of coverings or fibre bundles described in standard textbooks, e.g.~\cite{Hatcher:AT:2002}.

For infinite dimensional spaces, e.g.\ spaces of maps (infinite dimensional manifolds), the situation is more complicated.
It is known that for a smooth compact manifold $X$ its group of $\Cinfty$ diffeomorphisms $\Diff(X)$ endowed with $\Cinfty$ Whitney topology is a Fr\'echet manifold, and therefore has the homotopy type of a CW-complex, see~\cite{Palais:Top:1966} for discussion of homotopy types of infinite dimensional manifolds.
Homotopy types of path components of $\Diff(X)$ are completely computed for all compact manifolds with $\dim X\leq 2$, \cite{Smale:ProcAMS:1959, EarleEells:BAMS:1967, EarleEells:JGD:1969, EarleSchatz:DG:1970, Gramain:ASENS:1973}.
There is also a lot of results about the homotopy types of $\Diff(X)$ for $\dim X = 3$, see e.g.~\cite{Hatcher:AnnM:1983, Gabai:JDG:2001, HongKalliongisMcCulloughRubinstein:LMN:2012} and references therein.
In higher dimensions it is known very little (just because that every finitely presented group is the fundamental group of some manifold, so the classification of compact manifolds contains classification of finitely presented groups), and the results are mostly concern with identifying certain non-trivial elements of homotopy groups, e.g.~\cite{Novikov:IzvAN:1965, Schultz:Top:1971, Hajduk:BP:1978, DwyerSzczarba:IJM:1983, Kupers:GT:2019, BerglundMadsen:AM:2020}, see also good reviews by N.~Smolentsev~\cite{Smolentsev:SMP:2006} and J.~Milnor~\cite{Milnor:NAMS:2011}.

Another important class of infinite dimensional spaces are groups of homeomorphisms and diffeomorphisms preserving leaves of a certain foliation.
Their homotopy types are studied even worse.
The most relevant results are obtained in the papers by and K.~Fukui and S.~Ushiki~\cite{FukuiUshiki:JMKU:1975} and K.~Fukui~\cite{Fukui:JJM:1976}.
They studied regular foliations on $3$-manifolds with finitely many Reeb components and proved that the identity path component of the group of foliated (sending leaves to leaves) diffeomorphisms has the homotopy type of a product of circles, see Remark~\ref{rem:Fukui} below.

Most of other results related with the structure of diffeomorphism groups of foliations concern with extensions of results by M.~Herman~\cite{Herman:CR:1971}, W.~Thurston~\cite{Thurston:BAMS:1974}, J.~Mather~\cite{Mather:Top:1971, Mather:BAMS:1974},  D.~B.~A.~Epstein~\cite{Epstein:CompMath:1970} on proving perfectness groups of compactly supported diffeomorphisms isotopic to the identity.
Recall that a group is perfect, if it coincides with its commutator subgroup $G=[G,G]$.
Then perfectness means triviality of the first homology group $H_1(K(G,1),\bZ)$ of the corresponding Eilenberg-MacLane space, and this is where the homotopy theory appears.
The technique developed in the latter result by Epstein~\cite{Epstein:CompMath:1970} was extended by W.~Ling~\cite{Ling:CM:1984} to a certain class of groups of homeomorphisms, see also~\cite{Anderson:AJM:1958}.
That allowed in turn to extend the above results on perfectness of diffeomorphism groups, to groups of leaf-preserving  homeomorphisms and diffeomorphisms of nonsingular foliations, see e.g.\
\cite{Rybicki:MonM:1995, Rybicki:SJM:1996, Rybicki:DM:1996, Rybicki:DGA:1999, Rybicki:DGA:2001, HallerTeichmann:AGAG:2003, AbeFukui:CEJM:2005, LechRybicki:BCP:2007} and references therein.

However for singular foliations their groups of diffeomorphisms are less studied, e.g.~\cite{Fukui:JMKU:1980, Rybicki:DM:1998, Maksymenko:OsakaJM:2011, LechMichalik:PMD:2013}.
One of the important class of singular foliations constitute the so-called \myemph{Morse-Bott} foliations, and in particular, foliations by level sets of Morse-Bott functions.
They play an important role in Hamiltonian dynamics and Poisson geometry, see e.g.~\cite{Fomenko:UMN:1989, VuNgoc:Top:2003,ScarduaSeade:JDG:2009, MafraScarduraSeade:JS:2014, Wiesendorf:JDG:2014, MartinezAlfaroMezaSarmientoOliveira:JDE:2016, MartinezAlfaroMezaSarmientoOliveira:TMNA:2018, EvangelistaSuarezTorresVera:JS:2019}.

Also in~\cite{Maksymenko:TA:2003} and \cite[Theorem~3.7]{Maksymenko:OsakaJM:2011} the second author gave rather wide sufficient conditions on a vector field $F$ on a manifold $X$ under which the identity path component $\DiffId(F)$ of the group of orbit-preserving diffeomorphisms of $F$ is either contractible or homotopy equivalent to the circle.
Given $\dif\in\DiffId(F)$ one can associate to each $x\in X$ the time $\alpha_{\dif}(x)$ between $x$ and $\dif(x)$ along its orbit of $F$, see Section~\ref{sect:smooth_shifts_along_orbits}.
Such a time $\alpha_{\dif}(x)$ is not uniquely defined, however it is shown that certain assumptions on zeros of $F$ and the fact that $\dif\in\DiffId(F)$ allows to choose $\alpha_{\dif}$ to be $\Cinfty$ and continuous in $\dif$.

That result as well as all the technique were applied further to study the homotopy types of stabilizers and orbits of functions $\func\in\Ci{X}{\bR}$ under the natural ``left-right'' action of the product $\Diff(X)\times\Diff(\bR)$ on the space $\Ci{X}{\bR}$ defined by the following rule: $(\dif,\phi)\cdot \func = \phi\circ\func\circ\dif$.
Notice that this action includes the natural ``right'' action of the subgroup $\Diff(X)\times\id_{\bR}$ on $\Ci{X}{\bR}$.
In~\cite{Maksymenko:BSM:2006} there were given wide conditions on $\func$ under which the inclusion of ``right'' stabilizers and orbits into the corresponding ``left-right'' stabilizers and orbits are homotopy equivalences.
Moreover, for the case $\dim X = 2$, the homotopy types of ``right'' stabilizers and orbits were almost completely computed, see e.g.~\cite{Maksymenko:AGAG:2006, MaksymenkoFeshchenko:MS:2015, Maksymenko:TA:2020, KuznietsovaMaksymenko:PIGC:2020, KuznietsovaSoroka:UMJ:2021}, and references therein.
However, the technique developed in the above papers is applicable only to $1$-dimensional foliations.
Let us also mention that E.~Kudryavtseva~\cite{Kudryavtseva:MatSb:1999, Kudryavtseva:SpecMF:VMU:2012, Kudryavtseva:MathNotes:2012, Kudryavtseva:MatSb:2013, Kudryavtseva:DAN:2016} studied the homotopy types of spaces of Morse functions on surfaces, and partially rediscovering the above results obtained a general structure of ``right'' orbits, see~\cite{Maksymenko:TA:2020} for discussion.

In recent three papers~\cite{KhokhliukMaksymenko:IndM:2020, KhokhliukMaksymenko:PIGC:2020, KhokhliukMaksymenko:2022} the present authors developed several techniques for computations of homotopy types of groups of diffeomorphisms of Morse-Bott and more general classes of ``singular'' foliations in higher dimensions.
In this paper we present explicit computations for some simplest Morse-Bott foliations on the solid torus and lens spaces,see Theorems~\ref{th:DFdT_full_variant} and~\ref{th:Lpq_full_variant}.

\subsection{Main result}
Let $\RCircle{\arad} = \{ \az\in \bC \mid \nrm{\az} = \arad \}_{\arad\in[0;1]}$ be the family of concentric circles in $\bC$ together with the origin $\RCircle{0} = 0$,
\[
    \ATor= \{ (\al,\az) \in \bC^2 \mid \nrm{\az} = 1, \ \nrm{\al} \leq 1 \} \cong \Circle\times D^2
\]
be the \myemph{solid torus} in $\bC^2$, $\MTor := \Circle\times\Circle$ be its boundary, and $\AFoliation = \{ \Circle \times \RCircle{\arad} \}_{\arad\in[0,1]}$ be the partition of $\ATor$ into $2$-tori parallel to the boundary and the central circle $\Circle\times 0$.

Equivalently, $\AFoliation$ is a partition into the level sets of the following Morse-Bott function $\func\colon \ATor\to\bR$, $\func(\al,\az) = \nrm{\az}^2$, for which the central circle $\Circle\times 0$ is a non-degenerate critical submanifold.

Say that a diffeomorphism $\dif:\ATor\to\ATor$ is \myemph{$\AFoliation$-foliated} if for each leaf $\omega\in\AFoliation$ its image $\dif(\omega)$ is a (possibly distinct from $\omega$) leaf of $\AFoliation$ as well.
Also $\dif$ is \myemph{$\AFoliation$-leaf preserving}, if $\dif(\omega)=\omega$ for each leaf $\omega\in\AFoliation$.
Define the following groups:
\begin{itemize}
\item $\Diff(\ATor)$ is the group of diffeomorphisms of $\ATor$;
\item $\Diff(\AFoliation)$ is the group of $\AFoliation$-leaf preserving diffeomorphisms of $\ATor$;
\item $\DiffFix{\ATor}{\MTor}$ is the subgroup of $\Diff(\ATor)$ consisting of diffeomorphisms fixed on $\MTor$;
\item $\DiffFix{\AFoliation}{\MTor} := \Diff(\AFoliation)  \cap \DiffFix{\ATor}{\MTor}$.
\end{itemize}
Endow all those groups with the corresponding $\Cinfty$ Whitney topologies.
Our main result is the following theorem proved in Section~\ref{sect:proof:th:DFdT_contr}:
\begin{subtheorem}\label{th:DFdT_contr}
The group $\DiffFix{\AFoliation}{\MTor}$ is weakly contractible, that is weakly homotopy equivalent to a point (i.e.\ all its homotopy groups vanish).
\end{subtheorem}

Notice further that the homotopy types of the groups $\Diff(\ATor)$ and $\DiffFix{\ATor}{\MTor}$ are well known.
In fact, $\DiffFix{\ATor}{\MTor}$ is contractible, see~\cite{Ivanov:LOMI:1976}, while the homotopy type of $\Diff(\ATor)$ can be described as follows, see Section~\ref{sect:diff_of_T2} for details.
Consider the following subgroup of $\GL(2,\bZ)$:
\[
    \stMapClGr:=\left\{ \amatr{\eps}{0}{m}{\delta} \mid m\in\bZ, \, \eps,\delta\in\{\pm1\} \right\}.
\]
It naturally acts on $\bR^2$ by linear automorphisms and preserves the integral lattice $\bZ^2$.
Hence it yields an action on $2$-torus $\MTor = \bR^2/\bZ^2$ by its automorphisms as a Lie group, and so one can define the semidirect product $\stAffSubgr:=(\bR^2/\bZ^2)\rtimes\stMapClGr$ corresponding to that action.
Moreover, every element of $\stAffSubgr$ naturally extends to a certain $\AFoliation$-leaf preserving diffeomorphism of $\ATor$, so we get an inclusion $\tJincl:\stAffSubgr \subset \Diff(\AFoliation)$.
It is well known that the inclusion $\stAffSubgr \subset \Diff(\ATor)$ is a homotopy equivalence, see Lemma~\ref{lm:DiffFixTdT_contr} below.
As a consequence of Theorem~\ref{th:DFdT_contr} we get the following statement which will be proved in Section~\ref{sect:proof:th:DFdT_full_variant}:
\begin{subtheorem}\label{th:DFdT_full_variant}
Each of the following inclusions
\begin{align*}
&\stAffSubgr  \ \xmonoArrow{\tJincl} \ \Diff(\AFoliation)   \ \subset \ \Diff(\ATor), &
&\{\id_{\ATor}\} \ \subset \ \DiffFix{\AFoliation}{\MTor}  \ \subset \ \DiffFix{\ATor}{\MTor}
\end{align*}
is a weak homotopy equivalence.
In particular, every path component of $\Diff(\AFoliation)$ is weakly homotopy equivalent to $2$-torus $\RotSub$, $\pi_0\Diff(\AFoliation) \cong \stMapClGr$, while $\pi_k\DiffFix{\AFoliation}{\MTor}=0$ for all $k\geq0$.
\end{subtheorem}

Furthermore, gluing two copies of $\ATor$ (each equipped with the same foliation $\AFoliation$) via some diffeomorphism of $\MTor$, we will obtain a lens space $\Lpq{p}{q}$ and a foliation $\FolLpq{p}{q}$ on it into two singular circles and parallel $2$-tori.
We will call such a foliation \myemph{polar}.
Then Theorem~\ref{th:DFdT_contr} will also allow us to compute the homotopy type of the group $\Diff(\FolLpq{p}{q})$ of $\FolLpq{p}{q}$-leaf preserving diffeomorphisms.
Namely, the following statement holds:
\begin{subtheorem}\label{th:Lpq_hom_type}
The identity path component of $\Diff(\FolLpq{p}{q})$ is weakly homotopy equivalent to $2$-torus $T^2$, and
\begin{align*}
 \pi_0 \Diff(\FolLpq{p}{q}) =
 \begin{cases}
    \stMapClGr,           & \text{for } \Lpq{0}{1} = \Circle\times S^2, \\
    \bZ_2 \oplus \bZ_2,   & \text{for } \Lpq{1}{0} = S^3, \Lpq{2}{1} = \bR{P}^3, \\
    \bZ_2                 & \text{for } p>2.
 \end{cases}
\end{align*}
Moreover, if $p>2$ and $q^2\not=\pm1 (\bmod\ p)$, then the inclusion $\Diff(\FolLpq{p}{q})\subset\Diff(\Lpq{p}{q})$ is a weak homotopy equivalence.
\end{subtheorem}
In fact, in Section~\ref{sect:proof:th:DFdT_contr} we will obtain a more explicit description of the homotopy type of $\Diff(\FolLpq{p}{q})$, see Theorem~\ref{th:Lpq_full_variant}.

\begin{subremark}\label{rem:Fukui}\rm
Notice that the above foliations $\AFoliation$ and $\FolLpq{p}{q}$ are in some sense <<dual>> to the ones considered by K.~Fukui and S.~Ushiki~\cite{FukuiUshiki:JMKU:1975} and K.~Fukui~\cite{Fukui:JJM:1976}.
Namely, a solid torus $\ATor$ admits a Reeb foliation $\mathcal{R}$ having  $\partial\ATor$ as a leaf and being transversal to $\AFoliation$ in the interior of $\ATor$.
Further, let $\Lpq{p}{q}$ be a lens space obtained by gluing two copies $\ATor_0$ and $\ATor_1$ of $\ATor$ via some diffeomorphism between their boundaries.
Then that Reeb foliation $\mathcal{R}$ of each $\ATor_i$ gives a foliation $\mathcal{R}_{p,q}$ of the corresponding lens space $\Lpq{p}{q}$ which has a joint leaf $\partial\ATor_0=\partial\ATor_1$ with $\FolLpq{p}{q}$ and is transversal to other leaves of $\FolLpq{p}{q}$.
\end{subremark}

\subsection*{Structure of the paper}
For the proof of contractibility of $\ggi{0}=\DiffFix{\AFoliation}{\MTor}$ in Theorem~\ref{th:DFdT_contr} we will construct four nested subgroups $\ggi{4}\subset\ggi{3}\subset\ggi{2}\subset\ggi{1}\subset\ggi{0}$ and show that the inclusions $\ggi{i+1}\subset\ggi{i}$, $i=0,1,2,3$, are homotopy equivalences, while the smallest group $\ggi{4}$ is weakly contractible.
Construction of a deformation of $\ggi{i}$ into $\ggi{i+1}$, $i=0,1,2,3$, requires a specific technique, and in each case we will establish a more general result
which holds for foliations by level sets of arbitrary smooth and definite homogeneous functions.

Section~\ref{sect:preliminaries} contains several notations and constructions used throughout the paper and also certain general results about fiberwise definite homogeneous functions on vector bundles.
Let $\vbp:\Eman\to\Bman$ be a vector bundle over a smooth manifold $\Bman$ which we identify with its image in $\Eman$ under zero section.
In Section~\ref{sect:linearization_theorem} we prove a ``linearization'' Theorem~\ref{th:linearization_simpler} allowing to isotopy diffeomorphisms of the pair $(\Eman,\Bman)$ to diffeomorphisms coinciding with some vector bundle morphisms near $\Bman$.
That statement is a particular case of a recent preprint~\cite{KhokhliukMaksymenko:2022} by the authors, and we present here another and independent proof.
It will be used for the proof of homotopy equivalence $\ggi{3}\subset\ggi{2}$.

In Section~\ref{sect:triv_vb_1} we also obtain several consequences of linearization theorem for vector bundles of rank $1$, and, in particular, Lemma~\ref{lm:collar} used for the inclusion $\ggi{1}\subset\ggi{0}$, and Lemma~\ref{lm:triv_nbh:loc_triv_restr} used in the proof of Theorem~\ref{th:Lpq_full_variant}.
Also, in Section~\ref{sect:product_fol}, we recall the results from our paper~\cite{KhokhliukMaksymenko:PIGC:2020} allowing to prove weak contractibility of $\ggi{4}$, see  Lemma~\ref{sect:product_fol}.

In Section~\ref{sect:vb_over_1manifolds} we prove Theorem~\ref{th:DiffFixNFQS} for the case $\Bman=\Circle$ which allows to ``unloop diffeomorphisms from $\Diff(\AFoliation)$ along the longitude'' and prove homotopy equivalence $\ggi{2}\subset\ggi{1}$.

The proofs of Theorem~\ref{th:DFdT_contr} and of the remained inclusion $\ggi{4}\subset\ggi{3}$ are given in Section~\ref{sect:proof:th:DFdT_contr}.

In Section~\ref{sect:diff_of_T2} we recall the description of the homotopy types of groups $\Diff(T^2)$, $\Diff(\ATor)$, $\DiffFix{\ATor}{\MTor}$, see Lemmas~\ref{lm:hom_type_DiffT2}, \ref{lm:DiffFixTdT_contr}, and in particular, prove deduce Theorem~\ref{th:DFdT_full_variant} from Theorem~\ref{th:DFdT_contr}, see Section~\ref{sect:proof:th:DFdT_full_variant}.

Finally, in Section~\ref{sect:diff_lens_spaces}, we compute the homotopy type of $\Diff(\FolLpq{p}{q})$, and prove Theorem~\ref{th:Lpq_full_variant} being a detailed variant of Theorem~\ref{th:Lpq_hom_type}.

\section{Preliminaries}\label{sect:preliminaries}

For a smooth compact manifold $\Mman$ we will denote by $\DiffId(\Mman)$ the identity path component of the group $\Diff(\Mman)$ of its diffeomorphisms with respect to $\Cinfty$ Whitney topology.
Also the arrows $\monoArrow$ and $\epiArrow$ will mean \myemph{monomorphism} and \myemph{epimorphism} respectively.

\subsection{Deformations}\label{sect:deformations}
Let $X$ be a topological space and $H\colon X\times[0;1]\to X$ be a homotopy.
A subset $A\subset X$ will be called \myemph{invariant under $H$} if $H(A\times[0;1]) \subset A$.
Then is $H$ a \myemph{deformation of $X$ into $A$} if $H_0=\id_{X}$, $H(A\times[0;1])\subset A$, and $H_1(X)\subset A$.
It is well known and is easy to see that in this case the corresponding inclusion map $i\colon A\subset X$ is a homotopy equivalence and the map $H_1\colon X \to A$ is its homotopy inverse.

Moreover, if $B\subset X$ is another subset being invariant under $H$, then the restriction $\restr{H}{B\times[0;1]}\colon B\times[0;1] \to B$ is a deformation of $B$ into $A\cap B$, whence the inclusion $A\cap B \subset B$ is a homotopy equivalence as well.
In that case $H$ can be regarded as a \myemph{deformation of the pair $(X,B)$ into the pair $(A,A\cap B)$}, so the inclusion $(A,A\cap B)\subset(X,B)$ is a homotopy equivalence of pairs.
We will use this observation several times.

\subsection{Notation for several diffeomorphism groups}\label{sect:diff_groups}
Throughout the paper the term \myemph{smooth} means $\Cinfty$ and all manifolds and diffeomorphisms are assumed to be smooth if the contrary is not said.
Also all spaces of smooth maps, and in particular, diffeomorphism groups, are endowed with $\Cinfty$ Whitney topology.

Let $\Mman$ be a manifold.
Then for each subset $\Bman\subset\Mman$ we will use the following notations%
\footnote{Those notations are used in different parts of the paper, so the reader may skip this subsection and refer to it on necessity.}%
:
\begin{itemize}[wide]
\item
$\DiffInvMB = \{ \dif\in\Diff(\Mman) \mid \dif(\Bman)=\Bman \}$ is the group of diffeomorphisms of $\Mman$ leaving $\Bman$ invariant;

\item
$\DiffFixMB$ is the group of diffeomorphisms of $\Mman$ fixed on $\Bman$;

\item
$\DiffNbMB$ is the group of diffeomorphisms of $\Mman$ fixed on some neighborhood of $\Bman$ (which may vary for distinct diffeomorphisms).
\end{itemize}

Moreover, suppose $\Bman$ is a proper submanifold of $\Mman$, i.e.\ $\partial\Bman=\Bman\cap\partial\Mman$, and $\vbp\colon\Eman\to\Bman$ is a regular neighborhood of $\Bman$, i.e.\ a retraction of an open neighborhood $\Eman \subset\Mman$ of $\Bman$ endowed with a structure of a vector bundle.
Then, in addition,
\begin{itemize}[wide]
\item
$\DiffInvLMB$ is the subgroup of $\DiffInvMB$ consisting of diffeomorphisms $\dif$ which coincide with some vector bundle morphism $\hat{\dif}\colon\Eman\to\Eman$ on \myemph{some} neighborhood of $\Bman$ (depending on $\dif$); note that $\DiffNbMB\subset\DiffFixLMB$ since every diffeomorphism $\dif$ fixed near $\Bman$ coincides thus with the identity automorphism $\id_{\Eman}$ of $\vbp$;

\item
$\DiffFixLMB := \DiffInvLMB \cap \DiffFixMB$;

\item
$\DiffInvLUMB$ is the \myemph{subset}%
\footnote{In general, it is not a group, since $\Uman$ might not be not invariant under diffeomorphisms from $\DiffInvLMB$}
of $\DiffInvLMB$ consisting of diffeomorphisms $\dif$ which coincide with some vector bundle morphism $\hat{\dif}\colon\Eman\to\Eman$ on a \myemph{given} neighborhood $\Uman$ of $\Bman$;

\item
$\DiffFixLUMB := \DiffInvLUMB \cap \DiffFixMB$.
\end{itemize}

Finally, let $\Qman\subset\Mman$ be a subset and $\apath\colon[0;1]\to\Mman$ a path such that $\apath(0), \apath(1)\in\Qman$.
\begin{itemize}[wide]
\item
Then $\DiffHFix{\Mman}{\Qman,\apath}$ will denote the subgroup of $\DiffFix{\Mman}{\Qman}$ consisting of diffeomorphisms for which the paths $\apath$ and $\dif\circ\apath$ (having common ends) are homotopic relatively their ends.

\item
We also put
\begin{gather*}
    \DiffHFix[\vbp]{\Mman}{\Qman\cup\Bman,\apath} := \DiffHFix{\Mman}{\Qman,\apath} \cap \DiffFixLMB, \\
    \DiffHFix[\vbp,\Uman]{\Mman}{\Qman\cup\Bman,\apath} := \DiffHFix{\Mman}{\Qman,\apath} \cap \DiffFixLUMB.
\end{gather*}

\item
If, in addition, $\apath\bigl((0;1)\bigr) \subset \Mman\setminus\Qman$, then $\DiffHNb{\Mman}{\Qman,\apath}$ will denote the subgroup of $\DiffNb{\Mman}{\Qman}$ consisting of diffeomorphisms $\dif$ for which the \myemph{open} paths $\apath, \dif\circ\apath\colon(0;1) \to\Mman\setminus\Qman$ are homotopic (as paths into $\Mman\setminus\Qman$) relatively to $(0;\eps] \cup[1-\eps;1)$ for some $\eps>0$;
\end{itemize}

Evidently, if $\Bman\subset\Qman$, $\eta\colon[0;1]\to\Mman$ is a path such that $\eta(0),\eta(1)\in\Qman$, and $\Uman$ is a neighborhood of $\Bman$, then we have the following commutative diagram of inclusions:
\begin{equation}\label{equ:diff_inclusions}
\gathered
\xymatrix@R=2.8ex@C=3ex{
    \DiffHNb{\Mman}{\Qman,\apath}         \ar@{^(->}[r] \ar@{^(->}[d] &
    \DiffHFix[\vbp]{\Mman}{\Qman,\apath}  \ar@{^(->}[r] \ar@{^(->}[d] &
    \DiffHFix{\Mman}{\Qman,\apath}        \ar@{^(->}[d] \\
    \DiffNbMB     \ar@{^(->}[r] &
    \DiffFixLMB   \ar@{^(->}[r] \ar@{^(->}[d] &
    \DiffFixMB    \ar@{^(->}[d] \\
    &
    \DiffInvLMB   \ar@{^(->}[r] &
    \DiffInvMB
}
\endgathered
\end{equation}

\subsection{Partitions}\label{sect:partition}
Let $\Foliation$ be a partition of a manifold $\Mman$.
The elements of $\Foliation$ will also be called \myemph{leaves}.
Given a subset $\Uman\subset\Mman$, we will denote by $\restr{\Foliation}{\Uman}$ the partition of $\Uman$ into the intersections $\omega\cap\Uman$ of leaves of $\Foliation$ with $\Uman$.
Also $\Uman$ is called \myemph{$\Foliation$-saturated}, if $\Uman$ is a union of leaves of $\Foliation$.
A map $\dif\colon\Uman\to\Mman$ will be called
\begin{itemize}
\item
\myemph{$\Foliation$-foliated} whenever for each $\omega\in\Foliation$ the image $\dif(\omega\cap\Uman)$ is contained in some (possibly distinct from $\omega$) leaf of $\Foliation$;
\item
\myemph{\FLP} if $\dif(\omega\cap\Uman) \subset\omega$ for all $\omega\in\Foliation$.
\end{itemize}

Denote by $\DiffFol$ be the group of \FLP\ diffeomorphisms of $\Mman$.

If now $\Diff^{*}_{\bullet}(\Mman,\Bman,\star)$ is one of the above groups, then we denote by $\Diff^{*}_{\bullet}(\Foliation,\Bman,\star)$ the intersection $\DiffFol\cap\Diff^{*}_{\bullet}(\Mman,\Bman,\star)$, i.e.\ just replace $\Mman$ with $\Foliation$ in the notations.
For instance, $\DiffFixFolB := \DiffFol \cap \DiffFixMB$ is the group of \FLP\ diffeomorphisms fixed on $\Bman$, $\DiffFixLFolB := \DiffFol \cap \DiffFixLMB$, and so on.

Notice that if $\Uman \subset \Mman$ is an $\Foliation$-saturated neighborhood of $\Bman$, then $\Diff(\Foliation,\Bman,\vbp,\Uman)$ is a \myemph{group}.

\subsection{Tangent maps along fibers}\label{sect:tangent_map_along_fibers}
Let $\vbp\colon\Eman\to\Bman$ be a smooth vector bundle of rank $n$ over a compact manifold $\Bman$.
We will identify $\Bman$ with the image of zero section of $\vbp$ in $\Eman$.
For every $t>0$ it will be convenient to consider the following \myemph{homothety} map $\delta_t\colon\Eman\to\Eman$, $\delta_t(\px) = t\px$.
A subset $\Uman\subset\Eman$ will be called \myemph{star-like} if $\delta_t(\Uman)\subset\Uman$ for all $t\in[0;1]$.

Let $\dif\colon\Eman\to\Eman$ be a $\Cr{r}$ map, $r\geq1$, such that $\dif(\Bman)\subset\Bman$.
Then there is a well-defined $\Cr{r-1}$ vector bundle morphism $\tfib{\dif}\colon\Eman\to\Eman$ given by
\begin{equation}\label{equ:tfib}
\tfib{\dif}(\px)
    = \lim\limits_{t\to0}\delta_{t}^{-1} \circ \dif\circ\delta_{t}(\px)
    = \lim\limits_{t\to0}\tfrac{1}{t} \dif(t\px), \ \px\in\Eman.
\end{equation}
We will call it the \myemph{tangent map of $\dif$ along $\Bman$ in the direction of fibers of $\vbp$}, see e.g.~\cite{KhokhliukMaksymenko:2022} for discussion.

For instance, if $\Bman$ is a point, so $\Eman=\bR^{n}$, and $\dif\colon\Eman\to\Eman$ is a map such that $\dif(0)=0$, then $\tfib{\dif}\colon\bR^{n}\to\bR^{n}$ is the linear map given by the Jacobi matrix of $\dif$ at $0$.

It is easy to show that if $\dif$ is a diffeomorphism, then $\tfib{\dif}$ is a vector bundle isomorphism.
Moreover, we actually have a $\Cr{r-1}$-homotopy $H\colon[0;1]\times\Eman\to\Eman$,
\begin{equation}\label{equ:lin_homotopy}
H(t,\px) =
\begin{cases}
    \tfrac{1}{t} \dif(t\px), & t>0, \\
    \tfib{\dif}(\px), & t=0.
\end{cases}
\end{equation}
which is $\Cr{r}$ on $(0;1]\times\Eman$.

\subsection{Homogeneous functions on vector bundles}\label{sect:homogeneous_funcs_on_vb}
A continuous function $\func\colon\Eman\to\bR$ is \myemph{homogeneous} of (possibly fractional) degree $k>0$ if $\func(t \px) = t^{k}\px$ for all $t>0$ and $\px\in\Eman$, which can be rephrased so that $\func\circ\delta_t = \delta_{t^k}\circ\func$.
Evidently, for every homogeneous function $\func\colon\Eman\to\bR$ we have that $\Bman\subset\func^{-1}(0)$, however, in general, $\Bman\not=\func^{-1}(0)$.
A homogeneous function $\func\colon\Eman\to\bR$ will be called \myemph{definite}, whenever $\func(\px)>0$ for all $\px\in\Eman\setminus\Bman$.

For instance, if the vector bundle $\vbp\colon\Eman\to\Bman$ is endowed with an orthogonal structure, i.e.\ its structure group is orthogonal, then the corresponding norm $\|\cdot\|\colon\Eman\to[0;+\infty)$ is definite homogeneous of degree $1$.
Let $\Sman:=\{\px\in\Eman \mid \|\px\|=1\}$.
Then the projection $\restr{\vbp}{\Sman}\colon\Sman\to\Bman$ is called the \myemph{spherical bundle} associated with (the orthogonal structure of) $\vbp$.

The following lemma extends~\cite[Lemma~36]{Maksymenko:TA:2003} and is also a variation of~\cite[Lemma~5.1.3]{KhokhliukMaksymenko:2022}.
\begin{sublemma}[\rm{c.f.~\cite[Lemma~36]{Maksymenko:TA:2003}, \cite[Lemma~5.1.3]{KhokhliukMaksymenko:2022}}]
\label{lm:homo_func}
Let $\func\colon\Eman\to\bR$ be a homogeneous function of some degree $k$, and $\Foliation$ be the partition of $\Eman$ into connected components of level sets of $\func$.
Then the following statements hold.
\begin{enumerate}[leftmargin=*]
\item\label{enum:hf:delta_t_foliated}
For every $t>0$ the homothety map $\delta_t\colon\Eman\to\Eman$ is $\Foliation$-foliated.

\item\label{enum:hf:h}
Let $\dif\colon\Eman\to\Eman$ be a $\Cr{1}$ map such that $\dif(\Bman)\subset\Bman$.
For each $t>0$ define the map $\dif_t\colon\Eman\to\Eman$ by $\dif_t(\px) = \delta_t^{-1}\circ\dif\circ\delta_t(\px) = \tfrac{1}{t}\dif(t\px)$, so $\tfib{\dif}(\px) =\lim\limits_{t\to0}\dif_t(\px)$.

\begin{enumerate}[leftmargin=*, label={\rm(\alph*)}, ref={\rm(2\alph*)}]
\item\label{enum:hf:h:flp}
If $\func\circ\dif=\func$, so $\dif$ preserves partition of $\Eman$ into level sets of $\func$, then $\func\circ\dif_t=\func$ for all $t>0$, and therefore $\func\circ\tfib{\dif}=\func$ as well.

\item\label{enum:hf:h:fol}
If $\dif$ is \FLP, then so is $\dif_t$ for $t>0$.
If, in addition, the leaves of $\Foliation$ are closed, then $\tfib{\dif}$ is also \FLP.
\end{enumerate}

\item\label{enum:hf:definite}
If $\func\colon\Eman\to\bR$ is definite, then $\func$ has the following additional properties.
\begin{enumerate}[leftmargin=*, label={\rm(\alph*)}, ref={\rm(3\alph*)}]
\item\label{enum:defhom:level_0}
$\func^{-1}(0)=\Bman$.

\item\label{enum:defhom:level_t}
For every $t>0$ the projection $\vbp:\func^{-1}(t) \to \Bman$ is isomorphic to the spherical bundle $\vbp:\Sman\to\Bman$.
In other words, there exists a homeomorphism $\zeta:\func^{-1}(t) \to \Sman$ such that $\vbp\circ\zeta=\vbp$.
In particular, the path components of $\func^{-1}(t)$, i.e.\ the leaves of $\Foliation$, are closed.
Moreover, if $n = \rank\Eman = 1$ and $\vbp$ is a trivial vector bundle, then $\func^{-1}(t)$ consists of two path components.
In all other cases, $\func^{-1}(t)$ is path connected.

\item\label{enum:defhom:C1}
If $\func$ is $\Cr{1}$, then $\Bman$ coincides with the set of critical points of $\func$.
\end{enumerate}
\end{enumerate}
\end{sublemma}
\begin{proof}
\ref{enum:hf:delta_t_foliated}
Suppose points $\px,\py\in\Eman$ belong to the same leaf $\omega$ of $\Foliation$ being by definition a path component of a level set $\func^{-1}(c)$ for some $c\in\bR$.
We need to show that $\delta_t(\px)=t\px$ and $\delta_t(\py)=t\py$ belong to the same path component of some level set of $\func$.
Let $\gamma\colon[0;1]\to\omega$ be a path such that $\gamma(0)=\px$ and $\gamma(1)=\py$.
Then $\func\circ\gamma(s)=c$ for all $c\in[0;1]$.
Hence $\func\circ\delta_t\circ\gamma(s) = \func(t\gamma(s))=t^k\func(\gamma(s))=t^kc$.
In other words, $\delta_t\circ\gamma\colon[0;1]\to\Eman$ is a path in the level set $\func^{-1}(t^kc)$ between $t\px$ and $t\py$.
Therefore, those points belong to the same path component of the level set $\func^{-1}(t^kc)$, i.e.\ to the same leaf of $\Foliation$.

\ref{enum:hf:h:flp}
If $\func\circ\dif=\func$, then
\[
    \func\circ\dif_t(\px) =
    \func\bigl(\tfrac{1}{t} \dif(t\px) \bigr) =
    \tfrac{1}{t^k} \func\bigl(\dif(t\px)\bigr) =
    \tfrac{1}{t^k} \func(t\px) =
    \tfrac{t^k}{t^k} \func(\px) =
    \func(\px),
\]
and therefore $\func\circ\tfib{\dif}=\func$ as well.

\ref{enum:hf:h:fol}
Suppose $\dif$ is \FLP, and $\omega$ be any leaf of $\Foliation$.
Then by~\ref{enum:hf:delta_t_foliated}, $\delta_t(\omega) \in \Foliation$ for each $t>0$.
Hence $\dif(\delta_t(\omega))=\delta_t(\omega)$, and therefore $\dif_t(\omega) = \delta^{-1}_t(\dif(\delta_t(\omega)))=\delta^{-1}_t(\delta_t(\omega))=\omega$.

Finally, suppose $\omega$ is a closed leaf of $\Foliation$.
Let $H\colon[0;1]\times\Eman\to\Eman$ be the homotopy defined by~\eqref{equ:lin_homotopy}.
Then $H\bigl((0;1]\times\omega\bigr) \subset \omega$, whence, by continuity of $H$,
\[
    H\bigl([0;1]\times\omega\bigr) =
    H\bigl(\overline{(0;1]\times\omega}\bigr) \subset
    \overline{H\bigl((0;1]\times\omega\bigr)} \subset
    \overline{\omega} = \omega.
\]
Hence $\tfib{\dif}(\omega) = H_0(\omega)\subset\omega$, and thus $\tfib{\dif}$ is \FLP.

\ref{enum:hf:definite}
Suppose now that $\func$ is definite.

\ref{enum:defhom:level_0}
By assumption $\func>0$ on $\Eman\setminus\Bman$.
On the other hand, if $\px\in\Bman$, then $t\px=\px$ for all $t>0$.
Hence $t^{k}\func(\px)=\func(t\px)=\func(\px)$, which is possible only when $\func(\px)=0$.

\ref{enum:defhom:level_t}
Notice that for every $a\in(0;+\infty)$ there exists a retraction
\[
    r_{\func,a}\colon\Eman\setminus\Bman \to \func^{-1}(a),
    \qquad
    r_{\func,a}(\px) = (a/\func(\px))^{1/k}\px,
\]
satisfying $\vbp\circ r_{\func,a} = \vbp$ and a homeomorphism
\[
    \phi\colon\Eman\setminus\Bman \to \func^{-1}(a)\times(0;+\infty),
    \qquad
    \phi(\px)=\bigl(r_{\func,a}(\px), \func(\px)\bigr).
\]
Hence if $\gfunc\colon\Eman\setminus\Bman$ is another definite homogeneous function, and $b>0$, then the composition $\zeta:\restr{(r_{\gfunc,b}\circ r_{\func,a})}{\func^{-1}(a)}\colon \func^{-1}(a)\to\gfunc^{-1}(b)$ is a homeomorphism also satisfying $\vbp\circ\zeta=\vbp$.
In particular, if $\gfunc$ is the norm of some orthogonal structure on $\vbp$, and $b=1$, then $\gfunc^{-1}(1) = \Sman$ is the total space of the spherical bundle.

The latter statement about path components of $\Sman$ is a well-known property of sphere bundles.
We will briefly recall its proof.
Note that $\vbp:\Sman\to\Bman$ is a locally trivial $S^{n-1}$-bundle, and we get the corresponding exact sequence of homotopy groups and sets
\[
    \pi_1 S^{n-1} \xrightarrow{~i_1~}
    \pi_1 \Sman   \xrightarrow{~\vbp_1~}
    \pi_1 \Bman   \xrightarrow{~\partial_1~}
    \pi_0 S^{n-1} \xrightarrow{~i_0~}
    \pi_0 \Sman   \xrightarrow{~\vbp_0~}
    \pi_0\Bman,
\]
where the latter term $\pi_0\Bman$ is zero, since $\Bman$ is assumed to be path connected.
Hence if $n>1$, $\pi_0 S^{n-1}=0$, whence $\pi_0 \Sman=0$ as well, i.e.\ $\Sman$ is path connected.

Suppose $n=1$.
Then $\restr{\vbp}{\Sman}:\Sman\to\Bman$ is an $S^0$-bundle, i.e.\ a double cover of $\Bman$, and the above exact sequence reduces to the following one:
\[
    0             \to
    \pi_1 \Sman   \xrightarrow{~\vbp_1~}
    \pi_1 \Bman   \xrightarrow{~\partial_1~}
    \bZ_2         \xrightarrow{~i_0~}
    \pi_0 \Sman   \xrightarrow{~\vbp_0~}
    0,
\]
Evidently, if $\vbp$ is trivial, then $\restr{\vbp}{\Sman}$ is trivial as well.
Conversely, if $\restr{\vbp}{\Sman}$ is trivial, then $\Sman$ is a disjoint union of two copies $\Bman'$ and $\Bman''$ of $\Bman$.
Hence the inverse map $(\restr{\vbp}{\Bman'})^{-1}:\Bman\to\Bman'$ is a section of $\vbp$, whence $\vbp$ is trivial as a one-dimensional vector bundle admitting a section.

Now, if $\restr{\vbp}{\Sman}$ (and therefore $\vbp$) is non-trivial, then $\vbp_1$ is not surjective, whence the image of $\partial_1$ consists of more than one element, and therefore $\partial_1$ is surjective.
Hence, $i_0$ is a constant map, and thus $\vbp_0$ is a bijection, i.e.\ $\pi_0 \Sman=0$.

\ref{enum:defhom:C1}
Since $\func$ takes a minimum on all of $\Bman$, we see that $\Bman$ consists of critical points of $\func$.
On the other hand, let $\px\in\Eman\setminus\Bman$ and $\gamma\colon(-1;1)\to\Eman$ be given by $\gamma(t)=t\px$.
Then $\gamma(1)=\px$ and
\[
\restr{\tfrac{d}{dt}\func(\gamma(t))}{t=1} =
\restr{\tfrac{d}{dt}\func(t\px)}{t=1} =
\restr{\tfrac{d}{dt}(t^k\func(\px))}{t=1} =
\restr{\func(\px) \cdot (k t^{k-1})}{t=1} = k\func(\px) > 0.
\]
Hence $\px$ is a regular point of $\func$.
\end{proof}

\subsection{Shifts along orbits of flows}\label{sect:smooth_shifts_along_orbits}
Let $\Mman$ be a smooth manifold and $\Flow\colon\Mman\times\bR\to\Mman$ be a smooth flow generated by some vector field $\Fld$.
Then for an open subset $\Uman\subset\Mman$ and a smooth function $\alpha\colon\Uman\to\bR$ one can define the following smooth map
\[
    \Flow_{\alpha}\colon\Uman\to\Mman,
    \qquad
    \Flow_{\alpha}(\px) = \Flow(\px,\alpha(\px)), \ \px\in\Uman,
\]
We will call $\Flow_{\alpha}$ the \myemph{shift} along orbits of $\Flow$ via the function $\alpha$, and $\alpha$ will be called in turn a \myemph{shift function} for $\Flow_{\alpha}$.

Let $\Foliation$ be the partition of $\Mman$ into the orbits of $\Flow$.
Then the map $\Flow_{\alpha}$ is $\Foliation$-leaf preserving.
Moreover, the following family of map $\Flow_{t\alpha}(\px) = \Flow(\px,\alpha(\px))$, $t\in[0;1]$, is a homotopy of $\Flow_{\alpha}$ to the identity inclusion $\Uman\subset\Mman$ in the set of $\restr{\Foliation}{\Uman}$-leaf preserving\ maps $\Uman\to\Mman$.

The following lemma will be used several times.
Let $\ff{\alpha}{\Fld}\colon\Uman\to\Mman$ be derivative of $\alpha$ along $\Fld$.
\begin{sublemma}[{\cite[Eq.~(13)]{Maksymenko:TA:2003}}]\label{lm:shift:local_diff}
Let $\px\in\Uman$.
Then the map $\Flow_{\alpha}\colon\Uman\to\Mman$ is a local diffeomorphism at $\px$ iff $\ff{\alpha}{\Fld}(\px) \not=-1$.
In this case, for every non-constant orbit $\omega$ of $\Flow$, the restriction of $\Flow_{\alpha}\colon \omega\cap\Uman \to \omega$ is locally injective.
Moreover, $\Flow_{\alpha}$ preserves directions of orbits of $\Flow$ iff $\ff{\alpha}{\Fld}(\px) >-1$.
\qed
\end{sublemma}

\section{Linearization theorem}\label{sect:linearization_theorem}
Let $\Mman$ be a manifold, $\Bman$ its proper submanifold, and $\vbp\colon\Eman\to\Bman$ a regular neighborhood of $\Bman$, i.e.\ a retraction of an open neighborhood $\Eman$ of $\Bman$ endowed with a structure of a vector bundle.
In~\cite{KhokhliukMaksymenko:2022} the authors proved that for compact $\Bman$ the inclusion of pairs
\[
    \bigl(\DiffInvLMB,\DiffFixLMB\bigr)\subset\bigl(\DiffInvMB,\DiffFixMB\bigr)
\]
is a homotopy equivalence with respect to weak and strong Whitney topologies $\Wr{\infty}$ and $\Sr{\infty}$.
More precisely, it is shown that for any neighborhood $\Vman$ of $\Bman$ in $\Eman$ there exists a homotopy $H:\DiffInvMB\times[0;1]\to\DiffInvMB$ and a continuous function $\phi:\DiffInvMB \times [0;1] \times \Mman \to[0;1]$ such that
\begin{align*}
&H_0=\id_{\DiffInvMB}, &
&H_1(\DiffInvMB)\subset \DiffInvLMB, \\
&H(\DiffInvLMB\times[0;1])\subset \DiffInvLMB, &
&H(\DiffFixMB\times[0;1])\subset \DiffFixMB,
\end{align*}
\begin{itemize}
\item for each $\dif\in\DiffInvMB$ the map $\Mman\times[0;1]\to\Mman$, $(t,\px)\mapsto\phi(\dif,t,\px)$, is $\Cinfty$;
\item $\phi(\dif,t,\px)>0$ for $t>0$ and
\begin{equation}\label{equ:H_homotopy}
    H(\dif,t)(\px) = \tfrac{\dif(\phi(\dif,t,\px)\px)}{\phi(\dif,t,\px)}.
\end{equation}
\end{itemize}
In particular, $H(\dif,t)(\px) = \dif(\px)$ for $\px\in\Mman\setminus\Vman$ and $t\in[0;1]$.
Moreover, if $\Mman$ is endowed with a partition $\Foliation$ having certain ``homogeneity'' properties and $\dif$ preserves elements of $\Foliation$, then so does each $H(\dif,t)$.

The principal technical problem in~\cite{KhokhliukMaksymenko:2022} was to choose the function $\phi$ so that the map given by~\eqref{equ:H_homotopy} will be a diffeomorphism for all $(\dif,t)\in\DiffInvMB\times(0;1]$.

In this section we consider a particular case of the above result for a special class of foliations, see Theorem~\ref{th:linearization_simpler} below.
Its proof is essentially easier, and we will use it for the proof of Theorem~\ref{th:DFdT_full_variant}.

\subsection{Main result}\label{sect:linearization_settings}
As above let $\Mman$ be a manifold, $\Bman$ its proper submanifold, $\vbp\colon\Eman\to\Bman$ a regular neighborhood of $\Bman$, $\func\colon \Eman\to[0;+\infty)$ a $\Cinfty$ \myemph{definite} homogeneous function, and $\regNbh{t}:=\func^{-1}\bigl([0;t]\bigr)$ for each $t\geq0$.
Let also $\Foliation$ be any partition of $\Mman$ for which \myemph{$\Roo$ is saturated and $\restr{\Foliation}{\Roo}$ consists of path components of level sets $\func^{-1}(t)$, $t\in[0;\oConst]$}, as in Lemma~\ref{lm:homo_func}.

\begin{subremark}\rm
Thus we do not require any special properties of leaves of $\Foliation$ in $\Mman\setminus\Roo$.
For instance, $\Mman\setminus\Roo$ may be a whole element of $\Foliation$ or be consisted of distinct points of $\Mman\setminus\Roo$.
\end{subremark}

Let also
\begin{itemize}
\item
$\mu\colon[0;+\infty)\to[0;1]$ be a $\Cinfty$ function such that $\mu([0;\aConst])=0$ and $\mu([\bConst;+\infty))=1$;
\item
$\phi\colon[0;1]\times\Eman\to[0;1]$ be the $\Cinfty$ function given by $\phi(t,\px) = t + (1-t)\mu(\func(\px))$;
\item
$\Ghom\colon\DiffInvMB\times(0;1]\to\Cinfty(\Mman,\Mman)$ be the map defined by
\begin{equation}\label{equ:G_homotopy}
    \Ghom(\dif,t)(\px) =
    \begin{cases}
        \frac{\dif(\phi(t,\px)\px)}{\phi(t,\px)}, & (\dif,t,\px)\in\DiffInvMB\times[0;1]\times\Roo, \\
        \dif(\px), & (\dif,t,\px)\in\DiffInvMB\times[0;1]\times(\Mman\setminus\Rb).
    \end{cases}
\end{equation}
As $\mu=1$ on $[\bConst;+\infty)$, we have that $\phi(t,\px)=1$ on $\func^{-1}\bigl([\bConst;+\infty)\bigr)$, and thus $\Ghom(\dif,t)=\dif$ on $\Roo \cap (\Mman\setminus\Rb)$.
Hence $\Ghom(\dif,t)\colon\Mman\to\Mman$ is indeed a well-defined $\Cinfty$ map.
\end{itemize}
The following theorem is a particular case of results of~\cite{KhokhliukMaksymenko:2022}, however, due to the assumption that $\func$ is definite and a special choice of $\phi$, the proof now is essentially simpler.

\begin{subtheorem}[{\rm c.f.~\cite{KhokhliukMaksymenko:2022}}]
\label{th:linearization_simpler}
The map $\Ghom$ has the following properties.
\begin{enumerate}[leftmargin=*]
\item\label{enum:linsmp:G_general}
For every $t\in(0;1]$ and $\dif\in\DiffInvMB$ we have that:
\begin{enumerate}[label={\rm(\alph*)}, ref={\rm(1\alph*)}]
\item\label{enum:linsmp:G1}
$\Ghom(\dif,1)=\dif$;
\item\label{enum:linsmp:Gt_h_outB}
$\Ghom(\dif,t)(\px) = \dif(\px)$ for each $\px\in(\overline{\Mman\setminus\Rb})\cup\Bman$ and $t\in(0;1]$;
\item\label{enum:linsmp:Gt_hlin_nearB}
if $\dif = \tfib{\dif}$ on some \myemph{star-like} neighborhood $\Uman$ of $\Bman$, so $\dif\in\DiffInvLUMB$, then $\Ghom(\dif,t) = \dif = \tfib{\dif}$ on $\Uman$ as well.
\end{enumerate}

\item\label{enum:linsmp:G_prop}
For every $t\in(0;1]$ and $\dif\in\DiffFol$
\begin{enumerate}[label={\rm(\alph*)}, ref={\rm(2\alph*)}]

\item\label{enum:linsmp:Gt_flp}
the map $\Ghom(\dif,t)$ is \FLP, and, in particular, $\func \circ \Ghom(\dif,t) = \func$.

\item\label{enum:linsmp:Gt_homomorphism}
the map $\Ghom_t:\DiffFol\to\Ci{\Mman}{\Mman}$, $\Ghom_t(\dif):=\Ghom(\dif,t)$, is a \myemph{homomorphism of monoids} with respect to the natural composition of maps, i.e.\ $\Ghom_t(\gdif\circ\dif) = \Ghom_t(\gdif)\circ \Ghom_t(\dif)$ for all $\dif,\gdif\in\DiffFol$ and $\Ghom_t(\id_{\Mman})=\id_{\Mman}$.

\item\label{enum:linsmp:Ght_homeo}
$\Ghom_t(\DiffFol)\subset\DiffFol$.
\end{enumerate}

\item\label{enum:linsmp:Gext_prop}
$\restr{\Ghom}{\DiffFol\times(0;1]}$ extends to a $\Wr{\infty,\infty}$-continuous map $\Ghom:\DiffFol\times[0;1]\to\DiffFol$ such that all statements in~\ref{enum:linsmp:G_general} and~\ref{enum:linsmp:G_prop} hold for $t=0$ as well.
Moreover, $\Ghom(\dif,0)(\px) = \tfib{\dif}(\px)$ for each $\dif\in\DiffFol$ and $\px\in\Ra$, so, in particular, $\Ghom_0(\DiffFol)\subset\DiffInv[\vbp,\Ra]{\Foliation}{\Bman} \subset \DiffLFol$.

\item\label{enum:linsmp:Gext_glogal}
Let $\Qman\subset(\overline{\Mman\setminus\Rb})\cup\Bman$ be any subset, $\eta:[0;1]\to\Mman$ is a path with $\eta(0), \eta(1)\subset\Qman$.
Then each of the following groups $\DiffLFol$, $\DiffInv[\vbp,\Ra]{\Foliation}{\Bman}$, $\DiffFixFolB$, and $\DiffHFix{\Foliation}{\Qman,\apath}$ is invariant under $\Ghom$.
Hence $\Ghom$ is a deformation of the triple
\begin{equation}\label{equ:big_triple}
\bigl(
    \DiffFol, \
    \DiffFixFolB, \
    \DiffHFix{\Foliation}{\Qman,\apath}
\bigr)
\end{equation}
into the triple
\begin{equation}\label{equ:small_triple_U}
    \bigl(
        \DiffInv[\vbp,\Ra]{\Foliation}{\Bman}, \
        \DiffFix[\vbp,\Ra]{\Foliation}{\Bman}, \
        \DiffHFix[\vbp,\Ra]{\Foliation}{\Qman,\apath}
    \bigr)
\end{equation}
and also into the intermediate triple
\begin{equation}\label{equ:small_triple}
    \bigl(
        \DiffLFol, \
        \DiffFixLFolB, \
        \DiffHFix[\vbp]{\Foliation}{\Qman,\apath}
    \bigr)
\end{equation}
obtained by intersecting~\eqref{equ:big_triple} with $\DiffInv[\vbp,\Ra]{\Foliation}{\Bman}$ and $\DiffInv[\vbp]{\Foliation}{\Bman}$ respectively.
\end{enumerate}
\end{subtheorem}
\begin{proof}
\ref{enum:linsmp:G_general}
Let $t\in(0;1]$ and $\dif\in\DiffInvMB$.
Then statement~\ref{enum:linsmp:G1} is evident.

\ref{enum:linsmp:Gt_h_outB}
By formulas for $\Ghom$, we have that $\Ghom(\dif,t) = \dif$ on $\Mman\setminus\Rb$.
Moreover, let $\px\in\Bman$.
Then by assumption $\dif(\px)\in\Bman$, whence $s\px=\px$ and $s\dif(\px)=\dif(\px)$ for all $s\in\bR$.
In particular, put $s=\phi(t,\px)$.
Then $\Ghom(\dif,t)(\px) = \tfrac{1}{s}\dif(s\px)=\dif(\px)$.

\ref{enum:linsmp:Gt_hlin_nearB}
Suppose $\dif=\tfib{\dif}$ on some star-like neighborhood $\Uman \subset \Roo$ of $\Bman$.
Then for every $\px\in\Uman$ we have that $\phi(t,\px)\px\in\Uman$ as well, and
\[
    \Ghom(\dif,t)(\px) = \frac{\dif(\phi(t,\px)\px)}{\phi(t,\px)}=
    \frac{\tfib{\dif}(\phi(t,\px)\px)}{\phi(t,\px)}=
    \frac{\phi(t,\px)\tfib{\dif}(\px)}{\phi(t,\px)}=\tfib{\dif}(\px).
\]

\ref{enum:linsmp:G_prop}
Assume now that $t\in(0;1]$ and $\dif\in\DiffFol$.

\ref{enum:linsmp:Gt_flp}
Let $\px\in\Mman$ and $\omega$ be the leaf of $\Foliation$ containing $\px$.
If $\px\in\Mman\setminus\Roo$, then $\omega\subset\Mman\setminus\Roo$ since $\Roo$ is saturated.
Therefore, $\Ghom(\dif,t)(\px) = \dif(\px) \in \omega$, since $\dif$ is \FLP.

Suppose $\px\in\Roo$.
Put $s:=\phi(t,\px)$.
Then in the notation of Lemma~\ref{lm:homo_func}\ref{enum:hf:h}, $G(\dif,t) = \dif_s(\px)$.
Since $\dif_s(\omega)\subset\omega$, we see that $\Ghom(\dif,t)(\px)\in \omega$ as well, and thus $\Ghom(\dif,t)$ is an \FLP\ map.

\ref{enum:linsmp:Gt_homomorphism}
Evidently,
\[
    \Ghom_t(\id_{\Mman})(\px) =
    \begin{cases}
        \frac{\phi(t,\px)\px}{\phi(t,\px)}=\px, & \px\in\Roo, \\
        \id_{\Mman}(\px)=\px, & \px\in\Mman\setminus\Rb
    \end{cases}
\]
so $\Ghom_t(\id_{\Mman})=\id_{\Mman}$.
Moreover, notice that
\begin{equation}\label{equ:phi_t_Ght}
    \phi(t, G_t(\dif,t)(\px))
    = t + (1-t) \mu\bigl(\func(G_t(\dif,t)(\px))\bigr)
    \stackrel{\ref{enum:linsmp:Gt_flp}}{=}
    t + (1-t) \mu(\func(\px))
    = \phi(t,\px)
\end{equation}
for all $(\dif,t,\px)\in\DiffFol\times(0;1]\times\Roo$.
Hence if $\px\in\Roo$, then
\begin{align*}
\Ghom_t(\gdif)\circ \Ghom_t(\dif)(\px) &=
\frac{\gdif\bigl(\phi(t,\Ghom_t(\dif))\Ghom_t(\dif)\bigr)}{\phi(t,\Ghom_t(\dif))}
\stackrel{\eqref{equ:phi_t_Ght}}{=}
\frac{\gdif(\phi(t,\px)\Ghom_t(\dif))}{\phi(t,\px)} =
\frac{\gdif\bigl(\phi(t,\px) \frac{\dif(\phi(t,\px)\px)}{\phi(t,\px)} \bigr)}{\phi(t,\px)} = \\
& =
\frac{\gdif\bigl(\dif(\phi(t,\px)\px)\bigr)}{\phi(t,\px)} =
\Ghom_t(\gdif\circ\dif)(\px).
\end{align*}
On the other hand, if $\px\in\Mman\setminus\Rb$, then $\Ghom_t(\gdif)\circ \Ghom_t(\dif)(\px) = \gdif\circ\dif(\px) = \Ghom_t(\gdif\circ\dif)(\px)$ as well.

\ref{enum:linsmp:Ght_homeo}
Since $\Ghom_t$ is a homomorphism of monoids, for each $\dif\in\DiffFol$ we have that
\[
    \id_{\Mman} =
    \Ghom_t(\id_{\Mman}) =
    \Ghom_t(\dif\circ\dif^{-1}) =
    \Ghom_t(\dif)\circ \Ghom_t(\dif^{-1}).
\]
As both $\Ghom_t(\dif^{-1})$ and $\Ghom_t(\dif)$ are $\Cinfty$, they are diffeomorphisms of $\Eman$.
Moreover, due to~\ref{enum:linsmp:Gt_flp}, $\Ghom_t(\dif)\in\DiffFol$.

\ref{enum:linsmp:Gext_prop}
Similarly to~\cite{KhokhliukMaksymenko:2022} one can show that the map $\Ghom$ extends to a $\Wr{\infty,\infty}$-continuous map $\Ghom:\DiffFol\times[0;1]\to\Ci{\Mman}{\Mman}$ such that $\Ghom(\dif,0) = \tfib{\dif}$ on $\func^{-1}\bigl([0,\aConst]\bigr)$ for each $\dif\in\DiffFol$.
The proof is elementary, based on Hadamard lemma, and literally the same as in~\cite{KhokhliukMaksymenko:2022}.

Now properties~\ref{enum:linsmp:G_general} and~\ref{enum:linsmp:G_prop} for $\Ghom_0$ follow from continuity of $\Ghom$, and the proof of~\ref{enum:linsmp:Gt_flp} additionally uses the fact that the leaves of $\Foliation$ are closed.
We leave the details for the reader.

\ref{enum:linsmp:Gext_glogal}
Let $\dif\in\DiffLFol$, then by~\ref{enum:linsmp:Gt_hlin_nearB}, $\Ghom(\dif,t)=\tfib{\dif}$ on some neighborhood of $\Bman$, i.e.\ $\Ghom(\dif,t)\in\DiffLFol$ as well.

Moreover, if $\dif\in\DiffInv[\vbp,\Ra]{\Foliation}{\Bman}$, so $\dif=\tfib{\dif}$ on the star-like neighborhood $\Ra$ of $\Bman$, then  by~\ref{enum:linsmp:Gt_hlin_nearB}, $\Ghom(\dif,t)=\tfib{\dif}$ on $\Ra$ as well.

Suppose $\dif\in\DiffFix{\Foliation}{\Qman}$ and $t\in[0;1]$.
Then by~\ref{enum:linsmp:Gt_h_outB}, $\Ghom(\dif,t)(\px)=\dif(\px)=\px$ for all $\px\in(\overline{\Mman\setminus\Rb})\cup\Bman \supset \Qman$.
Hence $\Ghom(\dif,t)\in\DiffFix{\Foliation}{\Qman}$.

Finally, let $\dif\in\DiffHFix{\Foliation}{\Qman,\apath}$.
By the construction of $\Ghom$, the map $\Ghom_{\dif}\colon [0;1]\times\Roo\to\Roo$, $\Ghom_{\dif}(t,\px) = \Ghom(\dif,t)(\px)$, is continuous.
Hence the map
\[
    \Omega:[0;1]\times[0;1]\to\Roo,
    \qquad
    \Omega(t,s) = \Ghom_{\dif}(t,\apath(s)) = \Ghom(\dif,t)\circ\apath(s),
\]
is a homotopy of paths relative their ends.
Therefore $\apath$, being homotopic relatively ends to $\dif\circ\apath = \Ghom(\dif,1)\circ\apath$, is also homotopic to $\Ghom(\dif,t)\circ\apath$ for all $t\in[0;1]$.
In other words, $\Ghom(\dif,t)\in\DiffHFix{\Foliation}{\Qman,\apath}$ as well.

Thus we see that $\Ghom_1=\id_{\DiffFol}$, $\Ghom_0(\DiffFol)\subset \DiffInv[\vbp,\Ra]{\Foliation}{\Bman} \subset \DiffLFol$, and $\DiffLFol$, $\DiffInv[\vbp,\Ra]{\Foliation}{\Bman}$, $\DiffFix{\Foliation}{\Qman}$, and $\DiffHFix{\Foliation}{\Qman,\apath}$ are invariant under $\Ghom$.
Hence $\Ghom$ is a deformation of~\eqref{equ:big_triple} into either of triples~\eqref{equ:small_triple_U} and~\eqref{equ:small_triple}.
\end{proof}

\subsection{Trivial vector bundles of rank $1$}\label{sect:triv_vb_1}
It this subsection we will prove several statements concerning leaf preserving diffeomorphisms for the foliation of $\Bman\times\bR$ by the sets $\Bman\times\{t\}$ being level sets of the natural projection to $\Bman\times\bR\to\bR$.
\begin{sublemma}\label{lm:collar}
Let $\Bman$ be a boundary component of $\Mman$, and $\psi\colon\Bman\times[0;+\infty) \to \Mman$ be a \myemph{collar} of $\Bman$, i.e.\ an open embedding such that $\psi(\px,0)=\px$ for all $\px\in\Bman$.
Denote $\Cman := \psi\bigl(\Bman\times[0;\aConst]\bigr)$.
Let also $\Foliation$ be any partition of $\Mman$ such that the set $\Kman:=\psi\bigl(\Bman\times[0;1]\bigr)$ is $\Foliation$-saturated and the leaves of $\Foliation$ in $\Kman$ are the sets $\psi\bigl(\Bman\times t\bigr)$, $t\in[0;1]$.
If $\Bman$ is compact, then the inclusion $\DiffFix{\Foliation}{\Cman} \subset \DiffFix{\Foliation}{\Bman}$ is a homotopy equivalence.
\end{sublemma}
\begin{proof}
We will present two proofs.
The first one is ``conceptual'': we will deduce the statement from Theorem~\ref{th:linearization_simpler}, while the second proof is short and straightforward.
Let $\Eman := \psi\bigl(\Bman\times[0;+\infty)\bigr)$.
To simplify notation, we identify $\Bman\times[0;+\infty)$ with $\Eman$ via $\psi$.
Then $\Eman$ is a subset of a total space of the trivial vector bundle $\vbp:\Bman\times\bR\to\Bman$ and is invariant under multiplication by non-negative reals $t\geq0$.

\newcommand\DiffFixFBpC{\DiffFix[\vbp,\Cman]{\Foliation}{\Bman}}
The principal observation is that
$\DiffFix{\Foliation}{\Cman}$ coincides with the subgroup $\DiffFixFBpC$ of $\DiffFix[\vbp]{\Foliation}{\Bman}$ consisting of diffeomorphisms $\dif$ coinciding with the corresponding ``vector bundle morphism'' $\tfib{\dif}:\Bman\times[0;+\infty)\to\Bman\times[0;\infty)$ of $\Eman$ on $\Cman$, which is defined by the same formulas as in Section~\ref{sect:tangent_map_along_fibers}: $\tfib{\dif}(\px) = \lim\limits_{t\to0,\, t>0}\tfrac{1}{t} \dif(t\px)$, but we use only $t>0$.

Indeed, the inclusion $\DiffFix{\Foliation}{\Cman} \subset \DiffFixFBpC$ is evident, since if $\dif$ is fixed on $\Cman$, then $\tfib{\dif}$ is the identity.
Conversely, let $\dif\in\DiffFixFBpC$.
Then $\tfib{\dif}(\px,t) = (\px, \sigma(\px) t)$, where $\sigma:\Bman\to(0;+\infty)$ is some strictly positive $\Cinfty$ function.
Since $\tfib{\dif}$ also preserves the leaves $\Bman\times\{t\}$, $t\in[0;\aConst]$, we should have that $\sigma(\px)\equiv 1$, i.e.\ $\tfib{\dif}$ is the identity, and therefore $\restr{\dif}{\Cman} = \restr{\tfib{\dif}}{\Cman} = \id_{\Cman}$, so $\dif$ is fixed on $\Cman$, and thus belongs to $\DiffFix{\Foliation}{\Cman}$.

{\bf First proof.}
Now one can define the homotopy $\Ghom$ by the same formula~\eqref{equ:G_homotopy}, and prove Theorem~\ref{th:linearization_simpler} obtaining that the inclusion $\DiffFixLFolB \equiv \DiffFixFBpC \subset \DiffFixFolB$ is a homotopy equivalence.

{\bf Second proof.}
Since $\Bman$ is not a ``singular'' leaf of $\Foliation$, the formulas for $\Ghom$ can be simplified so that they do not require passing to limit at $t=0$.
Namely, note that for each $\dif\in\DiffFixFolB$ its restriction to $\Roo$ is given by
\[
    \dif(\px,s) = (\alpha_{\dif}(\px,s), s), \qquad (\px,s)\in\Roo=\Bman\times[0;1],
\]
where $\alpha_{\dif}:\Bman\times[0;1] \to \Bman$ is some $\Cinfty$ map being a diffeomorphism for each $s\in[0;1]$, i.e.\ $\alpha$ is an isotopy.
Then the following map $\Ghom:\DiffFixFolB\times[0;1]\to\DiffFixFolB$
\begin{equation}\label{equ:def_DFB_DFBC}
\Ghom(\dif,t)(\py) =
\begin{cases}
\bigl(\alpha_{\dif}(\px, ts + (1-t)\mu(s)), \, t \bigr), & \py=(\px,s)\in\Bman\times[0;1], \\
\dif(\py), & \py\in \Mman\setminus\bigl( \Bman\times[0;1] \bigr).
\end{cases}
\end{equation}
is a deformation of $\DiffFixFolB$ into $\DiffFix{\Foliation}{\Cman}$, where $\mu:[0;\infty)\to[0;1]$ is the same as above $\Cinfty$ function satisfying $\mu([0;\aConst])=0$ and $\mu([\bConst;+\infty))=1$.
We leave the verification for the reader.
\end{proof}

\begin{sublemma}[{c.f.~\cite[Theorem~4.3]{KhokhliukMaksymenko:IndM:2020}}]
\label{lm:triv_nbh:loc_triv_restr}
Let $\Mman$ be a smooth manifold, $\Bman \subset\Int{\Mman}$ a closed submanifold of codimension $1$ equipped with a regular neighborhood $\vbp:\Bman\times\bR\to\Bman$ being a trivial vector bundle.
Let also $\Foliation$ be a partition of $\Mman$ such that the sets $\Bman\times\{t\}$, $t\in[-1,1]$, are leaves of $\Foliation$.
Then the restriction map $\rho\colon\Diff(\Foliation) \to \Diff(\Bman)$, $\rho(\dif) = \restr{\dif}{\Bman}$ is a locally trivial principal $\DiffFixLFolB$-fibration.
In particular, the image of $\rho$ is a union of path components of $\Diff(\Bman)$.
\end{sublemma}
\begin{proof}
Consider the function $\func\colon \Bman\times\bR\to\Bman$, $\func(t,\px)=t^2$.
Then $\func$ is Morse-Bott, and $\restr{\Foliation}{[-1;1]\times\Bman}$ is a partition into level sets of $\func$.
Now our statement is just a particular case of~\cite[Theorem~4.3]{KhokhliukMaksymenko:IndM:2020}.

Let us also sketch another proof which is close to the second proof of Lemma~\ref{lm:collar}.
To show that $\rho$ is a locally trivial fibration, it suffices to find an open neighborhood $\mathcal{U}$ of $\id_{\Bman}$ in $\Diff(\Bman)$ and a section $s:\mathcal{U}\to\Diff(\Foliation)$ of $\rho$, i.e.\ $\rho\circ s(\dif)=\dif$ for all $\dif\in\mathcal{U}$.

As it is proved in~\cite{Lima:CMH:1964}, there exists a neighborhood $\mathcal{U}$ of $\id_{\Bman}$ in $\Diff(\Bman)$ and a contraction of $\mathcal{\Uman}$ into $\id_{\Bman}$, i.e.\ a homotopy $H:\mathcal{U}\times[0;1]\to\Diff(\Bman)$ such that $H_1=\id_{\mathcal{\Uman}}$, $H_0(\mathcal{\Uman}) = \{\id_{\Bman}\}$, and for every $\dif\in\mathcal{U}$ the map $H_{\dif}:[0;1]\times\Bman\to\Bman$, $H_{\dif}(t,\px)=H(\dif,t,\px)$ is $\Cinfty$.
Then the desired section $s:\mathcal{U}\to\Diff(\Foliation)$ of $\rho$ can be given by
\[  
s(\dif)(\py)  =
\begin{cases}
    \bigl(t,  H(\dif, t + (1-t)\mu(t), \px) \bigr), & \py = (t,\px)\in[-\oConst;\oConst]\times\Bman \\
    \py,                                            & \py\in\Mman\setminus\bigl( [-\bConst;\bConst]\times\Rb \bigr).
\end{cases}
\]
\end{proof}

\subsection{Product foliations and loop spaces}
\label{sect:product_fol}
Our final auxiliary statement reduces computation of homotopy types of certain groups of $\Cinfty$ diffeomorphisms to the continuous situation.

Let $\Bman$ be a smooth manifold and $\Foliation = \{ \Bman\times\{t\} \mid t\in[0;1] \}$ be the foliation on $[0;1]\times\Bman$ into level sets of the projection to $[0;1]$.
Then every $\dif\in\Diff(\Foliation)$ is given by $\dif(t,\px)=(t,\alpha_{\dif}(t,\px))$, where $\alpha_{\dif}:[0;1]\times\Bman\to\Bman$ is a $\Cinfty$ isotopy, which in turn defines a $\Wr{\infty}$-continuous path $\gamma_{\dif}: [0;1]\to \Diff(\Bman)$, $\gamma_{\dif}(t) = \restr{\dif}{\Bman\times\{t\}}$.
Moreover, if we endow $\Cont{[0;1]}{\Diff(\Bman)}$ with the compact open topology, then the correspondence  $\dif\mapsto\gamma_{\dif}$ will be a well-defined continuous injective map $\Gamma:\Diff(\Foliation)\to\Cont{[0;1]}{\Diff(\Bman)}$.

Let also $\OBD = \{\gamma\in \Cont{[0;1]}{\DiffId(\Bman)} \mid \gamma(0)=\gamma(1) = \id_{\Bman}\}$ be the loop space at $\id_{\Bman}$ of the identity path component $\DiffId(\Bman)$ of the group $\Diff(\Bman)$ of diffeomorphisms of $\Bman$.
Then $\Gamma\bigl( \DiffFix{\Foliation}{\{0,1\}\times\Bman} \bigr)  \subset \OBD$.
Let also $J=[0;\aConst] \cup [\bConst;1]$.

\begin{sublemma}[{\rm\cite[Theorem~1.9]{KhokhliukMaksymenko:PIGC:2020}}]
\label{lm:loop_spaces}
The inclusions
\[
    \DiffFix{\Foliation}{\Jman\times\Bman}
    \ \xmonoArrow{j} \
    \DiffFix{\Foliation}{\{0,1\}\times\Bman}
    \ \xmonoArrow{\Gamma} \
    \OBD
\]
are weak homotopy equivalences.
In particular, if $\Bman$ is a compact surface distinct from $2$-sphere and projective plane, then the connected components of the above spaces are weakly homotopy equivalent to a point.
\end{sublemma}

\begin{subcorollary}\label{cor:null_homotopic_loops}
Suppose $\Bman$ is connected.
Let $\px\in\Bman$ and $\DXid$ be the subgroup of $\DiffFix{\Foliation}{\Jman\times\Bman}$ consisting of diffeomorphisms $\dif$ for which the loop $t \mapsto \alpha_{\dif}(t,\px)$ (being the trace of $\px$ under the isotopy $\dif$) is null-homotopic.
Let also $\OBDid$ be the path component of $\OBD$ consisting of null-homotopic loops.
Then $\Gamma\circ j(\DXid) \subset \OBDid$, and the inclusion $\DXid \xmonoArrow{\Gamma\circ j} \OBDid$ is a weak homotopy equivalence.
\qed
\end{subcorollary}

\section{Vector bundles over $1$-dimensional manifolds}\label{sect:vb_over_1manifolds}
\subsection{Vector bundles over the circle}\label{sect:vb_over_circle}
\label{sect:vb_over_s1}
Let $\hvbp\colon\bRnR\to\bR$ be the trivial vector bundle.
Consider the following diffeomorphisms
$\alpha,\beta\colon\bRnR\to\bRnR$ and $\xi\colon\bR\to\bR$ given by
\begin{align*}
    \alpha(\px,\pvi{1},\pvi{2},\ldots,\pvi{n-1},\pvi{n})  &= \bigl(\px+1, \pvi{1},\ldots,\pvi{n-1},      \pvi{n}\bigr), \\
     \beta(\px,\pvi{1},\pvi{2},\ldots,\pvi{n-1},\pvi{n})  &= \bigl(\px+1, \pvi{1},\ldots,\pvi{n-1},-\pvi{n}\bigr), \\
     \xi(\px) &= \px+1,
\end{align*}
$\px\in\bR$, and $(\pvi{1},\ldots,\pvi{n})\in\bR^{n}$.
They have no fixed points and generate free and totally disconnected actions of the group $\bZ$.
Evidently, $\bR/\xi = \Circle$ and $(\bRnR)/\alpha = \SRn$.
Moreover, the quotient $(\bRnR)/\beta$ is usually denoted by $\StRn$.
Since $\hvbp\circ\alpha=\gamma\circ\hvbp$ and $\hvbp\circ\beta=\gamma\circ\hvbp$, we have natural projections between quotients: $\por\colon\SRn \to \Circle$ and $\pnor\colon\StRn \to \Circle$ being vector bundles such that $\por$ is trivial, while $\pnor$ is non-trivial.
It is well-known and easy that every $n$-dimensional vector bundle over $\Circle$ is isomorphic either to $\por$ or to $\pnor$.
Finally we get the following commutative diagrams:
\begin{equation}\label{equ:vb_over_S1}
\begin{gathered}
    \xymatrix{
        \bRnR \ar[d]_-{\qor} \ar[r]^-{\hvbp} & \bR \ar[d]^-{\qs} \\
        \SRn                  \ar[r]^-{\por}  & \Circle
    }
    \qquad\qquad
    \xymatrix{
        \bRnR \ar[d]_-{\qnor} \ar[r]^-{\hvbp} & \bR \ar[d]^-{\qs} \\
        \StRn                 \ar[r]^-{\pnor}  & \Circle
    }
\end{gathered}
\end{equation}
in which vertical arrows are infinite cyclic covering maps, and horizontal ones are vector bundles.

We will consider below the following subgroup
\[
    \DiffPR = \{ \dif\in\Diff(\bR) \mid \dif(\px+1)=\dif(\px)+1, \ \px\in\bR \}
\]
of $\Diff(\bR)$ consisting of diffeomorphisms commuting with $\xi$.

\subsection{Main result}\label{sect:vb_over_1manifolds:main_result}
Let $\vbp\colon\Eman\to\Circle$ be a vector bundle of rank $n$ over the circle, $\vbq\colon\bRnR\to\Eman$ be the universal covering of $\Eman$,
$\gamma\colon\bRnR\to\bRnR$ be the diffeomorphism generating the group $\bZ$ of covering transformations, and $\hvbp\colon\bRnR\to\bR$ be the trivial vector bundle.

As noted in Section~\ref{sect:vb_over_s1}, one can assume that either $\Eman=\SRn$ and $\gamma=\alpha$ or $\Eman=\StRn$ and $\gamma=\beta$.

Let $\tfunc\colon\bRnR\to[0;+\infty)$ be a $\Cinfty$ function having the following properties.
\begin{enumerate}[label={\rm(F\arabic*)}]
\item\label{enum:func:0}
$\tfunc^{-1}(0)=\bR\times0$ and this set coincides with the set of critical points of $\tfunc$.

\item\label{enum:func:fg=f}
$\tfunc\circ\gamma=\tfunc$, so $\tfunc$ induces a $\Cinfty$ function $\func\colon\Eman\to[0;+\infty)$ such that $\tfunc=\func\circ\vbq$.

\item\label{enum:func:T1_bounded}
The set $\tfunc^{-1}\bigl([0;1]\bigr) \cap \bigl( [0;1]\times\bR^{n} \bigr)$ is bounded in $\bRnR$, and thus compact.

\item\label{enum:func:fld}
There exists an everywhere non-zero $\Cinfty$ vector field
\[
    \Fld = a_0(\px,\pv)\ddd{}{\px} + a_1(\px,\pv)\ddd{}{\pvi{1}} + \cdots + a_n(\px,\pv)\ddd{}{\pvi{n}}
\]
on $\bRnR$ such that $\ff{\tfunc}{\Fld}\equiv0$ and $a_0(\px,\pv)>0$ on all of $\bRnR$.
\end{enumerate}
For $t>0$ denote
\begin{align*}
    \regNbh{t} &:= \func^{-1}\bigl([0;t]\bigr), &
    \tregNbh{t} &:= \tfunc^{-1}\bigl([0;t]\bigr) = \vbq^{-1}(\regNbh{t}).
\end{align*}
Let also $\Foliation$ be the partition of $\Roo$ into connected components of level sets $\func^{-1}(t)$, $t\in[0;1]$.
It follows that each $\regNbh{t}$ is $\Foliation$-saturated.
Finally, let
\begin{itemize}
\item
$\Qman\subset\overline{\Roo\setminus\Rb} \equiv \func^{-1}\bigl([\bConst,\oConst]\bigr)$ be a non-empty closed subset;
\item
$\DiffFixFQ$ be the group of \FLP\ diffeomorphisms preserving orientation of $\Circle$ and fixed on $\Qman$;
\item
$\DiffFixFQS$ be the subgroup of $\DiffFixFQ$ fixed on $\Circle$.
\end{itemize}
Fix any path $\eta\colon[0;1]\to\Roo$ such that $\eta(0)\in\Qman$ and $\eta(1)\in\Circle$ and denote by $\DiffFixNFQS$ the subset of $\DiffFixFQS$ consisting of diffeomorphisms $\dif$ for which the paths $\eta$ and $\dif\circ\eta$ are homotopic relative to their ends.
Our aim is to prove the following
\begin{subtheorem}\label{th:DiffFixNFQS}
$\DiffFixNFQS$ is a subgroup of $\DiffFixFQ$ and the inclusion
\[
    \DiffFixNFQS \,\subset\, \DiffFixFQ
\]
is a homotopy equivalence.
\end{subtheorem}

We will prove Theorem~\ref{th:DiffFixNFQS} in Section~\ref{sect:proof:th:DiffFixNFQS} and now consider several particular cases.

\subsection{Examples}\label{sect:vb_over_1manifolds:examples}
\begin{subexample}\label{enum:DiffFixNFQS:exmp:1}\rm
Let $\tgfunc\colon\bR^n\to[0;+\infty)$ be any $\Cinfty$ function such that $\tgfunc^{-1}(0)=\{0\} \in \bR^{n}$, this point is a unique critical point, and the set $\lambda^{-1}\bigl([0;1]\bigr)$ is bounded in $\bR^{n}$.
Define the function $\tfunc\colon\bRnR\to[0;+\infty)$ and a vector field $\Fld$ on $\bRnR$ by
\begin{align*}
    \tfunc(\px,\pv) &= \tgfunc(\pv), &
    \Fld(\px,\pv) = \ddd{}{\px}.
\end{align*}
Then $\tfunc$ and $\Fld$ satisfy~\ref{enum:func:0}-\ref{enum:func:fld} for $\gamma=\alpha$, i.e.\ for the case when $\Eman=\SRn$.

If in addition, the function $\tgfunc$ does not depend on the sign of $\pvi{n}$, that is
\[ \tgfunc(\pvi{1},\ldots,\pvi{n-1},-\pvi{n})=\tgfunc(\pvi{1},\ldots,\pvi{n-1},\pvi{n}),\]
then $\tfunc$ and $\Fld$ also satisfy~\ref{enum:func:0}-\ref{enum:func:fld} for $\gamma=\beta$, i.e.\ for non-orientable case.
\end{subexample}

\begin{subexample}\label{enum:DiffFixNFQS:exmp:2}\rm
For instance, the previous example holds if $\tgfunc(\pvi{1},\ldots,\pvi{n})=\pvi{1}^2+\cdots+\pvi{n}^2$.

If in that case $\Eman=\SRn$, then $\Foliation = \{ \Circle\times S^{n-1}_{t}\}_{t\in(0;1]} \cup \{ \Circle\times 0\}$ for $n\geq2$ is a partition of $\Roo=\Circle\times D_{1}^{n}$ into \myemph{tubes}, while for $n=1$, $\Foliation = \{ \Circle\times t\}_{t\in[-1;1]}$ is a partition into parallel circles.

On the other hand, if $\Eman=\StRn$, then for $n\geq2$, $\Foliation$ is a partition of $\Roo=\Circle\stimes D_{1}^{n}$ into $\Circle\times 0$ and non-trivial sphere bundles $\Circle\stimes S^{n-1}_{t}$, $t\in(0;1]$.
In particular, for $n=2$, $\Roo$ is a solid Klein bottle, while each leaf is a Klein bottle.
Also for $n=1$, $\Foliation$ is a partition of the M\"obius band $\Eman$ into the central circle and ``total spaces of two-sheeted coverings'' $\Circle\to\Circle$.
\end{subexample}

\begin{subexample}\label{enum:DiffFixNFQS:exmp:3}\rm
Let $\tgfunc\colon\bR^{n}\to[0;\infty)$ be a function as in~\ref{enum:DiffFixNFQS:exmp:1}, and $\phi\colon\bR\to(0;+\infty)$ any $\Cinfty$ strictly positive periodic function with period $1$, i.e.\ $\phi(\px+1)=\phi(\px)$ for $\px\in\bR$.
Suppose also that \myemph{$\tgfunc$ belongs to its Jacobi ideal}, i.e.\ there exist $\Cinfty$ functions $\delta_1,\ldots,\delta_n\colon\bRnR\to\bR$ such that
\begin{equation}\label{equ:in_Jacobi_ideal}
\tgfunc  \equiv  \delta_1 \tgfunc'_{\pvi{1}} + \cdots + \delta_n \tgfunc'_{\pvi{n}}.
\end{equation}
In other words, $\tgfunc$ is a linear combination with smooth coefficients of its partial derivatives,

For example, every homogeneous function $\tgfunc\colon\bR^{n}\to[0;\infty)$ of order $k$ has that property due to the well known Euler identity:
\[
\tgfunc = \tfrac{1}{k} \sum_{i=1}^{n} \pvi{i} \tgfunc'_{\pvi{n}}.
\]
More generally, $\tgfunc$ is called \myemph{quasi homogeneous} with weights $(\omega_1,\ldots,\omega_n)$, where $\omega_1>0$ and $\sum_{i=1}^{n} \omega_i=1$, whenever
\[
\tgfunc(t^{\omega_1}\pvi{1},\ldots,t^{\omega_n}\pvi{n}) = t \tgfunc(\pvi{1},\ldots,\pvi{n}).
\]
Then
\[
\tgfunc = \sum_{i=1}^{n} \tfrac{\pvi{i}}{\omega_i} \tgfunc'_{\pvi{n}}.
\]

Define the function $\tfunc\colon\bRnR\to[0;+\infty)$ and a vector field $\Fld$ on $\bRnR$ by
\begin{align*}
    \tfunc(\px,\pv) &= \phi(\px)\tgfunc(\pv), &
    \Fld(\px,\pv)   &= \ddd{}{\px} + \tfrac{\phi'_{\px}(\px)}{\phi(\px)} \bigl( \delta_1 \ddd{}{\pvi{1}} + \ldots + \delta_n \ddd{}{\pvi{n}} \bigr).
\end{align*}
Then condition~\ref{enum:func:0} holds.
Also
\begin{align*}
\ff{\tfunc}{\Fld}(\px,\pv) &=
\phi'_{\px}(\px) \tgfunc(\pv) + \tfrac{\phi'_{\px}(\px)}{\phi(\px)} \bigl( \delta_1 \phi(\px) \tgfunc'_{\pvi{1}} + \ldots + \delta_n \phi(\px) \tgfunc'_{\pvi{n}} \bigr) \stackrel{\eqref{equ:in_Jacobi_ideal}}{\equiv} 0,
\end{align*}
which proves~\ref{enum:func:fld}.
Moreover, $\tfunc\circ\alpha(\px,\pv) = \phi(\px+1)\tgfunc(\pv)= \phi(\px)\tgfunc(\pv)=\tfunc(\px,\pv)$, so~\ref{enum:func:fg=f}  holds in the orientable case of $\Eman$.
Also if $\tgfunc$ does not depend on the sign of $\pvi{n}$, then  $\tfunc\circ\beta=\tfunc$, so~\ref{enum:func:fg=f} also holds in the orientable case of $\Eman$.
\end{subexample}

\begin{sublemma}\label{lm:improve_vf_F}
Suppose a function $\tfunc\colon\bRnR\to[0;1]$ and a vector field $\Fld$ on $\bRnR$ satisfy conditions~\ref{enum:func:0}-\ref{enum:func:fld}.
Then, replacing $\Fld$ with some other vector field, one can assume that $\Fld$ has the following extended property:
\begin{enumerate}[label={\rm(F4$'$)}]
\item\label{enum:func:fld_better}
$\Fld = \ddd{}{\px} + a_1(\px,\pv)\ddd{}{\pvi{1}} + \cdots + a_n(\px,\pv)\ddd{}{\pvi{n}}$,
$\gamma^{*}\Fld = \Fld$, and $\ff{\tfunc}{\Fld}\equiv 0$.
\end{enumerate}
This implies that $\Fld$ generates a global flow $\Flow\colon\tRoo\times\bR\to\tRoo$ satisfying
\begin{enumerate}[label={\rm(\arabic*)}]
\item\label{enum:flow:pF_shift}
$\hvbp\circ\Flow(\px,\pv,t) = t+\px$ for all $(\px,\pv,t)\in\bRnR\times\bR$;
\item\label{enum:flow:Ft_g__g_Ft}
$\Flow_t\circ\gamma=\gamma\circ\Flow_t$;
\item\label{enum:flow:f_Ft__f}
$\tfunc\circ\Flow_t = \tfunc$ for all $t\in\bR$;
\end{enumerate}
In particular, \ref{enum:flow:pF_shift} implies that for every orbit $\omega$ of $\Flow$ the restriction $\vbp\colon\omega\to\bR\times0$ is a diffeomorphism.
\end{sublemma}
\begin{proof}
Fix any $\Cinfty$ function $\nu\colon\bR\to[0;1]$ such that $\nu=1$ on $[0;1]$ and $\nu=0$ on $(-\infty;-0.2] \cup [1.2;+\infty)$, and define the following vector field $\Fld' = \nu\Fld$ is supported in $[-0.2;1.2]\times\bR^{n}$, coincides with $\Fld$ on $[0;1]\times\bR^{n}$, and, in particular, is non-zero on the latter set.

Recall that by definition $\gamma^{*}\Fld = \tang{\gamma}\circ\Fld\circ\gamma^{-1}$.
Hence $\gamma^{*}\Fld' = \tang{\gamma}\circ(\nu\Fld)\circ\gamma^{-1} = (\nu\circ\gamma^{-1})\gamma^{*}\Fld$.
Then the following formula
\[
   \Fld'' := \sum_{i\in\bZ} (\gamma^{i})^{*}\Fld' =
             \sum_{i\in\bZ} (\nu\circ\gamma^{-i})\cdot (\gamma^{i})^{*}\Fld
\]
defines a vector field on $\bRnR$, since at each $(\px,\pv)\in\bRnR$ only two summands are non-zero.
Moreover, if $a_0$ is the component of $\Fld$ at $\ddd{}{\px}$, then the corresponding component of $\GFld$ at $\ddd{}{\px}$ is
\[
    b_0 = \sum_{i\in\bZ} (\nu\cdot a_0)\circ \gamma^{-i}.
\]
This function is strictly positive, since $a_0$ is so, and thus $\Fld''$ is non-zero everywhere on $\bRnR$.

We claim that the vector field $\GFld = \tfrac{1}{b_0} \Fld''$ satisfies~\ref{enum:func:fld_better}.
Indeed, its component at $\ddd{}{\px}$ is identically $1$.
It also directly follows from the formulas for $\Fld''$ and $b_0$ that $\gamma^{*}\Fld''=\Fld''$ and $b_0\circ\gamma=b_0$.
Hence $\gamma^{*}\GFld = \tfrac{1}{b_0\circ\gamma^{-1}} \gamma^{*} \Fld'' = \tfrac{1}{b_0} \Fld''=\GFld$.
To check that $\ff{\tfunc}{\GFld}\equiv 0$, recall that $\ff{\tfunc}{\Fld} \circ \gamma^{-1} = \ff{\tfunc\circ\gamma^{-1}}{(\gamma^{*}\Fld)}$.
Since $\tfunc\circ\gamma=\tfunc$, we obtain that for each $i\in\bZ$
\[
    \ff{\tfunc}{((\gamma^{i})^{*}\Fld')} =
    \ff{\tfunc\circ\gamma^{-i}}{((\gamma^{i})^{*}\Fld')} =
    \ff{\tfunc}{\Fld'}\circ\gamma^{-i}=
    \ff{\tfunc}{\nu\Fld}\circ\gamma^{-i}=
    (\nu\circ\gamma^{-i}) \cdot \ff{\tfunc}{\Fld} = 0.
\]
Hence
\[
    \ff{\tfunc}{\GFld} =
    \tfrac{1}{b_0} \ff{\tfunc}{\Fld''} =
    \tfrac{1}{b_0} \sum_{i\in\bZ} \ff{\tfunc}{((\gamma^{i})^{*}\Fld')} = 0.
\]

Thus assume that $\Fld$ satisfies~\ref{enum:func:fld_better}.
Then property $\gamma^{*}\Fld=\Fld$ together with condition~\ref{enum:func:T1_bounded} imply that $\Fld$ is bounded on $\tRoo$, while $\ff{\tfunc}{\Fld}=0$ implies that $\Fld$ is tangent to $\partial\tRoo=\tfunc^{-1}(1)$.
Hence $\Fld$ generates a global $\Cinfty$ flow $\Flow\colon\tRoo\times\bR\to\tRoo$.
Now, conditions~\ref{enum:flow:pF_shift}-\ref{enum:flow:f_Ft__f} are just reformulation of the property~\ref{enum:func:fld_better} in terms of $\Flow$.
\end{proof}

\subsection{Proof of Theorem~\ref{th:DiffFixNFQS}}\label{sect:proof:th:DiffFixNFQS}
Denote $\tQman = \vbq^{-1}(\Qman)$, and $\tregNbh{t} = \vbq^{-1}(\regNbh{t})$ for $t\geq0$.
Fix any lifting $\teta\colon[0;1]\to\bRnR$ of $\eta$, so $\eta = \vbp\circ\teta$ and $\teta(0)\in\tQman$.
Then for each $\dif\in\DiffFixFQ$ there exists a unique $\Cinfty$ lifting map $\tdif\colon\bRnR\to\bRnR$ such that $\vbq\circ\tdif=\dif\circ\vbq$ and $\tdif(\teta(0))=\teta(0)\in\tQman$.
Thus we have a well-defined map
\[
    \sigma\colon\DiffFixFQ\to\Ci{\bRnR}{\bRnR},
    \qquad
    \sigma(\dif)=\tdif.
\]
Let us check several properties of $\sigma$.

\begin{enumerate}[wide, itemsep=1ex, topsep=1ex]
\item\label{enum:sigma:homo}
\myemph{$\sigma$ is an \myemph{injective homomorphism of monoids} with respect to the composition of maps, that is $\sigma(\gdif\circ\dif)=\sigma(\gdif)\circ\sigma(\dif)$ for all $\gdif,\dif\in\DiffFixFQ$ and $\sigma(\id_{\tRoo})=\id_{\bRnR}$.
In particular, for each $\dif\in\DiffFixFQ$, $\sigma(\dif)$ is invertible in $\Ci{\bRnR}{\bRnR}$, and therefore it is a diffeomorphism.}

Indeed, if $\gdif,\dif\in\DiffFixFQ$, then $\vbq\circ(\tgdif\circ\tdif)=\gdif\circ\vbp\circ\tdif= (\gdif\circ\dif)\circ\vbq$, so $\tgdif\circ\tdif$ is a lifting of $\gdif\circ\dif$.
Moreover, $\tgdif\circ\tdif(\teta(0))=\tgdif(\teta(0))=\teta(0)$.
Now from uniqueness of liftings, we should have that $\sigma(\gdif\circ\dif)=\tgdif\circ\tdif=\sigma(\gdif)\circ\sigma(\dif)$.

Sinilarly, $\id_{\bRnR}$ is a lifting of $\id_{\tRoo}$ such that $\id_{\bRnR}(\teta(0))=\teta(0)$, whence, again by uniqueness of such liftings, $\sigma(\id_{\tRoo})=\id_{\bRnR}$.

\item\label{enum:sigma:teta_1}
\myemph{$\dif\in\DiffFixNFQS$ if and only if $\tdif(\teta(1))=\teta(1)$.}

Suppose $\dif\in\DiffFixNFQS$, so there exits a homotopy $\{\eta_{t}\colon[0;1]\to\tRoo\}_{t\in[0;1]}$ between the paths $\eta$ and $\dif\circ\eta$ relatively their ends.
Then it lifts to a homotopy
\[ \{\teta_{t}\colon[0;1]\to\bRnR\}_{t\in[0;1]}\]
between their liftings starting at $\teta(0)$.
By uniqueness of liftings of paths for the covering map $\vbp$, we must have that $\teta=\teta_0$ and $\tdif\circ\teta=\teta_1$.
Hence the ends of these paths coincide.

Conversely, suppose that $\tdif(\teta(1))=\teta(1)$, so the paths $\teta$ and $\tdif\circ\teta$ have common ends.
Since $\bRnR$ is simply connected, there is a homotopy $\{\teta_{t}\colon[0;1]\to\bRnR\}_{t\in[0;1]}$ between these paths relatively their ends.
Then $\{\vbp\circ\teta_{t}\colon[0;1]\to\tRoo\}_{t\in[0;1]}$ is a homotopy between $\eta$ and $\dif\circ\eta$ relatively their ends.

\item\label{enum:sigma:DiffFixNFQS_group}
\myemph{$\DiffFixNFQS$ is a group}.

Let $\gdif,\dif\in\DiffFixNFQS$.
As $\sigma$ is a homomorphism,
\[
    \sigma(\gdif\circ\dif^{-1})(\teta(0))
        \stackrel{\ref{enum:sigma:homo}}{=}
    \sigma(\gdif)\circ\sigma(\dif)^{-1}(\teta(0))
        =
    \sigma(\gdif)(\teta(0))= \teta(0).
\]
Hence by~\ref{enum:sigma:teta_1}, $\dif\circ\gdif^{-1}\in\DiffFixNFQS$, so $\DiffFixNFQS$ is a group.

\item\label{enum:sigma:r_DiffPR}
\myemph{For each $\dif\in\DiffFixFQ$ we have that $\sigma(\dif)(\bR\times0)=\bR\times0$, $\restr{\sigma(\dif)}{\bR\times0} \in \DiffPR$, and the restriction map
\[
    r\colon \DiffFixFQ \to \DiffPR, \qquad
    r(\dif) = \restr{\sigma(\dif)}{\bR\times0},
\]
is a homomorphism, and $\ker(r) = \DiffFixNFQS$.}

Notice that $\bR\times 0 = \vbq^{-1}(\Circle)$.
Also, since $\Circle$ is a leaf of $\Foliation$, we have that $\dif(\Circle)=\Circle$.
Now, the identity $\vbq\circ\tdif = \dif\circ\vbq$ implies that
\[
    \tdif(\bR\times 0) \subset \vbq^{-1}\bigl( \vbq\circ\tdif(\bR\times0) \bigr) =
    \vbq^{-1}\bigl( \dif\circ\vbq(\bR\times0) \bigr) =
    \vbq^{-1}\bigl( \dif(\Circle) \bigr) =
    \vbq^{-1}(\Circle) =
    \bR\times0.
\]
By the same arguments arguments applied to $\tdif^{-1} = \sigma(\dif^{-1})$ we will get the inverse inclusion $\bR\times0 \subset \tdif(\bR\times 0)$.

Let us check that $\ker(r) = \DiffFixNFQS$.
Notice that for each $\dif\in\DiffFixFQ$, the diffeomorpism $r(\dif)$ is a lifting of $\restr{\dif}{\Circle}$, so it is uniquely determined by the value $r(\dif)(\teta(1))$.

Hence, if $\dif\in\ker(r)$, that is $r(\dif) = \id_{\bR}$, then $\restr{\dif}{\Circle}=\id_{\Circle}$, i.e.\ $\dif\in\DiffFixNFQS$.
Conversely, if $\dif\in\DiffFixNFQS$, then $r(\dif)$ is a unique lifting of $\restr{\dif}{\Circle}=\id_{\Circle}$ fixing $\teta(1)$, and therefore $r(\dif)=\id_{\bR}$, so $\dif\in\ker(r)$.

\item\label{enum:sigma:sect_of_r}
\myemph{The homomorphism $r$ is surjective and admits a section $s\colon\DiffPR\to\DiffFixFQ$, i.e.\ $s$ is a continuous map (being not necessarily a homomorphism) satifying $r\circ s = \id_{\DiffPR}$.
Hence we get a homeomorphism
\[
    \zeta\colon\DiffFixNFQS \times \DiffPR \cong \DiffFixFQ,
    \qquad
    \zeta(\dif,\phi) = s(\phi)\circ\dif.
\]
Moreover, $s$ can be choosen so that $s(\id_{\bR})=\id_{\tRoo}$, whence
\[ \zeta(\DiffFixNFQS\times\id_{\bR}) = \DiffFixNFQS.\]
Since $\DiffPR$ is contractible, this implies that $\DiffFixNFQS$ is a strong deformation retract of $\DiffFixFQ$.}

\begin{enumerate}[label={\rm5.\arabic*)}, wide, topsep=1ex, itemsep=1ex]
\item\label{enum:th:5.1_simple_case}
First consider a particular case in which $s$ can be constructed very easy.
Suppose $\tfunc(\px,\pv) =\tgfunc(\pv)$ as in Example~\ref{enum:DiffFixNFQS:exmp:1}.
Then each $\phi\in\DiffPR$ can be extended to a diffeomorphism $\tphi\colon\bRnR\to\bRnR$ by the following formula
\begin{equation}\label{equ:tphi_simple}
\tphi(\px,\pv)  =
\bigl( \mu(\func(\pv))\px + (1-\mu(\func(\pv)))\phi(\px), \pv  \bigr),
\end{equation}
where $\mu$ is the same function as above.
Evidently, $\tphi$ commutes with $\alpha$ and $\beta$, and is fixed on $\overline{\tRoo\setminus\tRb} \supset \tQman$.
Hence in either of the cases of $\Eman$, $\tphi$ induces a well-defined diffeomorphism $s(\phi)$ of $\Roo$ fixed on $\overline{\Roo\setminus\Rb} \supset \Qman$.

We claim that $s(\phi)\in\DiffFixFQ$, i.e.\ \myemph{$s(\phi)$ preserves leaves of $\Foliation$}, being by definition path components of level sets of $\func$.
Indeed, let $a\in\Eman$ and $\tilde{a}=(\px,\pv)\in\bRnR$ be any point with $\vbq(\tilde{a})=a$.
Consider the following path
\[
    \theta\colon[0;1]\to\bRnR,
    \qquad
    \theta(t)=\bigl( t\mu(\func(\pv))\px + (1-t\mu(\func(\pv)))\phi(\px), \pv  \bigr).
\]
Then $\theta(0)=\tphi(\tilde{a})$ and $\theta(1)=\tilde{a}$.
Since $\tphi$ is a lifting of $s(\phi)$, we also have that
\[
\vbq\circ\tphi(\tilde{a}) = s(\phi)\circ \vbq(\tilde{a}) = s(\phi)(a),
\]
whence $\vbq\circ\theta\colon[0;1]\to\Eman$ is a path between $s(\phi)(a)$ and $a$.
Moreover,
\[
    \func\circ(\vbq\circ\theta)(t) = \tfunc\circ\theta(t)=\tgfunc(\pv),  \ t\in[0;1],
\]
that is $\func$ is constant along $\vbq\circ\theta$ and thus $a$ and $s(\phi(a))$ belong to the same path component of the level set of $\func$.

Thus $s(\phi)\in\DiffFixFQ$ and the correspondence $\phi\mapsto s(\phi)$ is the desired section of $r$.
The relation $s(\id_{\id_{\bR}})=\id_{\Roo}$ is obvious.

\item\label{enum:th:5.2_general_case}
Consider now the general case.
By Lemma~\ref{lm:improve_vf_F} we can assume that $\Fld$ satisfies condition~\ref{enum:func:fld_better}, and therefore it generates a flow $\Flow\colon\tRoo\times\bR\to\tRoo$ satisfying conditions~\ref{enum:flow:pF_shift}-\ref{enum:flow:f_Ft__f} of that lemma.

Let also $\phi\in\DiffPR$.
Since $\bR\times0 = \tfunc^{-1}(0)$ is a non-closed orbit of $\Flow$, for each $\px\in\bR$ there exists a unique number $\uphi(\px)\in\bR$ such that
\begin{equation}\label{equ:phi_shift_u}
    \phi(\px,0) = \Flow(\px,0,\uphi(\px)), \qquad (\px,0)\in\bR\times 0.
\end{equation}
Due to \cite[Eq.~(10)]{Maksymenko:TA:2003}, since $\phi$ is $\Cinfty$, $\uphi$ will be $\Cinfty$ as well.
Define another $\Cinfty$ function $\tuphi\colon\tRoo\to\bR$ by
\[ \tuphi(\px,\pv) = (1-\mu(\tfunc(\px,\pv)))\cdot\uphi(\px)\]
which can also be written as $\tuphi = (1-\mu\circ\tfunc)\cdot(\uphi\circ\hvbp)$.
Let also $\tphi = \Flow_{\tuphi}\colon\tRoo\to\tRoo$ be the shift along orbit of $\Flow$ by the function $\tuphi$, i.e.
\begin{equation}\label{equ:tphi_general}
    \tphi(\px,\pv) := \Flow_{\tuphi}(\px,\pv) = \Flow\bigl(\px,\pv, \tuphi(\px,\pv)\bigr)
\end{equation}
We will show that \myemph{$\tphi$ is a diffeomorphism which commutes with $\gamma$ and induces a certain diffeomorphism $s(\phi)\in\DiffFixFQ$ of $\Eman$}.
Then the correspondence $\phi\mapsto s(\phi)$ will be the required section of $r$.

\begin{enumerate}[wide, label={\rm(\roman*)}, itemsep=1ex]
\item\label{enum:tphi_emb}
\myemph{$\tphi=\Flow_{\tuphi}\colon\tRoo\to\tRoo$ is an embedding fixes on $\overline{\tRoo\setminus\tRb}$.}

Since $\mu=1$ on $\overline{\tRoo\setminus\tRb}$, $\tuphi=0$ on that set, whence $\tphi$ is fixed there.

We will show that show that $\ff{\tuphi}{\Fld}>-1$.
Then, by Lemma~\ref{lm:shift:local_diff}, $\tphi$ will be a local diffeomorphism, and for every orbit $\omega \subset \tRoo$ of $\Flow$, the restriction $\restr{\tphi}{\omega}\colon\omega \to \omega$ will be injective.
This will finally imply that $\tphi$ is an embedding.

Since $\bR\times0$ is an orbit of $\Flow$, it follows from Lemma~\ref{lm:improve_vf_F}\ref{enum:flow:pF_shift}, that for every $\px\in\bR$,
$\Flow(\px,0,t) = \hvbp\circ\Flow(\px,0,t) = (\px+t, 0)$.
Hence~\eqref{equ:phi_shift_u} implies that $\uphi(\px) = \phi(\px)-\px$, similarly to the previous case~\ref{enum:th:5.1_simple_case}.
Moreover, as $\phi\colon\bR\times0\to\bR\times0$ is a preserving orientation diffeomorphism of that orbit, $\phi' > 0$, whence
\[
    \ff{\uphi}{\restr{\Fld}{\bR\times0}} \equiv \uphi'_{\px} = \phi'_{\px} - 1 > -1,
\]
which is also guaranteed by Lemma~\ref{lm:shift:local_diff}.

Further, we have that $(\uphi\circ\hvbp)'_{\pvi{i}} = 0$ for all $i=1,\ldots,n$, whence
\begin{align*}
\ff{\uphi\circ\hvbp}{\Fld}
&=
    (\uphi\circ\hvbp)'_{\px} +
    a_1 \cdot (\uphi\circ\hvbp)'_{\pvi{1}} +
    \cdots +
    a_n \cdot (\uphi\circ\hvbp)'_{\pvi{n}} =
    \uphi'_{\px}(\px) > -1.
\end{align*}

Therefore
\begin{align*}
    \ff{\tuphi}{\Fld}&=
    \ff{(1-\mu\circ\tfunc)\cdot(\uphi\circ\hvbp)}{\Fld} = \\
    &=
    \ff{1-\mu\circ\tfunc}{\Fld} \cdot (\uphi\circ\hvbp) +
    (1-\mu\circ\tfunc) \cdot \ff{\uphi\circ\hvbp}{\Fld} \\
    &=
    - \mu'(\tfunc) \cdot \underbrace{\ff{\tfunc}{\Fld}}_{=0} \cdot (\uphi\circ\hvbp) +
    (1-\mu\circ\tfunc) \cdot \uphi'_{\px} > -1,
\end{align*}
where the latter inequality holds since $\uphi'_{\px}>-1$ and $\mu$ takes values in $[0;1]$.

\item\label{enum:tphi_gamma__gamma_tphi}
Let us show that $\tphi\circ\gamma=\gamma\circ\tphi$.

Since $\phi(\px+1)=\phi(\px)+1$, it follow that
\[
    \uphi\circ\gamma(\px)=\uphi(\px+1) = \phi(\px+1) - \px + 1 = \phi(\px) - \px  = \uphi(\px)
\]
Moreover, as $\tfunc\circ\gamma=\tfunc$ and $\hvbp\circ\gamma=\hvbp$, we see that
\[
    \tuphi\circ\gamma =
    (1-\mu\circ\tfunc\circ\gamma)\cdot(\uphi\circ\hvbp\circ\gamma) =
    (1-\mu\circ\tfunc)\cdot(\uphi\circ\hvbp) = \tuphi.
\]
Hence,
\begin{align*}
    \tphi\circ\gamma(\px,\pv) &=
    \Flow\bigl(\gamma(\px,\pv), \tuphi\circ\gamma(\px,\pv)\bigr)
    =
    \Flow\bigl(\gamma(\px,\pv), \tuphi(\px,\pv)\bigr)  =
    \Flow_{\tuphi(\px,\pv)}\circ \gamma(\px,\pv)
    \stackrel{\text{Lemma~\ref{lm:improve_vf_F}\ref{enum:flow:Ft_g__g_Ft}}}{=} \\
    &= \gamma\circ  \Flow_{\tuphi(\px,\pv)} =
    \gamma\circ\Flow\bigl(\px,\pv, \tuphi(\px,\pv)\bigr)
     = \gamma\circ\tphi(\px,\pv).
\end{align*}

\item\label{enum:tphi_surjective}
We claim that \myemph{$\tphi$ is surjective}.
For each $i\in\bZ$ denote $\Qman_i = \tRoo\cap (i\times\bR^{n})$.
Then
\begin{enumerate}[leftmargin=10ex, label={\rm(a\arabic*)}]
\item $\Qman_i$ is compact due to~\ref{enum:func:T1_bounded};
\item $\Qman_i = \gamma(\Qman_0)$ for each $i\in\bZ$;
\item $\Qman_i$ intersects each orbit of $\Flow$ at a unique point, since, by the last sentence of Lemma~\ref{lm:improve_vf_F}, each orbit of $\Flow$ projects bijectively onto $\bR\times0$.
\end{enumerate}

Therefore
\begin{enumerate}[leftmargin=10ex, label={\rm(b\arabic*)}]
\item\label{enum:tphi_Qi:compact}
$\tphi(\Qman_i)$ is also compact, and, in particular, $\hvbp(\tphi(\Qman_0)) \subset [a;b]$ for some $a,b\in\bR$;
\item\label{enum:tphi_Qi:ab}
$\tphi(\Qman_i) = \tphi\circ\gamma^{i}(\Qman_0) \stackrel{\ref{enum:tphi_gamma__gamma_tphi}}{=} \gamma^{i}\circ\phi(\Qman_0)$ for each $i\in\bZ$, whence $\hvbp\bigl(\tphi(\Qman_i)\bigr)\subset[a+i,b+i]$;
\item\label{enum:tphi_Qi:uniq_intersect}
$\tphi(\Qman_i)$ also intersects each orbit of $\Flow$ at a unique point, since $\tphi$ leaves invariant each orbit of $\Flow$.
\end{enumerate}

Now let $\py=(\px,\pv)\in\tRoo$ be any point and $\omega$ be its orbit.
Take $i\in\bN$ so large that $b-i < \px < a+i$.
Then, by~\ref{enum:tphi_Qi:ab}, $\hvbp\bigl(\tphi(\Qman_{-i})\bigr) \subset [a-i,b-i]$ and $\hvbp\bigl(\tphi(\Qman_i)\bigr) \subset [a+i,b+i]$.
Hence, due to~\ref{enum:tphi_Qi:uniq_intersect}, there exists a unique point $\py'=(\px',\pv')\in \tphi(\Qman_{-i})\cap\omega$ and a unique point $py''=(\px'',\pv'')\in \tphi(\Qman_{i})\cap\omega$.
Then $\px' \leq b-i < \px < a+i \leq \px''$, and thus the point $\py$ lays on the orbit $\omega$ between the points $\py'$ and $\py''$ being images under $\tphi$ of some points from $\omega$.
Now, by the intermediate value theorem, $\py$ is also an image of some point from $\omega$.

Thus $\tphi$ is surjective, and therefore a diffeomorphism of $\tRoo$ commuting with $\gamma$.
Hence it induces a unique diffeomorphism $s(\phi)$ of $\Roo$.

\item
\myemph{$s(\phi)$ preserves path components of level sets of $\func$}, i.e.\ $s(\phi)\in\DiffFixFQ$.

The proof is literally the same as in the previous case~\ref{enum:th:5.1_simple_case}, in which the path $\theta\colon[0;1]\to\tRoo$ should be given by $\theta(t)=\Flow(\px,\pv, t\tuphi(\px,\pv))$.
We leave the details for the reader.
\end{enumerate}

Thus the correspondence $\phi\mapsto s(\phi)$ is the desired section of $r$.
This completes Theorem~\ref{th:DiffFixNFQS}.
\qed
\end{enumerate}
\end{enumerate}

\begin{subremark}\rm
Applying the construction from~\ref{enum:th:5.2_general_case} to the case~\ref{enum:th:5.1_simple_case} we see that $\Fld=\ddd{}{\px}$ and $\Flow(\px,\pv,t) = (\px+t,\pv)$ for all $\pv\in\bR^{n}$ not only for $\pv=0$.
Then \[ \tuphi(\px,\pv) = (1-\mu(\func(\pv)))\cdot(\phi(\px)-\px),\]
so~\eqref{equ:tphi_simple} can be written as follows:
\begin{align*}
    \tphi(\px,\pv)  &=
    \bigl( \mu(\func(\pv))\px + (1-\mu(\func(\pv)))\phi(\px), \pv  \bigr) = \\
    &=
    \bigl( \px + (1-\mu(\func(\pv)))\cdot(\phi(\px)-\px), \pv  \bigr) =
    \bigl( \px + \tuphi(\px,\pv), \pv  \bigr) =
    \Flow(\px,\pv, \tuphi(\px,\pv)),
\end{align*}
which coincides with~\eqref{equ:tphi_general}.
In other words, the formulas in \ref{enum:th:5.1_simple_case} constitute a particular case of~\ref{enum:th:5.2_general_case}.
\end{subremark}

\section{Proof of Theorem~\ref{th:DFdT_contr}}\label{sect:proof:th:DFdT_contr}
First we will introduce more detailed notations.
Let $\vbp\colon\Circle\times\bC\to\Circle$ be the trivial vector bundle of rank $2$.
Define the following function $\func\colon\Circle\times\bC\to[0;+\infty)$, $\func(\al,\az)=\nrm{\az}^2$, and for each $\arad\in[0;+\infty)$ put
\[
\begin{gathered}
\begin{aligned}
 \RDisk{\arad}   &:= \{\az\in\bC \mid \nrm{\az} \leq \arad\}, \qquad\qquad &
 \RCircle{\arad} &:= \{\az\in\bC \mid \nrm{\az} = \arad\}, \\
 \RTor{\arad}    &:= \func^{-1}([0;t]) = \Circle\times\RDisk{\arad}, &
 \RTD{\arad}     &:= \func^{-1}(t) = \Circle\times\RCircle{\arad},
\end{aligned} \\
\RBd{\arad} := \overline{\XTorus\setminus\RTor{\arad}} = \Circle\times\{\az \mid \arad \leq \nrm{\az} \leq 1\}.
\end{gathered}
\]
Thus $\RTor{0} = \RTD{0} = \Circle\times (0,0)$ is a circle, $\RTor{\arad}$, for $t>0$, is a \myemph{solid torus}, and $\RTD{\arad} = \partial\RTor{\arad}$ is its boundary $2$-torus.

Let $\Foliation = \{\RTD{\arad}\}_{\arad\in[0;1]}$ be the partition (``singluar foliation'') of $\XTorus$ into $2$-tori and the central circle, see Figure~\ref{fig:tor}.
We need to show that \myemph{the group $\DiffFdT$ of diffeomorphisms fixed on $\partial\XTorus$ and leaving each $\RTor{\arad}$ invariant is weakly contractible}.
\begin{figure}[thb]
	\centering
    \includegraphics[width=5cm]{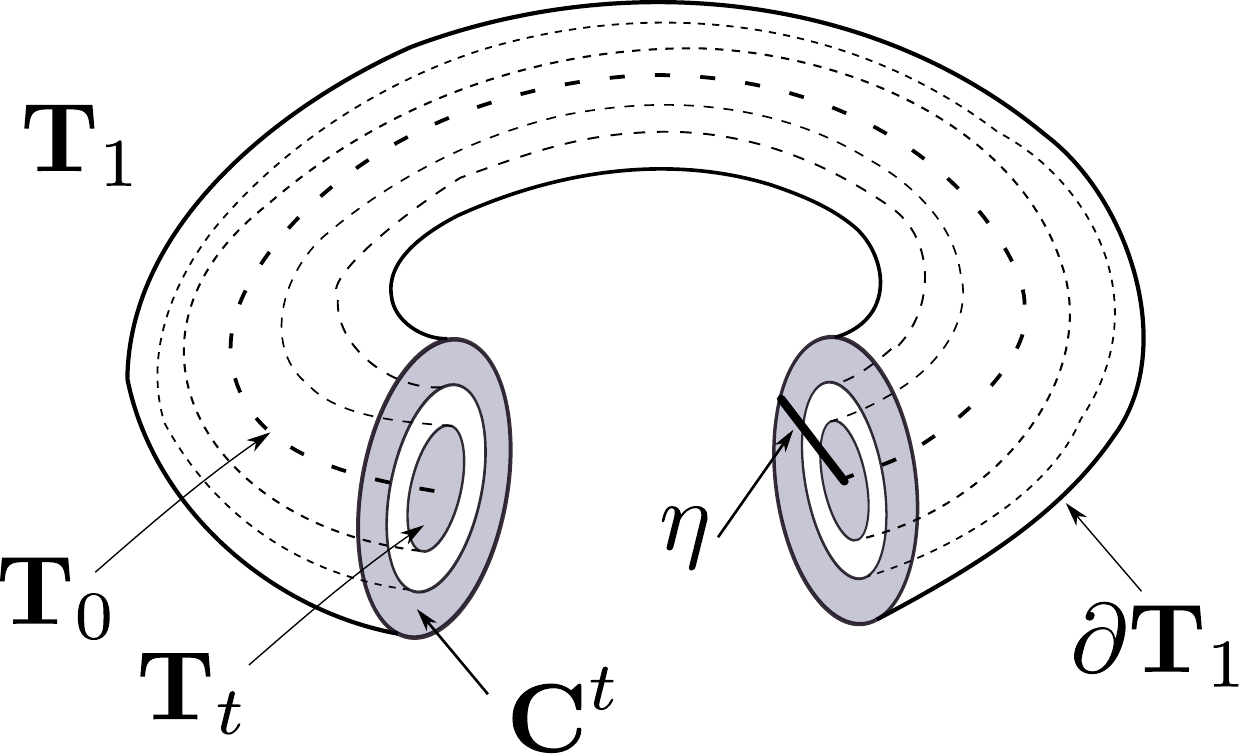}
	\caption{Foliation $\Foliation$}\label{fig:tor}
\end{figure}

Let $\apath\colon[0;1]\to\XTorus$, $\apath(\arad) = (1, \arad)$, be the radius in the disk $1\times\RDisk{1} \subset \Circle\times\RDisk{1}=\XTorus$, here $1$ and $\arad$ are reals regarded as complex numbers.
Denote by
\begin{equation}\label{equ:DiffFHNbdTNbC}
    \DiffFHNbdTNbC
\end{equation}
the group of \FLP\ diffeomorphisms of $\XTorus$ fixed on $\RBd{\bConst} \cup \RTor{\aConst}$ and such that the \myemph{half open paths} $\apath, \dif\circ\apath:(0;1]\to\XTorus\setminus\XC$ are homotopic (as maps into $\XTorus\setminus\XC$) relatively $(0;\aConst]\cup[\bConst;1]$.

Consider now the following subgroups of $\DiffFdT$:
\begin{multline*}
    \DiffFHNbdTNbC  \ \xmonoArrow{\,(4)\,} \
    \DiffFHLNbdTC   \ \xmonoArrow{\,(3)\,} \   \\ \xmonoArrow{~~~} \
    \DiffFHNbdTC    \ \xmonoArrow{\,(2)\,} \
    \DiffFNbdT      \ \xmonoArrow{\,(1)\,} \
    \DiffFdT.
\end{multline*}
Their notations (see Section~\ref{sect:diff_groups}) might look rather complicated, but they are very explicit and we will discuss each of these groups and the corresponding inclusion just below.
For the convenience we will denote them respectively by $\ggi{4}, \ggi{3}, \ggi{2}, \ggi{1}, \ggi{0}$, so that $\ggi{i+1}\subset\ggi{i}$ for $i=0,1,2,3$.
Our plan is to show that
\begin{itemize}
   \item each of the these inclusions $\ggi{i+1}\subset\ggi{i}$, $i=0,1,2,3$, is a homotopy equivalence (we will even present explicit formulas for a deformation of $\ggi{i}$ into $\ggi{i+1}$), and
   \item the smallest group $\ggi{4}:=\DiffFHNbdTNbC$, see~\eqref{equ:DiffFHNbdTNbC}, is weakly contractible.
\end{itemize}
This will imply a weak contractibility of $\ggi{0}:=\DiffFdT$.
In fact, all those statements, except for the inclusion (4), are particular cases of results established in previous sections.

\begin{figure}[htbp!]
	\centering
    \includegraphics[height=5cm]{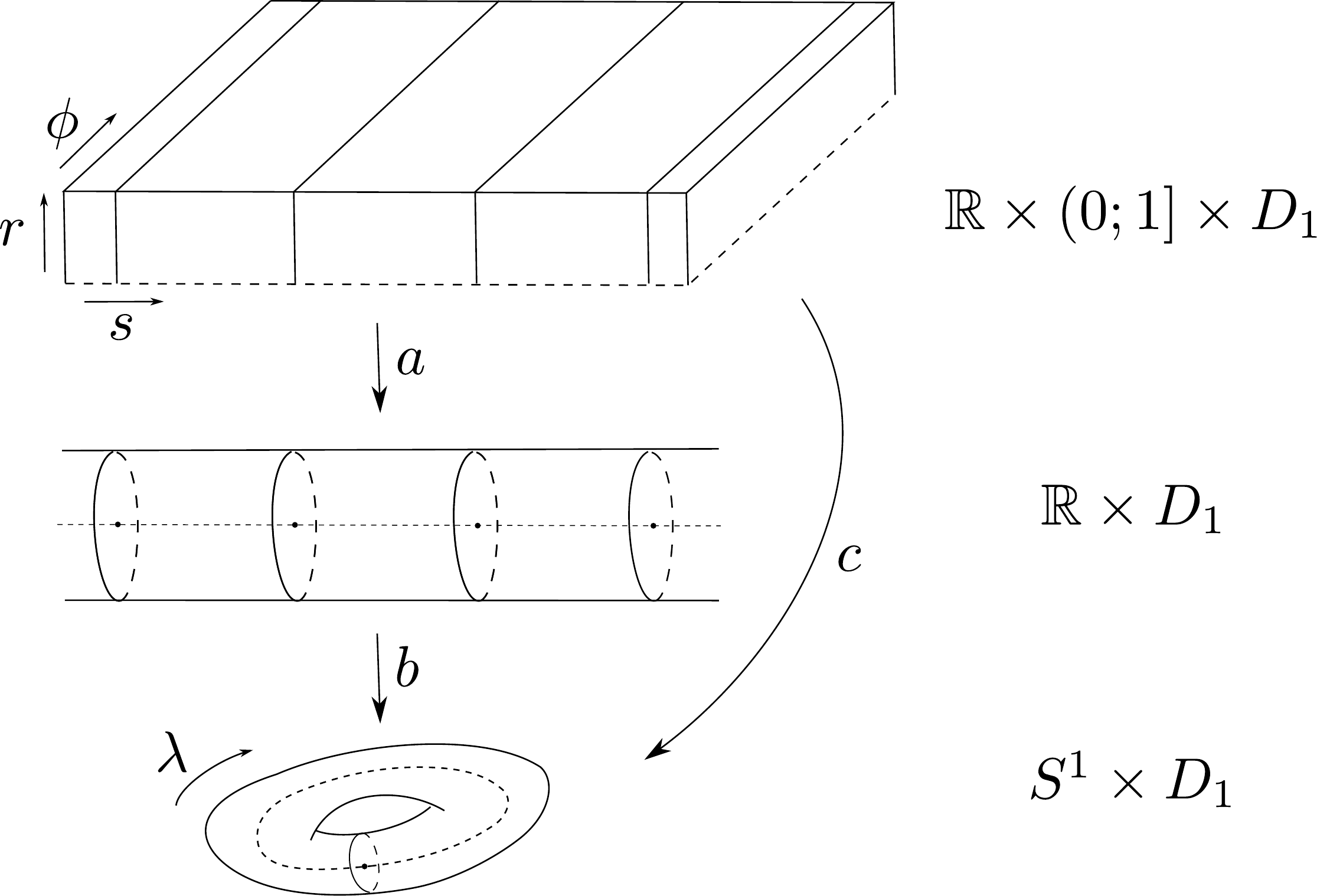}
	\caption{Covering spaces}\label{fig:covering}
\end{figure}

For the proof consider the following two universal covering maps, see Figure~\ref{fig:covering}:
\begin{align*}
\vba&\colon \bR\times(0;1]\times\bR \to \bR\times(\RDisk{1}\setminus0),  & & \vba(\as,\arad,\aphi) = (\as, \arad e^{2\pi i \aphi}), \\
\vbb&\colon \bR\times\RDisk{1} \to\XTorus,                               & & \vbb(\as,\az)         = (e^{2\pi i \as}, \az), \\
\vbc=\vbb\circ\vba&\colon \bR\times(0;1]\times\bR\to\XTorus\setminus\XC, & & \vbc(\as,\arad,\aphi) = (e^{2\pi i \as}, \arad e^{2\pi i \aphi}).
\end{align*}

Let $\dif=(\ACircle,\ADisk)\colon \XTorus\to\XTorus$ be a diffeomorphism belonging to $\ggi{0}=\DiffFdT$, so $\dif$ is fixed on $\partial\XTorus$ and leaves invariant each leaf of $\Foliation$.
In particular, $\dif(\XC)=\XC$.
Then $\dif$ lifts to a unique diffeomorphism
\[ \ldif=(\BRline,\BDisk)\colon\bR\times\RDisk{1}\to\bR\times\RDisk{1}\]
fixed on the boundary $\partial\bR\times\RDisk{1} = \bR\times\Circle$ and leaving invariant the tubes $\bR\times\RCircle{\arad}$, $\arad\in[0;1]$.
In particular, $\ldif$ also leaves invariant $\bR\times0$ and its complement $\bR\times(\RDisk{1}\setminus0)$.
Hence, the restriction $\restr{\ldif}{\bR\times(\RDisk{1}\setminus0)}$ lifts in turn to a unique diffeomorphism
\[
    \lldif=(\CRline, \CRadius, \CArg)\colon \bR\times(0;1]\times\bR \to  \bR\times(0;1]\times\bR
\]
fixed on the corresponding boundary $\bR\times 1 \times\bR = \partial(\bR\times(0;1]\times\bR)$ and leaving invariant the planes $\bR\times\arad\times\bR$, $\arad\in(0;1]$.

In terms of the coordinate functions the property of preserving leaves for $\dif$, $\ldif$, and $\lldif$ can be expressed as follows:
\begin{align*}
 & \nrm{\ADisk(\al,\az)} = \nrm{\BDisk(\as,\az)} = \nrm{\az}, &
 & \CRadius(\as,\arad,\aphi)\equiv \arad.
\end{align*}
Moreover, another properties of being liftings, i.e.\ $\dif\circ\vbb = \vbb\circ\ldif$ and $\ldif\circ\vba = \vba\circ\lldif$, can be rewritten as follows:
\begin{align}
    \dif(e^{2\pi i \as},\az)  &= \bigl( e^{2\pi i \BRline(\as,\az)}, \, \BDisk(\as,\az) \bigr), \\
    \ldif(\as, \arad e^{2\pi i \aphi}) &= \bigl( \CRline(\as,\arad,\aphi), \, r e^{2\pi i \CArg(\as,\arad,\aphi)} \bigr).
\end{align}
Finally notice that $\ldif$ and $\lldif$ commute with the corresponding covering slices, and thus they satisfy the following identities for all $k,l\in\bZ$:
\begin{align}
\label{equ:ldif:comm_with_slices}
\ldif(\as+k,\az)            &= (\BRline(\as,\az)+k, \, \BDisk(\as,\az)), \\
\label{equ:lldif:comm_with_slices}
\lldif(\as+k,\arad,\aphi+l) &= (\CRline(\as,\arad,\aphi)+k, \, r, \, \CArg(\as,\arad,\aphi)+l),
\end{align}

Now we can prove that the inclusions (1)-(4) are homotopy equivalences.
As above, fix a $\Cinfty$ function $\mu:[0;\infty)\to[0;1]$ such that $\mu([0;\aConst])=0$ and $\mu([\bConst;+\infty))=1$.

\begin{itemize}[wide, topsep=1ex, itemsep=1ex]
\item[(1)]
\myemph{Inclusion $\DiffFNbdT\subset\DiffFdT$}.

Thus $\ggi{0}:=\DiffFdT$ is the group of \FLP\ diffeomorphisms fixed on the boundary $\partial\XTorus$, while $\ggi{1}:=\DiffFNbdT$ is its subgroup fixed on the $\Foliation$-saturated collar $\RBd{\bConst}$ of $\partial\XTorus$.
Then by Lemma~\ref{lm:collar}, the inclusion (1) is a homotopy equivalence.

More explicitly, note that the map 
\[ 
    \psi\colon \partial\XTorus\times[t;1] \to \RBd{t} \subset \XTorus, 
    \qquad 
    \psi(\al,\az,\arad)=(\al,\arad\az),
\]
is a \myemph{collar} of $\partial\XTorus$ such that $\psi(\partial\XTorus\times\arad)$, $\arad\in[t;1]$, is a leaf of $\Foliation$.
Then the deformation $\Ghom:\ggi{0}\times[0;1]\to\ggi{0}$ of $\ggi{0}$ into $\ggi{1}$, given by formula~\eqref{equ:def_DFB_DFBC}, can be rewritten as follows:
\[
\Ghom(\dif,t)(\al,\az) =
\begin{cases}
\dif\Bigl(\al,    (t\nrm{\az} + (1-t)\mu(\nrm{\az})) \frac{\az}{\nrm{\az}} \Bigr), & \az\not=0, \\
\dif(\al,\az), & \nrm{\az}\leq \aConst.
\end{cases}
\]

\item[(2)]
\myemph{Inclusion $\DiffFHNbdTC \subset \DiffFNbdT$}.

By definition the group $\ggi{2}:=\DiffFHNbdTC$ consists of \FLP\ diffeomorphisms $\dif$ of $\XTorus$ fixed on the central circle $\XC$ and the collar $\RBd{\bConst}$ of the boundary $\partial\XTorus$, and such that the radius $\apath$ and its image $\dif\circ\apath$ are homotopic relatively to their ends.

Then the inclusion (2) is a homotopy equivalence due to ``longitude unlooping'' Theorem~\ref{th:DiffFixNFQS}.
Notice that $\Foliation$ is exactly the same as in Example~\ref{enum:DiffFixNFQS:exmp:2} for $n=2$, and the case~\ref{enum:th:5.1_simple_case} of the proof of that theorem also corresponds to our foliation $\Foliation$.
We will now write down exact formulas for the deformation of $\ggi{1}$ into $\ggi{2}$.

Let $\dif\in\ggi{1}$ and $\ldif$ be its lifting as above.
Since $\dif$ preserves the central circle $\XC$, we have that $\ldif(\bR\times0)=\bR\times0$, so $\ldif(s,0) = \bigl(\BRline(\as,0), 0)$, $\as\in\bR$.

It is easy to check that $\dif\in\ggi{2}$ iff $\BRline(\as)=\as$ for all $\as\in\bR$, i.e.\ $\ldif$ is fixed on $\bR\times 0$, see~\ref{enum:sigma:r_DiffPR} in Section~\ref{sect:proof:th:DiffFixNFQS} of the proof of Theorem~\ref{th:DiffFixNFQS}.

Moreover, due to~\eqref{equ:ldif:comm_with_slices}, $\BRline(\as+1,0)=\BRline(\as,0)+1$.
In particular, the function $\lufunc:\bR\to\bR$, given by $\lufunc(\as) = \BRline(\as,0)-s$, is $1$-periodic, and thus induces a well-defined $\Cinfty$ function $\ufunc:\Circle\to\bR$ such that $\lufunc(\as) = \ufunc(e^{2\pi i \as})$, $\as\in\bR$, and $\dif(\al,0) = (e^{2\pi i \ufunc(\al)},0)$, $\al\in\Circle$.

Define the following isotopy $\lift{\Gamma}:(\bR\times\RDisk{1})\times[0;1]\to\bR\times\RDisk{1}$ by
\[
    \lift{\Gamma}(\as,\az, t) = \bigl(\as + t(1-\mu(\nrm{\az})) \lufunc(\as), \, \az \bigr).
\]
Then $\lift{\Gamma}_0 = \id_{\bR\times\RDisk{1}}$, each $\lift{\Gamma}_t$ is fixed on $\bR\times\{\az\mid \nrm{\az}\in[\bConst;1]\} = \vbb^{-1}(\RBd{\bConst})$ and preserves the tubes $\bR\times\RCircle{\arad}$, $\arad\in[0;1]$, and $\lift{\Gamma}_1(\as,0) = \ldif(\as,0)$ for all $\as\in\bR$.
Moreover, $\lift{\Gamma}_t$ commutes with the transformation $(\as+1,\az) \mapsto(\as+1,\az)$, and therefore induces an isotopy $\Gamma:\XTorus\times[0;1]\to\XTorus$ such that $\Gamma_t\circ\vbb = \vbb\circ\lift{\Gamma}_t$.
In particular, $\Gamma_0 = \id_{\XTorus}$, each $\Gamma_t$ is fixed on $\RBd{\bConst}$ and preserves the leaves $\Circle\times\RCircle{\arad}=\RTD{\arad}$, $\arad\in[0;1]$ of $\Foliation$, and $\Gamma_1(\al,0) = \dif(\al,0)$ for all $\al\in\Circle$.
Now the deformation $\Ghom:\ggi{1}\times[0;1]\to\ggi{1}$ of $\ggi{1}$ into $\ggi{2}$ can be given by the following formula:
\begin{equation}\label{equ:unloop_paralels}
\Ghom(\dif,t) = (\Gamma_{t})^{-1}\circ\dif.
\end{equation}

\item[(3)]
\myemph{Inclusion $\DiffFHLNbdTC \subset \DiffFHNbdTC$}.

By definition, $\ggi{3}:=\DiffFHLNbdTC$ is the subgroup of $\ggi{2}$ consisting of diffeomorphisms $\dif$ coinciding with the vector bundle morphism $\tfib{\dif}:\Circle\times\bC\to\Circle\times\bC$ on $\RTor{\aConst}$.
In that case, since $\tfib{\dif}$ must preserve leaves of the foliation $\Foliation$, it follows that each disk $\al\times\RDisk{\aConst}$, $\al\in\Circle$, is invariant under $\dif$ and the restriction of $\dif$ to $\al\times\RDisk{\aConst}$ is a rotation.

Now the statement that the inclusion (3) is a homotopy equivalence is given by Linearization Theorem~\ref{th:linearization_simpler}\ref{enum:linsmp:Gext_glogal}.
More precisely, define the following $\Cinfty$ function
\[
    \psi\colon[0;1]\times\bR^2\to[0;1],
    \qquad
    \psi(t,\az) = t + (1-t)\mu(\nrm{\az}).
\]
Then the deformation $\Ghom:\ggi{2}\times[0;1]\to\ggi{2}$ of $\ggi{2}$ into $\ggi{3}$ can be given by formula~\eqref{equ:G_homotopy} which in our case can be rewritten as follows:
\[
\Ghom(\dif,t)(\al,\az) =
\begin{cases}
    \tfib{\dif}(\al,\az), & \psi(\al,\az)=0, \ \text{i.e.\ $t=0$ and $\nrm{\az} \leq \aConst$}, \\[1mm]
    \left(
        \ACircle\bigl(\al, \psi(t,\az)\bigr), \
        \frac{\ADisk(\al, \psi(t,\az)\az}{\psi(\al,\az)}
    \right), & \text{otherwise}.
\end{cases}
\]

\item[(4)]
\myemph{Inclusion $\DiffFHNbdTNbC \subset \DiffFHLNbdTC$}.

Let $\dif\in\ggi{3}:=\DiffFHLNbdTC$ and $\ldif$ and $\lldif$ be its liftings as above.
Then $\dif$ preserves each disk $\al\times\RDisk{\aConst}$, $\al\in\Circle$, and induces a rotation of that disk.
Moreover, we know from (2) that $\ldif$ is fixed on $\bR\times0$, whence $\ldif$ preserves each disk $\as\times\RDisk{\aConst}$, $\as\in\bR$, and also induces rotation on it.
Therefore, $\lldif$ preserves the strips $\as\times(0;\aConst]\times\bR$ and induces on each of them a constant shift of the last coordinate $\aphi$.
In other words,
\[
    \lldif(\as,\arad,\aphi) = (\as,\arad,\aphi + \llufunc(\as)),
\]
where $\llufunc:\bR\to\bR$ is a $\Cinfty$ function, which is also $1$-periodic due to~\eqref{equ:lldif:comm_with_slices}, i.e.\ $\llufunc(\aphi+1) = \llufunc(\aphi)$, and therefore it induces a $\Cinfty$ function $\ufunc:\Circle\to\bR$ such that $\dif(\al, \az) = \bigl(\al, e^{2\pi i \ufunc(\al)}\az\bigr)$ for $(\al,\az)\in\RTor{\aConst} = \Circle\times\RDisk{\aConst}$.

Out main observation is that $\dif\in\DiffFHNbdTNbC$ if and only if $\llufunc\equiv0$, and thus $\ufunc\equiv 0$.

Define now the following isotopy $\lift{\Gamma}:(\bR\times(0;1]\times\bR)\times[0;1]\to\bR\times(0;1]\times\bR$ by
\[
    \llift{\Gamma}(\as,\arad,\aphi, t) = \bigl(\as , \, \arad, \aphi + t(1-\mu(\nrm{\az})) \llufunc(\as) \bigr).
\]
Then $\llift{\Gamma}_0 = \id_{\bR\times(0;1]\times\bR}$, each $\llift{\Gamma}_t$ is fixed on $\bR\times\{\az\mid \nrm{\az}\in[\bConst;1]\}\times\bR = \vbc^{-1}(\RBd{\bConst})$ and preserves the planes $\bR\times\arad\times\bR$, $\arad\in[0;1]$, and $\llift{\Gamma}_1$ is fixed on $\bR\times(0;\aConst]\times\bR$.

Moreover, it is evident that $\llift{\Gamma}_t$ commutes with the transformations
\[ (\as,\arad,\aphi) \mapsto(\as+k,\arad,\aphi+l), \quad k,l\in\bZ,\]
and therefore it induces an isotopy $\Gamma:(\XTorus\setminus\XC)\times[0;1]\to(\XTorus\setminus\XC)$ fixed on $\XC$ and such that such that $\Gamma_t\circ\vbc = \vbc\circ\llift{\Gamma}_t$.
Note that for each $\al\in\Circle$ the restriction of $\Gamma_t$ to $\al\times\RDisk{\aConst}\setminus0$ is a rotation, whence $\Gamma$ extends to an isotopy $\Gamma:\XTorus\times[0;1]\to\XTorus$ fixed on $\XC$.

Thus, $\Gamma_0 = \id_{\XTorus}$, each $\Gamma_t$ is fixed on $\RBd{\bConst}$ and preserves the leaves $\Circle\times\RCircle{\arad}=\RTD{\arad}$, $\arad\in[0;1]$ of $\Foliation$, and $\Gamma_1(\al,\az) = \dif(\al,\az)$ for all $(\al,\az)\in\RTor{\aConst}$.
Now the deformation $\Ghom:\ggi{3}\times[0;1]\to\ggi{3}$ of $\ggi{3}$ into $\ggi{4}$ can be given by the following formula similar to~\eqref{equ:unloop_paralels}:
\begin{equation}\label{equ:unloop_meridians}
    \Ghom(\dif,t) = (\Gamma_{t})^{-1}\circ\dif.
\end{equation}
The difference with~\eqref{equ:unloop_paralels} is that in the case (2), $\Gamma$ ``unloops'' $\dif$ along parallels, while now, in~\eqref{equ:unloop_meridians}, $\Gamma$ ``unloops'' $\dif$ along meridians.

\item[(5)]
Let us show finally that \myemph{the group $\ggi{4}:=\DiffFHNbdTNbC$ is weakly homotopy equivalent to the point, i.e. it is path connected and all its homotopy groups vanish}.

Let $\GFoliation = \{\partial\XTorus\times t\}_{t\in[0;1]}$ be the foliation on the direct product $\partial\XTorus\times[0;1]$ into parallel $2$-tori, and $\Jman=[0;\aConst]\cup[\bConst;1]$.
Then there exists a natural isomorphism of topological groups
\[ \sigma:\DiffGJT \cong \DiffFNbdTNbC. \]
Indeed, similarly to (1), we have the following diffeomorphism $\psi\colon \partial\XTorus\times(0;1] \to \XTorus\setminus\XC$, $\psi(\al,\az,\arad)=(\al,\arad\az)$, sending leaves of $\GFoliation$ onto leaves of $\Foliation$.
Then one can define $\sigma(\dif): \XTorus\to\XTorus$ by
\[
    \sigma(\dif)(\al,\az,\arad) =
    \begin{cases}
        \psi\circ\dif\circ\psi^{-1}(\al,\az,\arad), & t>0,\\
        (\al,0), & \az=0.
    \end{cases}
\]

Further, as noticed in Section~\ref{sect:product_fol}, every $\dif\in\DiffGJT$ can be regarded as an isotopy of $T^2$ which starts and finishes at $\id_{T^2}$.
This gives an inclusion
\[ j:\DiffGJT \to \Omega(\DiffId(T^2),\id_{T^2})\]
into the loop space at $\id_{T^2}$ of the identity path component $\DiffId(T^2)$ of the group of diffeomorphism of the torus $T^2$.
Moreover, by Lemma~\ref{lm:loop_spaces}, $j$ is a weak homotopy equivalence.
Hence the inclusion $j\circ\sigma:\DiffFNbdTNbC \to \Omega(\DiffId(T^2),\id_{T^2})$ is a weak homotopy equivalence as well.

This also implies that for every path component $\mathcal{X}$ of $\Omega(\DiffId(T^2),\id_{T^2})$ the restriction
\[ j\circ\sigma: (j\circ\sigma)^{-1}(\mathcal{X}) \to \mathcal{X} \]
is also a weak homotopy equivalence.
Then, by Corolary~\ref{cor:null_homotopic_loops}, $j\circ\sigma$ maps $\DiffFNbdTNbC$ onto the path component $\mathcal{X}_0$ of $\Omega(\DiffId(T^2),\id_{T^2})$ consisting of null-homotopic loops.

It remains to note, e.g.~\cite{EarleEells:JGD:1969, Gramain:ASENS:1973}, see also Lemma~\ref{lm:hom_type_DiffT2} below, that
\[
    \pi_k \mathcal{X}_0 = \pi_{k} \Omega(\DiffId(T^2),\id_{T^2}) \cong \pi_{k+1}(T^2) = 0,
    \quad k\geq0,
\]
Hence $\mathcal{X}_0$ and therefore $\DiffFNbdTNbC$ are weakly contractible.
\end{itemize}
Theorem~\ref{th:DFdT_contr} is completed.

\section{Diffeomorphisms of the solid torus}\label{sect:diff_of_T2}
In this section we recall the structure of the homotopy types of the groups $\Diff(\MTor)$, $\Diff(\ATor)$, $\Diff(\ATor,\MTor)$, and, in particular, show that the homotopy type of $\Diff(\ATor)$ can be computed from the contractibility of $\Diff(\ATor,\MTor)$, see Lemmas~\ref{lm:hom_type_DiffT2} and~\ref{lm:DiffFixTdT_contr}.

Then we pass to the case of $\AFoliation$-leaf preserving diffeomorphism, and also show that the homotopy type of $\Diff(\AFoliation)$ can be deduced by the similar arguments from contractibility of $\Diff(\AFoliation,\MTor)$.
In particular, this allows to obtain Theorem~\ref{th:DFdT_full_variant} from Theorem~\ref{th:DFdT_contr}, see Section~\ref{sect:proof:th:DFdT_full_variant}.

The exposition below is well-known and can be found in different sources with different variations, e.g.~\cite{Reidemeister:AMSUH:1935, Bonahon:Top:1983, KalliongisMiller:KMJ:2002}.
We will try to make it elementary, short, and self-contained as much as possible, briefly repeating known facts and arguments.
This approach allows to distinguish the principal technical parts which are new in this paper from the known classical arguments.

\subsection{Diffeomorphisms of $2$-torus}
Let $\pi:\bR^2\to\MTor$, $\pi(\pu,\pv) = (e^{2\pi i \pu}, e^{2\pi i \pv})$, be the universal covering map which is also a homomorphism with $\ker(\pi)=\bZ^2$.
For $A = \amatr{r}{p}{s}{q} \in\GL(2,\bZ)$, $b=\avect{k}{l}\in\bR^2$, and $c = (\alpha,\beta)\in\MTor$ define the following diffeomorphisms:
\begin{align}
&\affr{A}{b}:\bR^2\to\bR^2, & & \affr{A}{b}(\px) = b + A\px, \ \px\in\bR^2,     \nonumber \\
\label{equ:afft_Ab}
&\afft{A}{c}:\MTor\to\MTor, & & \afft{A}{c}(\al,\az) = \bigl( \alpha \al^{r}\az^{p}, \, \beta \al^{s}\az^{q} \bigr), \ (\al,\az)\in\MTor.
\end{align}
One easily checks that $\pi \circ \affr{A}{b} = \afft{A}{\pi(b)} \circ \pi$, i.e.\ $\affr{A}{b}$ is a lifting of $\afft{A}{\pi(b)}$.
Moreover, let
\begin{align*}
    &\AffSubgr = \{ \affr{A}{b} \mid A\in \GL(2,\bZ),  \ b\in\bR^2\}, &
    &\tAffSubgr = \{ \afft{A}{c} \mid A\in \GL(2,\bZ), \ c\in\MTor\}
\end{align*}
be the subgroups of $\Diff(\bR^2)$ and $\Diff(\MTor)$ generated by those diffeomorphisms.
Then the correspondence
$\matrt:\AffSubgr \to \tAffSubgr$, $\matrt(\affr{A}{b}) = \afft{A}{\pi(b)}$, is an epimorphism whose kernel is the group $\bZ^2$ of translations by integer vectors.
We also have the following short exact sequence
\[
    \bR^2 \xmonoArrow{ ~ b \ \mapsto \  \affr{E}{b} ~  } \AffSubgr \xepiArrow{ ~ \affr{A}{b} \ \mapsto \ A ~ } \GL(2,\bZ),
\]
admitting a section $s: \GL(2,\bZ) \to \AffSubgr$, $s(A)(\px) := \affr{A}{0}(\px) = A\px$.
Since
\begin{equation}\label{equ:AhbAinv}
    s(A)\circ \affr{E}{b} \circ s(A)^{-1}(\px) =
    A \affr{E}{b}(A^{-1}\px) = A (b + A^{-1}\px) = Ab + \px = \affr{E}{Ab},
\end{equation}
for every $A\in \GL(2,\bZ)$, we see that $\AffSubgr \cong \bR^2\rtimes \GL(2,\bZ)$ is a semidirect product corresponding to the natural action of $\GL(2,\bZ)$ on $\bR^2$.
We will identify $\GL(2,\bZ)$ with a subgroup of $\AffSubgr$ via $s$.
Clearly, $\bR^2$ is the identity path component of $\AffSubgr$, and $\pi_0\AffSubgr = \GL(2,\bZ)$.

The image $\matrt(\bR^2)$ of all translations is often called the \myemph{rotation subgroup} of $\Diff(\MTor)$ and is denoted by $\RotSub$.
Since $\ker(\matrt) = \bZ^2\subset \bR^2$, we have that $\RotSub\cong \bR^2/\bZ^2$ is a $2$-torus and it consists of diffeomorphisms
\begin{equation}\label{equ:incl:rotsubgr}
    \afft{A}{(\alpha,\beta)}(\al,\az) = (\alpha\al, \beta\az), \qquad (\alpha,\beta), (\al,\az)\in\MTor.
\end{equation}
More precisely, $\RotSub$ coincides with the group of left shifts of the Lie group $\MTor$.
It follows further that $\tAffSubgr$ is a semidirect $\RotSub \rtimes \GL(2,\bZ)$ corresponding to the induced action of $\GL(2,\bZ)$ on $\RotSub=\bR^2/\bZ^2$, $\RotSub$ is the identity path component of $\tAffSubgr$, and $\pi_0\tAffSubgr \cong \pi_0\AffSubgr = \GL(2,\bZ)$.
One might think of $\tAffSubgr$ as a group of ``affine'' diffeomorphisms of $\MTor$ and also regard $\GL(2,\bZ)$ as a subgroup of $\Diff(\MTor)$.
The following statement is a classical result about the homotopy type of $\Diff(\MTor)$.
\begin{sublemma}[\cite{EarleEells:JGD:1969, Gramain:ASENS:1973}]\label{lm:hom_type_DiffT2}
The inclusion $\RotSub\subset\DiffId(\MTor)$ is a homotopy equivalence.
Moreover, so is the inclusion $\tAffSubgr \subset \Diff(\MTor)$ and, in particular, $\pi_0\Diff(\MTor)\cong \GL(2,\bZ)$.
\end{sublemma}
\begin{proof}[Sketch of proof]
The statement that $\RotSub\subset\DiffId(\MTor)$ is a homotopy equivalence is proved in~\cite{EarleEells:JGD:1969, Gramain:ASENS:1973}.
Further notice that each $\dif\in\Diff(\MTor)$ yields an automorphism of the first homology group $H_1(\MTor,\bZ)$ which is isomorphic with $\pi_1 T^2\cong \bZ^2$ being in turn the group of covering translations of the universal cover $\pi:\bR^2\to\MTor$.
Choose an identification $H_1(\MTor,\bZ) = \bZ^2$ so that $[\Circle\times 1] = (1,0)$ and $[1\times\Circle]=(0,1)\in\bZ^2$.
This gives the action homomorphism $\hhom:\Diff(\MTor)\to\Aut(\bZ^2) = \GL(2,\bZ)$.
Let $p:\Diff(\MTor) \to \pi_0\Diff(\MTor)=\Diff(\MTor)/\DiffId(\MTor)$ be the natural projection associating to each $\dif\in\Diff(\MTor)$ its isotopy class.
Since isotopic diffeomorphisms yield the same automorphisms on homologies, there exists a unique homomorphism $\hat{\hhom}:\pi_0\Diff(\MTor)\to\GL(2,\bZ)$ such that $\hhom = \hat{\hhom}\circ p$.
Thus we get the following diagram
\[
\xymatrix@R=1em{
    \bR^2 \ar@{^(->}[d] \ar@{->>}[r]^-{\matrt} & \RotSub \ar@{^(->}[d] \ar@{^(->}[r] & \DiffId(\MTor)  \ar@{^(->}[d] \\
    \AffSubgr \ar@{->>}[r]^-{\matrt} \ar@/^5pt/@{^->>}[d] & \tAffSubgr \ar@{^(->}[r] \ar@{->>}[d]&  \Diff(\MTor) \ar@{->>}[d]^{p}  \ar@/^8pt/[dl]_-   {\hhom} \\
    \GL(2,\bZ) \ar@/^5pt/[u]^{s} \ar@{=}[r] & \GL(2,\bZ) & \pi_0\Diff(\MTor) \ar[l]^{\hat{\hhom}}
}
\]
One easily checks that $\hhom(\afft{A}{b}) = A$ for all $A\in\GL(2,\bZ)$ and $b\in\bR^2$, which means that $p\circ \matrt\circ s$ is the inverse of $\hat{\hhom}$, and thus $\hat{\hhom}$ is an  isomorphism.
This also implies that the inclusion $\tAffSubgr \subset \Diff(\MTor)$ is a homotopy equivalence as well.
\end{proof}

\subsection{One lemma}
The following simple Lemma~\ref{lm:whe_cond} will be used several times.
For a topological group $\aGrp$ denote by $\aGrp_e$ the path component of $\aGrp$ containing the unit $e$ of $\aGrp$.
Then $\aGrp_e$ is a normal subgroup and the quotient $\aGrp/\aGrp_e$ is naturally identified with $\pi_0(\aGrp,e)$.
Note also that for $k\geq1$ we have natural isomorphisms $\pi_k\aGrp_e \cong \pi_k\aGrp$ induced by the inclusion $\aGrp_e\subset\aGrp$, where the homotopy groups are based at $e$.
\begin{sublemma}\label{lm:whe_cond}
Let $\tJincl:\aGrp\to\bGrp$ and $\tRestr:\bGrp\to\cGrp$ we two homomorphisms of topological groups and $\kGrp = \ker(\tRestr)$.
Suppose the following conditions hold:
\begin{enumerate}[leftmargin=*, label={\rm(W\arabic*)}]
\item\label{enum:whe:r}
$\tRestr:\bGrp\to\cGrp$ is a locally trivial fibration over its image $\tRestr(\bGrp)$ and the kernel $\kGrp$ (being the fibre of $\tRestr$) is weakly contractible, i.e.\ $\pi_i K = 0$ for all $i\geq0$;
\item\label{enum:whe:rj}
the composition $\tRestr\circ\tJincl:\aGrp_e\to C_e$ is a weak homotopy equivalence and the induced map $\tRestr_0\circ\tJincl_0:\pi_0\aGrp\to\pi_0\cGrp$ is injective;
\item\label{enum:whe:img_r_rj}
either \, $\tRestr_0(\pi_0\bGrp) \subset \tRestr_0\circ\tJincl_0(\pi_0\aGrp)$ \, or \, $\tJincl_0(\pi_0\aGrp) = \pi_0\bGrp$.
\end{enumerate}
Then $\tJincl:\aGrp\to\bGrp$ is a weak homotopy equivalence as well.
\end{sublemma}
\begin{proof}
Notice that we have the following commutative triangles:
\begin{equation}\label{equ:triangles}
\begin{aligned}
&\xymatrix{
    & \pi_0\bGrp \ar@{^(->}[dr]^-{\tRestr_0} \\
    \pi_0\aGrp \ar[ur]^-{\tJincl_0} \ar@{^(->}[rr]^-{(\tRestr\circ\tJincl)_0} && \pi_0\cGrp
}
&\qquad
\xymatrix{
    & \pi_k\bGrp_e \ar[dr]^-{\tRestr_k}_-{\cong} \\
    \pi_k\aGrp_e \ar[ur]^-{\tJincl_k} \ar[rr]^-{(\tRestr\circ\tJincl)_k}_-{\cong} && \pi_k\cGrp_e, & k\geq1.
}
\end{aligned}
\end{equation}
By~\ref{enum:whe:r}, $\tRestr$ is a locally trivial fibration with weakly contractible fiber.
Then from the exact sequence of $\tRestr$, we see that $\tRestr_0$ is injective and $\tRestr_k$ is an isomorphism for all $k\geq1$.
Moreover, by~\ref{enum:whe:rj}, $(\tRestr\circ\tJincl)_0$ is injective, and $(\tRestr\circ\tJincl)_k$ is an isomorphism for all $k\geq1$.

To prove that $\tJincl$ is a weak homotopy equivalence we need to check that $\tJincl_0$ is a bijection and $\tJincl_k$ is an isomorphism for all $k\geq1$.
But since $(\tRestr\circ\tJincl)_k$ and $\tRestr_k$ are isomorphisms for $k\geq1$, so is $\tJincl_k = (\tRestr_k)^{-1} \circ (\tRestr\circ\tJincl)_k$.
Moreover, as $(\tRestr\circ\tJincl)_0$ and $\tRestr_0$ are both injective, it follows that either of assumptions of~\ref{enum:whe:img_r_rj} implies that $\tJincl_0$ is a bijection.
\end{proof}

\subsection{Diffeomorphisms of $\ATor$}
Consider the following subgroup of $\GL(2,\bZ)$:
\begin{equation}\label{equ:group_A}
    \stMapClGr:=\left\{ \amatr{\eps}{0}{m}{\delta} \mid m\in\bZ, \, \eps,\delta\in\{\pm1\} \right\}.
\end{equation}

It is easy to see that $\stMapClGr$ is generated by the following matrices:
\begin{align}\label{equ:dlmt}
    \hDtwist  &= \amatr{1}{0}{1}{1},  &
    \hLambda  &= \amatr{-1}{0}{0}{1}, &
    \hMu      &= \amatr{1}{0}{0}{-1} = -\hLambda, &
    \hTau     &= \amatr{-1}{0}{0}{-1} = -E,
\end{align}
satisfying the following identities:
\begin{align*}
    &\hLambda^2 = \hMu^2 = E, &
    &\hTau = \hLambda\hMu = \hMu\hLambda, &
    &\hLambda\hDtwist\hLambda = \hMu\hDtwist\hMu = \hDtwist^{-1}, &
    &\hTau\hDtwist=\hDtwist\hTau.
\end{align*}
Evidently, $\hDtwist$ has infinite order, and the group generated by $\hDtwist,\hLambda$ has the following presentation: $\langle\hDtwist,\hLambda \mid \hLambda^2=1, \hLambda\hDtwist\hLambda = \hDtwist^{-1} \rangle$.
Hence it is isomorphic to the \myemph{infinite dihedral group} $\bD_{\infty}$ being in turn a unique non-trivial semidirect product $\bZ\rtimes\bZ_2$.
Moreover, $\tau$ commutes with that group, and therefore
$\stMapClGr = \langle \hDtwist, \hLambda \rangle \times  \langle \hTau\rangle \cong \bD_{\infty} \times \bZ_2$.

Note also that $\stMapClGr$ is exactly the subgroup of $\GL(2,\bZ)$ consisting of matrices for which the vector $\avect{0}{1}$ is eigen with eigen value $\pm 1$.

Denote by $\GAffSubgr:=\bR^2\rtimes\stMapClGr$ the subgroup of $\AffSubgr = \bR^2\rtimes\GL(2,\bZ)$ generated by translations $\bR^2$ and $\stMapClGr$.
Then its image $\stAffSubgr:=\matrt(\GAffSubgr) \subset \Diff(\MTor)$ is the corresponding semidirect product $\RotSub \rtimes \stMapClGr$.

It will be convenient to consider also the following maps $j^{1}, j^{-1}:\bC\to\bC$, $j^{1}(\az)=\az$, $j^{-1}(\az)=\bar{\az}$, so $j^{1}$ is the identity, while $j^{-1}$ is the complex conjugation map.
Evidently, $j^{-1}(\az)=\az^{-1}$ if $\nrm{\az}=1$, and $j^{\delta}\circ j^{\delta'} = j^{\delta\delta'}$ for $\delta,\delta'\in\{\pm1\}$.

Now let $c=(\alpha,\beta)\in\RotSub$ and $A=\amatr{\eps}{0}{m}{\delta}\in\stMapClGr$.
Then the diffeomorphism $\afft{A}{c}\in\stAffSubgr$ extends to a diffeomorphism $\affst{A}{c}:\ATor\to\ATor$ by the formula similar to~\eqref{equ:afft_Ab} but now $\az$ is allowed to run over $D^2$:
\begin{equation}\label{equ:hAb_solidTorus}
    \affst{A}{c}(\al,\az) = \bigl( \alpha \al^{\eps}, \, \beta \al^{m} j^{\delta}(\az) \bigr),
    \quad (\al,\az)\in\ATor=\Circle\times D^2,
\end{equation}
This gives an inclusion
\begin{equation}\label{equ:tJincl}
    \tJincl:\stAffSubgr \subset \Diff(\ATor),
    \qquad
    \tJincl(\afft{A}{c})=\affst{A}{c}.
\end{equation}
Let $\tRestr:\Diff(\ATor) \to \Diff(\MTor)$, $\tRestr(\dif)=\restr{\dif}{\MTor}$, be the ``restriction to $\MTor$ map''.
Evidently, $\tRestr$ is a homomorphism and $\ker(\tRestr) = \DiffFix{\ATor}{\MTor}$.
Then we have the following commutative diagram:
\[
\xymatrix@R=1em{
    & \Diff(\ATor) \ar[dr]^-{\tRestr} \\
    \stAffSubgr \ar[ur]^-{\tJincl} \ar[rr]^-{\subset} && \Diff(\MTor)
}
\]
so $\tRestr \circ \tJincl:\stAffSubgr \to \Diff(\MTor)$ is the identity inclusion described in Lemma~\ref{equ:incl:rotsubgr}, and $\tJincl$ can also be regarded as a section of $\tRestr$ defined on $\stAffSubgr$.
The following (another classical) statement describes the homotopy types of $\Diff(\ATor)$ and $\DiffFix{\ATor}{\MTor}$.
\begin{sublemma}\label{lm:DiffFixTdT_contr}
The group $\DiffFix{\ATor}{\MTor}$ is contractible and the inclusion $\tJincl:\RotSub \to \DiffId(\ATor)$ is a homotopy equivalence.
Moreover, the inclusion $\tJincl:\stAffSubgr \subset \Diff(\ATor)$ is also a homotopy equivalence, and in particular, $\pi_0\Diff(\ATor) \cong \stMapClGr$.
\end{sublemma}
\begin{proof}[Sketch of proof]
It suffices to check that $\tJincl$ and $\tRestr$ satisfy assumptions of Lemma~\ref{lm:whe_cond}.

\ref{enum:whe:r}
It is shown in each of the papers~\cite{Cerf:BSMF:1961, Palais:CMH:1960, Lima:CMH:1964} that $\tRestr$ is a locally trivial fibration over its image.
Moreover, contractibility of its kernel $\DiffFix{\ATor}{\MTor}$ is proved in~\cite[Theorem~2]{Ivanov:LOMI:1976}, which is also the first statement of the lemma.

\ref{enum:whe:rj}
As mentioned above, $\tRestr\circ\tJincl\equiv \id_{\stAffSubgr}: \stAffSubgr \subset \Diff(\MTor)$ is the identity inclusion.
Hence, by Lemma~\ref{lm:hom_type_DiffT2}, the inclusion $\tRestr\circ\tJincl: \RotSub \subset \DiffId(\MTor)$ is a homotopy equivalence, and the homomorphism $(\tRestr\circ\tJincl)_0: \pi_0 \stAffSubgr = \stMapClGr \to \pi_0\Diff(\MTor)$ is injective.

\ref{enum:whe:img_r_rj}
Let us prove that $\tJincl_0:\pi_0 \stAffSubgr = \stMapClGr \to \pi_0\Diff(\ATor)$ is surjective.
Keeping generators of $H_1(\MTor,\bZ)$ as in Lemma~\ref{lm:hom_type_DiffT2}, choose a generator of $H_1(\ATor)=\bZ$ so that $[\Circle\times0] = 1$.
Then the inclusion $i:\MTor\subset\ATor$ induces the homomorphism $i_1:\bZ^2\to\bZ$, $i_1(\pu,\pv)=\pu$.

Let $\dif\in\Diff(\ATor)$ and $A = \amatr{k}{l}{m}{n} = \hhom\circ\tRestr(\dif)$ be the automorphism of $H_1(\MTor,\bZ^2)$ induced by $\tRestr(\dif)=\restr{\dif}{\MTor}$.
We will show that $\matrt(A)\in\stAffSubgr$, which will imply that $\tJincl_0(\matrt(A)) = [\dif]\in\pi_0\Diff(\ATor)$, and thus prove that $\tJincl_0$ is surjective.

Notice that the induced isomorphism of homologies $\dif_1\colon\bZ\cong H_1(\ATor,\bZ)\to H_1(\ATor,\bZ)\cong\bZ$ is a multiplication by some $\eps\in\{-1,1\}$.
Also, since $\dif$ preserves the boundary $\MTor=\dATor$, we have the following commutative diagram of homology homomorphisms:
\begin{equation}\label{equ:induced_iso_solid_torus}
\begin{gathered}
\xymatrix@R=1.8em{
    H_1(\MTor,\bZ) \ar@{^(->}[d]_-{i_1} \ar[rr]^-{A}       & & H_1(\MTor,\bZ) \ar@{^(->}[d]^-{i_1} \\
    H_1(\ATor,\bZ)                      \ar[rr]^-{\dif_1}  & & H_1(\ATor,\bZ)
}
\end{gathered}
\end{equation}
which means that $(\eps,0) = (1,0)\amatr{k}{l}{m}{n}$.
Hence $k=\eps$, $l=0$, and since $\det(A)=\pm1$, we should have that $n = \pm 1$.
In other words, $A\in\stMapClGr$.
\end{proof}

\begin{subremark}\rm
The proof that $\tJincl_0: \stMapClGr\equiv \pi_0\stAffSubgr \to \pi_0\Diff(\ATor)$ is a bijection can be found in~\cite[Theorem~14]{Wajnryb:FM:1998}, where it is done by means of isotopies of meridian disks.
\end{subremark}

Notice that every diffeomorphism~\eqref{equ:hAb_solidTorus} preserves the leaves of $\AFoliation$, whence $\tJincl(\stAffSubgr) \subset \Diff(\AFoliation)$.

\subsection{Proof of Theorem~\ref{th:DFdT_full_variant}}
\label{sect:proof:th:DFdT_full_variant}
We need to show that each of the following inclusions
\begin{align*}
&\stAffSubgr  \ \xmonoArrow{\tJincl} \ \Diff(\AFoliation)   \ \subset \ \Diff(\ATor), &
&\{\id_{\ATor}\} \ \subset \ \DiffFix{\AFoliation}{\MTor}  \ \subset \ \DiffFix{\ATor}{\MTor}
\end{align*}
is a weak homotopy equivalence.
In particular, every path component of $\Diff(\AFoliation)$ is weakly homotopy equivalent to $2$-torus $\RotSub$, $\pi_0\Diff(\AFoliation) \cong \stMapClGr$, while $\pi_k\DiffFix{\AFoliation}{\MTor}=0$ for all $k\geq0$.

We need only to prove that the map $\tJincl: \stAffSubgr\to\Diff(\AFoliation)$, see~\eqref{equ:tJincl}, is a weak homotopy equivalence.
All other statements are established in Theorem~\ref{th:DFdT_contr} and Lemma~\ref{lm:hom_type_DiffT2}.

For formal reason, it will be convenient to distinguish the map of $\stAffSubgr$ into $\Diff(\AFoliation)$ from the map $\tJincl:\stAffSubgr \to \Diff(\ATor)$ considered in Lemma~\ref{lm:hom_type_DiffT2}.
Therefore we define the inclusion map $i:\Diff(\AFoliation) \subset \Diff(\ATor)$, and this gives the map $\stJincl:\stAffSubgr\to\Diff(\AFoliation)$ such that $\tJincl = i \circ \stJincl:\stAffSubgr\to\Diff(\ATor)$.
Our aim is to show that in fact $\stJincl$ is a weak homotopy equivalence.

Let $\stRestr:\Diff(\AFoliation)\to\Diff(\MTor)$, $\stRestr(\dif) = \restr{\dif}{\MTor}$, be the restriction map to the boundary $\MTor$ of $\ATor$.
Then $\stRestr = \tRestr \circ i$, where $\tRestr:\Diff(\ATor)\to\Diff(\MTor)$ is the same restriction map defined on all of $\Diff(\ATor)$ as in Lemma~\ref{lm:DiffFixTdT_contr}.
We will check that $\stJincl$ and $\stRestr$ satisfy conditions of Lemma~\ref{lm:whe_cond}, which will imply that $\stJincl$ is a weak homotopy equivalence.

Note that on the level of $\pi_0$-groups we have the following sequence of homomorphisms:
\[
\xymatrix{
\pi_0\stAffSubgr           \ar[r]_-{\stJincl_0} \ar@/^15px/[rr]^-{\tJincl_0} &
\pi_0\Diff(\AFoliation) \ar[r]^-{i_0} \ar@/_10px/[rr]_-{\stRestr_0} &
\pi_0\Diff(\ATor)       \ar[r]^-{\tRestr_0} &
\pi_0\Diff(\MTor),
}
\]

\ref{enum:whe:r}
By Lemma~\ref{lm:triv_nbh:loc_triv_restr}, $\stRestr$ is a locally trivial fibration over its image.
Moreover, its kernel $\DiffFix{\ATor}{\MTor}$ is weakly contractible due to Theorem~\ref{th:DFdT_contr}.

\ref{enum:whe:rj}
As mentioned in Lemma~\ref{lm:DiffFixTdT_contr}, $\stRestr\circ\stJincl\equiv \id_{\stAffSubgr}: \stAffSubgr \subset \Diff(\MTor)$ is the identity inclusion.
Therefore, by Lemma~\ref{lm:hom_type_DiffT2}, the inclusion $\stRestr\circ\stJincl: \RotSub \subset \DiffId(\MTor)$ of the corresponding identity path components is a homotopy equivalence, and the homomorphism $(\stRestr\circ\stJincl)_0: \pi_0 \stAffSubgr = \stMapClGr \to \pi_0\Diff(\MTor)$ is injective.

\ref{enum:whe:img_r_rj}
Again, by Lemma~\ref{lm:DiffFixTdT_contr}, $\tJincl_0 = i_0\circ\stJincl_0: \pi_0\stAffSubgr \to \pi_0\Diff(\ATor)$ is a bijection.
Hence
\[
\stRestr_0(\Diff(\AFoliation)) \subset \tRestr_0(\pi_0\Diff(\ATor)) = \tRestr_0\circ\tJincl_0(\pi_0\Diff(\ATor)) =
\stRestr_0\circ\stJincl_0(\pi_0\Diff(\ATor)),
\]
i.e.\ the image of $(\stRestr\circ\stJincl)_0$ is contained in the image of $\stRestr_0$.
\qed

\section{Diffeomorphisms of lens spaces}\label{sect:diff_lens_spaces}
\subsection{Polar Morse-Bott foliations on lens spaces}
Consider the product $\TLman = \ATor\times\{0,1\}$, which can be regarded as a disjoint union of two copies of $\ATor$.
Let also $\xi:\MTor\to\MTor$ be a diffeomorphism, and $\sim$ be the equivalence relation on $\TLman$ such that $(\px,0)\sim(\xi(\px),1)$ for $\px\in\MTor$, and for any other point $(\py,i) \in \TLman$ with $\py\not\in\MTor$ its equivalence class consists of that point only.
Then the quotient space $\Lman_{\xi}:=\TLman/\sim$ is called a \myemph{lens space}.
Thus, $\Lman$ is obtained from two copies of the solid torus $\ATor$ by some diffeomorphism between their boundaries.

Let also $\psi\colon \TLman\to\Lman_{\xi}$ be the corresponding quotient map.
It will be convenient to denote the images of $\ATor\times 0$, $\ATor\times 1$, and $\MTor$ in $\Lpq{p}{q}$ under $\psi$ by $\ATor$, $\BTor$, and $\MTor$ respectiely.

It is well known, that if $\hat{\xi}\colon \MTor\to\MTor$ is a diffeomorphism isotopic to $\xi$, then $\Lman_{\xi}$ and $\Lman_{\hat{\xi}}$ are diffeomorphic.
Therefore, by Lemma~\ref{lm:hom_type_DiffT2}, one can assume that $\xi  = \matrt(A) \in \tAffSubgr \subset \Diff(\MTor)$ and is given by formula~\eqref{equ:afft_Ab} for some $A = \amatr{r}{p}{s}{q}\in\GL(2,\bZ)$, i.e.\ $\xi(\al,\az) = (\al^{r}\az^{p},\al^{s}\az^{q})$.
Moreover, one can further replace $\TLman$ with its diffeomorphic copy, which, due to Lemma~\ref{lm:DiffFixTdT_contr}, allows to replace $A$ with $PAQ$ for any matrices $P,Q\in\stMapClGr$, see~\eqref{equ:group_A}.
Also, one can regard $\sim$ as ``gluing the tori from outside'', hence it is natural to assume that $\det(A) = rq - sp = -1$, so $\xi$ reverses orientation.
In this case $\Lman_{\xi}$ is denoted by $\Lpq{p}{q}$.

All the above simplification of $A$ easily imply that for each $p =0,1,2$, there is a unique lens space (up to a diffeomorphism).
Namely,
\begin{align*}
\bullet~\Lpq{0}{1} &= \Circle\times S^2, & A &= \amatr{-1}{0}{0}{1} = \hLambda, & \xi(\al,\az) &= (\bar{\al},\az); \\
\bullet~\Lpq{1}{0} &= S^3,               & A &= \amatr{0}{1}{1}{0},  & \xi(\al,\az) &= (\az, \al); \\
\bullet~\Lpq{2}{1} &= \bR{P}^3,          & A &= \amatr{1}{2}{1}{1},  & \xi(\al,\az) &= (\al\az^2, \al\az).
\end{align*}
Moreover, for $p> 2$ we can assume that $0 < q < p$, $(p,q)=1$, and can also replace $q$ with $q^{-1}(\bmod\,p)$.
Lens spaces are well studied, e.g.~\cite{Reidemeister:AMSUH:1935, Brody:AM:1960, Gadgil:GT:2001}, and their mapping class groups $\pi_0\Diff(\Lpq{p}{q})$ are described in~\cite{Hatcher:Top:1976, Ivanov:DANSSSR:1979, Ivanov:DANSSSR:Corrections:1979, Rubinstein:TrAMS:1979, Hatcher:ProcAMS:1981, Bonahon:Top:1983}, see also~\cite{Ivanov:LOMI:1982}. 

Notice that the foliation $\AFoliation$ on each $\ATor\times i$ gives a foliation on $\Lpq{p}{q}$ into parallel $2$-tori and two singular circles.
We will denote that foliation by $\FolLpq{p}{q}$ and call \myemph{polar}.

\begin{subremark}\rm
Notice that $\FolLpq{p}{q}$ coincides with the partition of $\Lpq{p}{q}$ into the level sets of the following Morse-Bott function
\[
    \func\colon \Lpq{p}{q}\to\bR,
    \qquad
    \func(\al,\az) =
    \begin{cases}
        \nrm{\az}^2,   & (\al,\az)\in\ATor\times 0,\\
        2-\nrm{\az}^2, & (\al,\az)\in\ATor\times 1,
    \end{cases}
\]
whose critical submanifolds are central circles of $(\Circle\times 0) \times i$ of $\ATor\times i$, $i=0,1$.
\end{subremark}

Let $\Diff(\FolLpq{p}{q})$ be the group of $\FolLpq{p}{q}$-leaf preserving diffeomorphisms of $\Lpq{p}{q}$ endowed with the corresponding $\Cinfty$ Whitney topologies.
Let also $\dif\in\Diff(\FolLpq{p}{q})$.
Since $\ATor$ and $\BTor$ are the unions of leaves of $\FolLpq{p}{q}$, they are invariant under $\dif$, whence $\dif$ lifts to a unique diffeomorphism $\widehat{\dif}:\TLman\to\TLman$ such that for all $\px\in\MTor$ and $i=0,1$:
\begin{align*}
    &\psi\circ\widehat{\dif} = \dif\circ\psi, &
    &\widehat{\dif}(\ATor\times i)=\ATor\times i, &
    &\dif\circ\xi(\px,0) = \xi\circ\dif(\px,0).
\end{align*}
Let $Q \in \GL(2,\bZ)$ be the matrix of the automorphism of $H_1(\MTor,\bZ)$ induced by $\restr{\widehat{\dif}}{\MTor\times0}$.
Since it extends to the diffeomorphism $\restr{\widehat{\dif}}{\ATor\times0}$ of the solid torus $\ATor\times0$, we have by Lemma~\ref{lm:DiffFixTdT_contr} that $\eta(Q)\in\stAffSubgr$.
On the other hand, the restriction $\restr{\widehat{\dif}}{\MTor\times1} = \restr{\xi}{\MTor\times1}\, \circ \, \restr{\widehat{\dif}}{\MTor\times0} \, \circ \, \restr{\xi^{-1}}{\MTor\times0}$
also extends to a diffeomorphism $\restr{\widehat{\dif}}{\ATor\times 1}$ of the solid torus $\ATor\times1$, so $\eta(AQA^{-1}) = \xi\eta(A)\xi^{-1}\in\stAffSubgr$.

Consider the following subgroup of $\stAffSubgr$:
\begin{equation}\label{equ:MapClGrp_Lpq}
    \lpqAffSubgr_{p,q} \, := \, \stAffSubgr \,\cap\, \xi^{-1}\stAffSubgr\xi.
\end{equation}
Then $\xi\dif\xi^{-1}\in\stAffSubgr$ for every $\dif\in \lpqAffSubgr_{p,q}$, whence we get a diffeomorphism $\widehat{\phi}_{\dif}:\TLman\to\TLman$ given by $\restr{\widehat{\phi}_{\dif}}{\ATor\times0} = \tJincl(\dif)$ and $\restr{\widehat{\phi}_{\dif}}{\ATor\times1} = \tJincl(\xi\dif\xi^{-1})$.
Moreover, $\restr{\widehat{\phi}_{\dif}}{\MTor\times0}$ and $\restr{\widehat{\phi}_{\dif}}{\MTor\times1}$ agree with the projection $\psi$, and therefore $\widehat{\phi}_{\dif}$ induces a well-defined diffeomorphism $\phi_{\dif}:\Lpq{p}{q} \to \Lpq{p}{q}$.
As $\tJincl(\dif),\tJincl(\xi\dif\xi^{-1})\in\Diff(\BFoliation)$, we obtain that $\phi_{\dif} \in\Diff(\FolLpq{p}{q})$.

Our final result is the following theorem containing Theorem~\ref{th:Lpq_hom_type}:
\begin{subtheorem}\label{th:Lpq_full_variant}
The inclusion $\lpqJincl: \lpqAffSubgr_{p,q} \subset \Diff(\FolLpq{p}{q})$, $\lpqJincl(\dif):= \phi_{\dif}$, is a weak homotopy equivalence.
Moreover,
\[
    \lpqAffSubgr_{p,q} =
    \begin{cases}
        \stAffSubgr = \RotSub\rtimes\stMapClGr,
                & \text{for $\Lpq{0}{1} = \Circle\times S^2$}, \\
        \langle \RotSub, \hLambda, \hMu \rangle \cong \Ort(2)\times\Ort(2),
                & \text{for $\Lpq{1}{0} = S^3$}, \\
        \langle \RotSub, \hLambda\hDtwist, \tau \rangle \cong (\RotSub \rtimes \langle\hLambda\hDtwist\rangle) \times \bZ_2,
                & \text{for $\Lpq{2}{1} = \bR{P}^3$}, \\
        \langle \RotSub, \tau \rangle \cong \RotSub \times \bZ_2,
                & \text{for $p>2$}.
    \end{cases}
\]
In particular, the identity path component of $\Diff(\FolLpq{p}{q})$ is weakly homotopy equivalent to $2$-torus $\RotSub$, and
\begin{align*}
 \pi_0 \Diff(\FolLpq{p}{q}) \cong
 \begin{cases}
    \langle \hDtwist, \hLambda, \hMu \rangle = \stMapClGr \cong \bD_{\infty} \times \bZ_2,  & \text{for } \Lpq{0}{1} = \Circle\times S^2, \\
    \langle \hLambda, \hMu \rangle \cong \bZ_2 \oplus \bZ_2,                             & \text{for } \Lpq{1}{0} = S^3, \\
    \langle \hLambda\hDtwist, \tau\rangle \cong \bZ_2 \oplus \bZ_2,                      & \text{for } \Lpq{2}{1} = \bR{P}^3, \\
    \langle \tau \rangle \cong \bZ_2,                                                    & \text{for $p>2$}.
 \end{cases}
\end{align*}
If $p>2$ and $q^2\not=\pm1 (\bmod\ p)$, then the inclusions $\lpqAffSubgr_{p,q} \subset \Diff(\FolLpq{p}{q}) \subset \Diff(\Lpq{p}{q})$ are weak homotopy equivalences.
\end{subtheorem}
\begin{proof}
A)
First we need to prove that \myemph{the inclusion $\lpqJincl: \lpqAffSubgr_{p,q} \subset \Diff(\FolLpq{p}{q})$ is a weak homotopy equivalence}.
Notice that $\MTor$ is a leaf of $\FolLpq{p}{q}$, whence it is invariant under each $\dif\in\Diff(\FolLpq{p}{q})$.
Therefore, we have the restriction map $\lpqRestr:\Diff(\FolLpq{p}{q})\to\Diff(\MTor)$, $\lpqRestr(\dif)=\restr{\dif}{\MTor}$, being a homomorphism with kernel is $\DiffGdTA$.
It suffices to check that $\lpqJincl$ and $\lpqRestr$ satisfy assumptions of Lemma~\ref{lm:whe_cond}.

\ref{enum:whe:r}
By Lemma~\ref{lm:triv_nbh:loc_triv_restr}, $\rmap$ is a locally trivial principal $\DiffGdTA$-fibration over its image.

Let us prove that its kernel \myemph{$\DiffGdTA$ is weakly contractible}.
Let $\mathbf{C}$ and $\mathbf{C}'$ be the collars $\RBd{\bConst}$ of $\MTor$ in $\ATor$ and $\BTor$ respectively, so their union $\RNb{\bConst} = \mathbf{C} \cup \mathbf{C}'$ is a $\FolLpq{p}{q}$-saturated neighborhood of $\MTor$.
Let $\DiffGdN$ be the subgroup of $\DiffGdTA$ consisting of diffeomorphisms fixed on $\RNb{\bConst}$.
Then, as in Lemma~\ref{lm:collar}, one can show that $\DiffGdN$ coincides with the group $\DiffFix[\vbp,\RNb{\bConst}]{\FolLpq{p}{q}}{\MTor}$.
Then by Linearization Theorem~\ref{th:linearization_simpler} the inclusion $\DiffGdN = \DiffFix[\vbp,\RNb{\bConst}]{\FolLpq{p}{q}}{\MTor} \subset \DiffGdTA$ is a homotopy equivalence.

Note also that we have the following natural isomorphism of topological groups
\[
    \DiffGdN \cong \DiffFix{\AFoliation}{\mathbf{C}} \times \DiffFix{\BFoliation}{\mathbf{C}'},
    \qquad
    \dif\mapsto (\restr{\dif}{\ATor}, \restr{\dif}{\BTor}),
\]
and, by Theorem~\ref{th:DFdT_contr}, both multiples $\DiffFix{\AFoliation}{\mathbf{C}}$ and $\DiffFix{\BFoliation}{\mathbf{C}'}$ are weakly contractible.
Hence so is $\DiffGdN$ and therefore $\DiffGdTA$.

\ref{enum:whe:rj}
Again, $\lpqRestr\circ\lpqJincl\equiv \id_{\stAffSubgr}: \stAffSubgr \subset \Diff(\MTor)$ is the identity inclusion.
Therefore, by Lemma~\ref{lm:hom_type_DiffT2}, the inclusion of the identity path components $\stRestr\circ\stJincl: \RotSub \subset \DiffId(\MTor)$ is a homotopy equivalence, and the homomorphism $(\stRestr\circ\stJincl)_0: \pi_0 \stAffSubgr = \stMapClGr \to \pi_0\Diff(\MTor)$ is injective.

\ref{enum:whe:img_r_rj}
It remains to check that the homomorphism $\lpqJincl_0:\pi_0\lpqAffSubgr_{p,q}\to \pi_0\Diff(\FolLpq{p}{q})$ is surjective.

Since $\lpqRestr$ is locally trivial fibration with weakly contractible fiber, the homomorphism
\[ \lpqRestr_0:\pi_0\Diff(\FolLpq{p}{q})\to\pi_0\Diff(\MTor) = \GL(2,\bZ) = \Aut(H_1(\MTor,\bZ)) \]
is injective, see left triangle in~\eqref{equ:triangles}.
In other words, the isotopy class of $\dif$ in $\pi_0\Diff(\FolLpq{p}{q})$ is determined the automorphism $\gdif_1$ induced by $\gdif$ of the first homologies of $\MTor$.
But as noted in before~\eqref{equ:MapClGrp_Lpq}, $\gdif_1 \in \stAffSubgr \cap \xi^{-1}\stAffSubgr\xi = \lpqAffSubgr_{p,q}$, whence $\lpqJincl_0$ is surjective.

Thus conditions of Lemma~\ref{lm:whe_cond} hold, and $\lpqJincl$ is a weak homotopy equivalence.

B)
Let us \myemph{compute the group $\lpqAffSubgr_{p,q} = \stAffSubgr \cap \xi^{-1}\stAffSubgr\xi$}.

Since $\RotSub$ is a normal subgroup of $\stMapClGr$, and $\xi\in\stMapClGr$, it follows that $\RotSub \subset \lpqAffSubgr_{p,q}$.
Hence $\lpqAffSubgr_{p,q}$ is generated by $\RotSub$ and the image under $\matrt$ of the intersection $\stMapClGr \cap A^{-1}\stMapClGr A$.
Thus we need to compute the latter intersection.
Note also that $\hTau=-E$ commutes with any matrix $A$, and therefore we need only to check the matrices $\hDtwist$, $\hLambda$ and $\hMu$.

1) Let $p=0$, so $\Lpq{0}{1} = \Circle\times S^2$, and $A = \amatr{-1}{0}{0}{1}=\hLambda \in \stMapClGr$.
Then $\stMapClGr = \xi^{-1}\stMapClGr\xi$, and $\lpqAffSubgr_{p,q} = \stAffSubgr$.

2) Assume that $p\geq1$.
First let us make one observation which will simplify computations.
Let $P\in\amatr{\eps}{0}{m}{\delta} \in \stMapClGr$.
Then the vector $\px = \avect{0}{1}$ is an eigen vector of $P$ with eigen value $\delta = \pm 1$.
Conversely, $\stMapClGr$ is exactly the subgroup of $\GL(2,\bZ)$ consisting of matrices for which $\px$ is an eigen vector with eigen value $\pm 1$.

Hence if $Q=A^{-1}PA \in \amatr{\eps'}{0}{m'}{\delta'} \in \stMapClGr$, then $\px$ is still an eigen vector for $Q$, i.e.\
$\delta'\px = Q\px = A^{-1}PA \px$.
Therefore, $\delta'A\px = P A\px$, that is $A\px = \amatr{r}{p}{s}{q}\avect{0}{1} = \avect{p}{q}$ is an eigen vector for $P$ with eigen value $\delta'$.
This effect is, of course, a consequence of the geometric situation that the diffeomorphism $\eta(Q)$ should preserve isotopy classes of meridians of both $\ATor$ and $\BTor$.

Thus, $\avect{\delta' p}{\delta' q} = \amatr{\eps}{0}{m}{\delta} \avect{p}{q} = \avect{\eps p}{mp+\delta q}$.
Since $p\not=0$, we obtain that $\delta' = \eps$ and
\begin{align}\label{equ:identity_lpq}
 q(\eps-\delta) = mp.
\end{align}
Consider several cases.

2.1)
Let $(p,q)=(1,0)$, so $\Lpq{1}{0} = S^3$, and $A = \amatr{0}{1}{1}{0}$.
Then~\eqref{equ:identity_lpq} implies that $m=0$.
Hence $Q=\amatr{\eps}{0}{0}{\delta}$, so $Q$ can be either of the following matrices: $E$, $\hTau=-E$, $\hLambda$, $\hMu$.
One easily checks that $A^{-1}\hLambda A^{-1} = \hMu \in \stMapClGr$ and $A^{-1}\hMu A^{-1} = \hLambda$.
Hence $\lpqAffSubgr_{p,q} = \langle \RotSub, \hLambda, \hMu \rangle$.

2.2)
Let $(p,q)=(2,1)$, so $\Lpq{2}{1} = \bR{P}^3$, and $A = \amatr{1}{2}{1}{1}$.
Then~\eqref{equ:identity_lpq} implies that $\eps-\delta = 2m$.
Hence $Q$ can be either of the following matrices:
\begin{align*}
    E &= \amatr{1}{0}{0}{1},&
    &\hTau = -E = \amatr{-1}{0}{0}{-1}.
    &\hLambda\hDtwist = \amatr{1}{0}{-1}{-1},   &
    &\hTau\hLambda\hDtwist = -\hLambda\hDtwist = \amatr{-1}{0}{1}{1},
\end{align*}
which implies that $\lpqAffSubgr_{p,q} = \langle \RotSub, \hLambda\hDtwist, \hTau \rangle$.

2.3)
Suppose $p>2$.
Then $q\not=0$ as well and is relatively prime with $p$.
Hence there exists $a\not=0$ such that $m = aq$ and $\eps-\delta=ap$.
If $a=0$, then $m=0$ and $\eps=\delta$ so $Q=\pm E$ and it commutes with $A$.
On the other hand, if $a\not=0$, then $\eps-\delta \in \{-2,0,2\}$ implies that $\nrm{ap}=\nrm{\eps-\delta}=2$ which contradicts to the assumption that $p>2$.
Hence in this case $\lpqAffSubgr_{p,q} = \langle \RotSub, \hTau \rangle$.

C)
Suppose $p>2$ and $q^2\not=1 (\bmod\ p)$.
In this case, due to~\cite[Theorem~2.3]{KalliongisMiller:KMJ:2002}, the group $\lpqAffSubgr_{p,q}$ coincides with the isometry group $\Isom(\Lpq{p}{q})$ of the natural elliptic metric on $\Lpq{p}{q}$, see below.
Now by the Smale conjecture for $\Lpq{p}{q}$, the inclusion $\Isom(\Lpq{p}{q}) \subset \Diff(\Lpq{p}{q})$ is a homotopy equivalence, \cite{HongKalliongisMcCulloughRubinstein:LMN:2012}.
Hence the inclusions $\lpqAffSubgr_{p,q} \subset \Diff(\FolLpq{p}{q}) \subset \Diff(\Lpq{p}{q})$ are weak homotopy equivalences.
Theorem~\ref{th:Lpq_full_variant} is completed.
\end{proof}

\subsection{Concluding remarks}
Let us briefly recall the known description of the homotopy types of $\Diff(\Lpq{p}{q})$ and compare them with the homotopy types of $\Diff(\FolLpq{p}{q})$ given in Theorem~\ref{th:Lpq_full_variant}.

\begin{enumerate}[wide, label={\arabic*)}, itemsep=1ex]
\item
For $p=0$ we have the subgroup $\mathcal{Q}=\bigl(\Omega(\Ort(3)) \rtimes \Ort(3) \bigr) \rtimes \Ort(2) \subset \Diff(\Circle\times S^2)$ being a certain iterated semidirect product and consisting of diffeomorphisms $\gdif_{\alpha,A}:\Circle\times S^2\to\Circle\times S^2$ of the form
\[
    \gdif_{\alpha,A}(\al,\pu) = (\alpha \al, A(\al) \pu),  \qquad (\al,\pu)\in\Circle\times S^2,
\]
where $\alpha\in\Circle$, and $A:\Circle\to\Ort(3)$ is a continuous map.
Topologically $\mathcal{Q}$ is a topological product $\Omega(\Ort(3)) \times \Ort(3) \times \Ort(2)$ and A.~Hatcher proved in~\cite{Hatcher:ProcAMS:1981} that the above inclusion $\mathcal{Q} \subset \Diff(\Circle\times S^2)$ is a homotopy equivalence.

Then the inclusion $\tJincl: \lpqAffSubgr_{0,1} = \RotSub\rtimes\stMapClGr \to \mathcal{Q}$ is given by the following formula.
Let $\al\in\Circle$ and $\pu=(\az,\pv)\in S^2 \subset \bR^3 = \bC\times\bR$ be the point on $S^2$.
Then for $b=(\alpha,\beta)\in\RotSub$ and $A=\amatr{\eps}{0}{m}{\delta}\in\stMapClGr$
\begin{equation}\label{equ:hAb_lpq}
    \tJincl(\affst{A}{b})(\al,\az,\pv) = \bigl( \alpha \al^{\eps}, \, \beta \al^{m}\az^{\delta}, \pv \bigr).
\end{equation}

\item
For $p=1$ we have the inclusion $\Ort(4) \subset \Diff(S^3)$ which is also proved by A.~Hatcher to be a homotopy equivalence, see~\cite{Hatcher:AnnM:1983}.
In this case we have a canonical embedding $\lpqAffSubgr_{1,0} = \Ort(2) \times \Ort(2) \subset \Ort(4)$ independently acting on the first and second coordinate.

\item
Suppose $p\geq2$.
Let $\hTau\in\Diff(\Lpq{p}{q})$ be given by $\hTau(\al,\az)=(\bar{\al}, \bar{\az})$ on each $\ATor\times\{i\}$, $i=0,1$, so formally it is the image $\lpqJincl(\matrt(\hTau))$ of the matrix $\hTau$, however to simplify notation we denote it here by the same letter.
If $q^2\equiv 1(\bmod\,p)$, then one can define a diffeomorphism $\sigma_{+}\in\Diff(\Lpq{p}{q})$ exchanging $\ATor$ and $\BTor$ and defined by
$\xymatrix{
    \sigma_{+}: \ATor \ar@{<->}[rr]^-{ (\al,\az) \mapsto  (\al,\az) }_-{ (\al,\az) \leftmaspto  (\al,\az) }  && \BTor.
}$
Also, if $q^2\equiv -1(\bmod\,p)$, then there exists another diffeomorphism $\sigma_{-}\in\Diff(\Lpq{p}{q})$ also exchanging $\ATor$ and $\BTor$ and given by
$\xymatrix{
    \sigma_{-}: \ATor \ar@{<->}[rr]^-{ (\al,\az) \mapsto  (\bar{\al},\az) }_-{ (\al,\bar{\az}) \leftmaspto  (\al,\az) }  && \BTor.
}$
It easily follows that $\sigma_{+}$ preserves orientation and $\sigma_{+}^2 = \id_{\Lpq{p}{q}}$, while $\sigma_{-}$ reverses orientation and $\sigma_{-}^2 = \hTau$.

F.~Bonahon~\cite[Theorem~1]{Bonahon:Top:1983} proved that every diffeomorphism $\dif$ of $\Lpq{p}{q}$ is isotopic to a diffeomorphism which leaves the common boundary $\MTor$ of $\ATor$ and $\BTor$ invariant, and moreover, $\pi_0\Diff(\Lpq{p}{q})$ is generated by the isotopy classes of $\hTau$, $\sigma_{+}$, and $\sigma_{-}$, (more precisely by those $\sigma_{\pm}$ which are defined for the given $(p,q)$).

\item
Notice that for $p\geq2$ the universal cover of $\Lpq{p}{q}$ is $S^3$, whence $\Lpq{p}{q}$ admits a Riemannian metric locally isomorphic with the standard metric on $S^3$, and for that reason such $\Lpq{p}{q}$ are also called \myemph{elliptic} $3$-manifolds, see~\cite[Theorem~2.3]{KalliongisMiller:KMJ:2002} for the description of the isometry groups $\mathrm{Isom}(\Lpq{p}{q})$.
In the book~\cite{HongKalliongisMcCulloughRubinstein:LMN:2012} it was proved that for all $\Lpq{p}{q}$ with $p>2$ (i.e.\ except for $\Lpq{2}{1}=\bR{P}^3$) the inclusion of the isometry group $\mathrm{Isom}(\Lpq{p}{q}) \subset \Diff(\Lpq{p}{q})$ is a homotopy equivalence.
Such a statement is called \myemph{Smale conjecture}, and as mentioned above it was proved by A.~Hatcher for $S^3$.

Moreover, in a recent unpublished yet preprint by R.~Bamler and B.~Kleiner~\cite{BalmerKleiner:RF:2019} the Smale conjecture is proved for all elliptic $3$-manifolds including $S^3$ and $\bR{P}^3$.
In particular, it implies that $\pi_0\Diff(\bR{P}^3)$ is homotopy equivalent to $\mathrm{Isom}(\bR{P}^3) = \SO(4)/\{\pm E\} \cong \bR{P}^3\times\bR{P}^3$.

Finally, note that in~\cite{KalliongisMiller:KMJ:2002} there were studied groups of diffeomorphisms of $\Lpq{p}{q}$, $(p>2)$, fixed on $\MTor$, while in~\cite{KalliongisMiller:TA:2003} there were considered diffeomorphisms leaving invariant $\MTor$.
\end{enumerate}

\subsection*{Acknowledgments}
The authors are grateful to Andy Putman for the information on MathOverflow (\cite{Putman:MO:Answer}) about the homotopy type of $\Diff(\bR{P}^3)$ and reference to the paper~\cite{BalmerKleiner:RF:2019}, to Ryan Budney and Neil Strickland for comments, and to Ivan Smith for reference to the paper~\cite{Fukui:JJM:1976} by K.~Fukui.
The authors also thanks the anonymous Referee for careful reading the paper, useful comments, and suggestions.

\def\cprime{$'$} \def\cprime{$'$} \def\cprime{$'$} \def\cprime{$'$}
  \def\cprime{$'$} \def\cprime{$'$} \def\cprime{$'$} \def\cprime{$'$}
  \def\cprime{$'$} \def\cprime{$'$} \def\cprime{$'$} \def\cprime{$'$}
  \def\cprime{$'$} \def\cprime{$'$}

\end{document}